\documentclass[reqno, 11pt]{amsart}

\usepackage{tikz}
\usepackage{mathrsfs}
\usepackage{amsthm}
\usepackage{amsmath,amssymb}
\usepackage[all,poly,knot]{xy}
\usepackage{hyperref}
\usepackage{amsfonts}
\usepackage{color} 
\usepackage{diagbox}
\usepackage{chemfig} 
\usepackage{tikz}
\usepackage[square,sort,comma,numbers]{natbib}

\newtheorem{thm}{Theorem}[section]
\newtheorem{cor}[thm]{{Corollary}}
\newtheorem{cor-defi}[thm]{{Corollary-Definition}}
\newtheorem{lem}[thm]{{Lemma}}
\newtheorem{prop}[thm]{Proposition}

\newtheorem{conj}[thm]{{Conjecture}}

\newtheorem{defi}[thm]{Definition}
\theoremstyle{remark}
\newtheorem*{rmk}{Remark}

\def\C{\mathbb C}

\def\R{\mathbb R}
\def\Q{\mathbb Q}
\def\Z{\mathbb Z}

\def\N{\mathbb N}

\def\D{\mathcal D}  
\def\U{\mathcal U}
\def\X{\mathcal X}

\def\cech{\check{\mathrm{C}}\mathrm{ech}}
\def\P{\mathbb{P}}
\def\F{\mathbb{F}}

\newcommand{\bD}{{\mathbb D}}
\newcommand{\bE}{{\mathbb E}}
\newcommand{\bF}{{\mathbb F}}

\newcommand{\bH}{{\mathbb H}}

\newcommand{\bP}{{\mathbb P}}
\newcommand{\bQ}{{\mathbb Q}}

\newcommand{\bZ}{{\mathbb Z}}

\newcommand{\mB}{{\mathcal B}}

\newcommand{\mD}{{\mathcal D}}
\newcommand{\mE}{{\mathcal E}}
\newcommand{\mF}{{\mathcal F}}

\newcommand{\mL}{{\mathcal L}}

\newcommand{\mO}{{\mathcal O}}

\newcommand{\mT}{{\mathcal T}}
\newcommand{\mU}{{\mathcal U}}
\newcommand{\mV}{{\mathcal V}}

\newcommand{\mX}{{\mathcal X}}
\newcommand{\mY}{{\mathcal Y}}
\newcommand{\mZ}{{\mathcal Z}}

\newcommand{\sO}{{\mathscr O}}

\newcommand{\sU}{{\mathscr U}}

\newcommand{\sX}{{\mathscr X}}

\newcommand{\nc}{\newcommand}
\nc{\on}{\operatorname}

\nc{\Aut}  {{\on{\mathrm  {Aut}}}}
\nc{\End}  {{\on{\mathrm  {End}}}}
\nc{\Fil}  {{\on{\mathrm  {Fil}}}}
\nc{\Frac} {{\on{\mathrm  {Frac}}}}
\nc{\Gal}  {{\on{\mathrm  {Gal}}}}
\nc{\GL}   {{\on{\mathrm  {GL}}}}
\nc{\Gr}   {{\on{\mathrm  {Gr}}}}
\nc{\Hom}  {{\on{\mathrm  {Hom}}}}
\nc{\id}   {{\on{\mathrm  {id}}}}
\nc{\PGL}  {{\on{\mathrm  {PGL}}}}
\nc{\rank} {{\on{\mathrm  {rank}}}}
\nc{\rmd}  {{\on{\mathrm  {d}}}}
\nc{\Spec} {{\on{\mathrm  {Spec}}}}

\def\can {{\mathrm  {can}}}

\nc{\HDF}  {{\on{\mathcal  {HDF}}}}
\nc{\HIG}  {{\on{\mathcal  {HIG}}}}
\nc{\IC}   {{\on{\mathcal {IC}}}}
\nc{\MCF}  {{\on{\mathcal {MCF}}}}
\nc{\MCFa} {{\on{\mathcal {MCF}_{[0,a]}}}}
\nc{\MF}   {{\on{\mathcal {MF}}}}
\nc{\MFa}  {{\on{         \MF_{[0,a]}}}}
\nc{\MFaf} {{\on{         \MF_{[0,a],f}}}}
\nc{\MIC}  {{\on{\mathcal {MIC}}}} 
\nc{\MICa} {{\on{\mathcal {MIC}_{[0,a]}}}} 
\nc{\THDF} {{\on{\mathcal {THDF}}}}
\nc{\THDFa}{{\on{\mathcal {THDF}_{[0,a]}}}}
\nc{\TMF}  {{\on{\mathcal {TMF}}}}
\nc{\TMFa} {{\on{\TMF_{[0,a]}}}}
\nc{\TMFaf}{{\on{\TMF_{[0,a],f}}}}
\nc{\tMIC} {{\on{\widetilde{\mathcal{MIC}}}}} 

\def\wt{\check}
\def\wh{\hat}

\def\hR{{\widehat{R}}} 
\def\tmB{{\widetilde{\mB}}}
\def\mBR{{\mB_{R_\pi}}} 
\def\tmBR{{\widetilde{\mB}_{R_\pi}}}
\def\tnabla{{\widetilde{\nabla}}}
\def\tV{{\widetilde V}}
\def\tsX{{\widetilde{\sX}}}

\def\mEnd{\mE nd}

\def\R{R}
\numberwithin{equation}{section}
\begin{document}
	
	\title[Projective representations and Twisted Higgs-de Rham flows]{Projective Crystalline Representations of \'Etale Fundamental Groups and Twisted Periodic Higgs-de Rham Flow} 
	\author{Ruiran Sun}
	\author{Jinbang Yang}
	\author{Kang Zuo}
	
	\email{ruirasun@uni-mainz.de, yjb@mail.ustc.edu.cn, zuok@uni-mainz.de} 
	\address{Institut F\"ur Mathematic, Universit\"at Mainz, Mainz, 55099, Germany}
\thanks{This work was supported by SFB/Transregio 45 Periods, Moduli Spaces and Arithmetic of Algebraic
	Varieties of the DFG (Deutsche Forschungsgemeinschaft)  and also supported by National Key Basic Research Program of China (Grant No. 2013CB834202).}  
	
	\date{}

	\maketitle
	
\begin{abstract} This paper contains three new results. {\bf 1}.We introduce new notions of projective crystalline representations and twisted periodic Higgs-de Rham flows.
	These new notions generalize crystalline representations of \'etale fundamental groups introduced in~\cite{Fal89,FoLa82} and periodic Higgs-de Rham flows introduced in~\cite{LSZ13a}.  We establish an equivalence between the categories of projective crystalline representations and twisted periodic Higgs-de Rham flows via the category of twisted Fontaine-Faltings module which is also introduced in this paper. {\bf 2.}We study the base change of these objects over very ramified valuation rings
and show that a stable periodic Higgs bundle gives rise to a geometrically absolutely irreducible crystalline representation.
{\bf 3.} We investigate the dynamic of self-maps induced by the Higgs-de Rham flow on the moduli spaces of rank-2 stable Higgs bundles of degree 1 on $\mathbb{P}^1$ with logarithmic structure on marked points $D:=\{x_1,\,...,x_n\}$  for $n\geq 4$ and construct infinitely many geometrically absolutely irreducible $\mathrm{PGL_2}(\Z_p^{\mathrm{ur}})$-crystalline representations of   $\pi_1^\text{et}(\mathbb{P}^1_{{\mathbb{Q}}_p^\text{ur}}\setminus D)$. We find an explicit formula of the self-map for the case $\{0,\,1,\,\infty,\,\lambda\}$  and conjecture that 
 a Higgs bundle is periodic if and only if the zero of the Higgs field is the image of a torsion point in the associated elliptic curve $\mathcal{C}_\lambda$ defined by $ y^2=x(x-1)(x-\lambda)$ with the order coprime to $p$. 
\end{abstract}
	\tableofcontents

\section{Introduction}

	The nonabelian Hodge theory established by Hitchin and Simpson associates representations of the topological fundamental group of an algebraic variety $X$ over $\C$ to a holomorphic object on $X$ named Higgs bundle. Later, Ogus and Vologodsky established the nonabelian Hodge theory in positive characteristic in their fundamental work~\cite{OgVo07}. They constructed the Cartier functor and the inverse Cartier functor, which give an equivalence of categories between the category of nilpotent Higgs modules of exponent $\leq p-1$ and the category of nilpotent flat modules of exponent $\leq p-1$ over a smooth proper and $W_2(k)$-liftable variety. This equivalence generalizes the classical Cartier descent theorem. 
Fontaine-Laffaille~\cite{FoLa82} for $\X= \Spec\,W(k)$ and Faltings   in general case have introduced the category $\mathcal{MF}_{[a,b]}^{\nabla}(\X/W)$.   The objects   in  $\mathcal{MF}_{[a,b]}^{\nabla}(\X/W)$ are the so-called \emph{ Fontaine-Faltings modules} and consist of a quadruple $(V,\nabla,\Fil,\varphi)$, where $(V,\nabla,\Fil)$ is a filtered de Rham bundle over $\X$ and $\varphi$ is a relative Frobenius which is horizontal with respect to $\nabla$ and satisfies the strong $p$-divisibility condition. The latter condition is a $p$-adic analogue of the Riemann-Hodge bilinear relations. Then the Fontaine-Laffaille-Faltings correspondence gives a fully faithful functor from $\mathcal{MF}_{[0,w]}^{\nabla}(\X/W)$ $(w \leq p-2)$ to the category of \emph{crystalline representations} of $\pi^\text{\'et}_1(X_K)$, where $X_K$ is the generic fiber of $\X$. This can be regarded as a $p$-adic version of the Riemann-Hilbert correspondence.\\[.1cm]
	Faltings \cite{Fal05} has established an equivalence of categories between the category of generalized representations of the geometric fundamental group and the category of Higgs bundles over a p-adic curve, which has generalized the earlier work of Deninger-Werner \cite{DW} on a partial p-adic analogue of Narasimhan-Seshadri theory. \\[.1cm]
Lan, Sheng and Zuo have established a $p$-adic analogue of the Hitchin-Simpson correspondence between the category of  $\mathrm{GL}_r(W_n(\mathbb F_q))$-crystalline representations and the category of graded periodic Higgs bundles by introducing the notion of  \emph{Higgs-de Rham flow}. It is a sequence of graded Higgs bundles and filtered de Rham bundles, connected by the inverse Cartier transform defined by Ogus and Vologodsky \cite{OgVo07} and the grading functor by the attached Hodge filtrations on the de Rham bundles (for details see Section $3$ in \cite{LSZ13a} or Section \ref{section HDF} in this paper).   
	\\[.1cm]
	A periodic Higgs bundle must have trivial Chern classes. This fact limits the application of the $p$-adic Hitchin-Simpson correspondence. For instance, Simpson constructed a canonical Hodge bundle $\Omega^1_X \oplus \mathcal{O}_X$ on $X$ in his proof of the Miyaoka-Yau inequality (Proposition 9.8 and Proposition 9.9 in~\cite{Simpson}), which has nontrivial Chern classes in general. In fact, the classical nonabelian Hodge theorem tells us that the Yang-Mills-Higgs equation is still solvable for a polystable Higgs bundle with nontrivial Chern classes. Instead of getting a flat connection, one can get a \emph{projective flat connection} in this case, whose monodromy gives a $\mathrm{PGL}_r$-representation of the fundamental group. This motivates us to find a $p$-adic Hitchin-Simpson correspondence for graded Higgs bundles with nontrivial Chern classes.\\[.1cm]
 As the first main result of this paper we introduce the $1$-periodic \emph{twisted Higgs-de Rham flow} over $X_1$ as follows
	\[
	\xymatrix{ 
		&  (V,\nabla,\Fil)_0\ar[dr]^{\mathrm{Gr}(\cdot)\otimes (L,0)}  &     \\
		(E,\theta)_0 \ar[ur]^{C_1^{-1}}  &  & (E,\theta)_1\otimes(L,0) \ar@/^1pc/[ll]^{\phi_L}_\sim 
	}
	\]
	Here $L$ is called a twisting line bundle on $X_1$, and $\phi_L : (E_1,\theta_1) \otimes (L,0) \cong (E_0,\theta_0)$ is called the twisted $\phi$-structure.\\
	On the Fontaine module side, we also introduce the \emph{twisted Fontaine-Faltings module} over $X_1$. The latter consists of the following data: a filtered de Rham bundle $(V,\nabla,\Fil)$ together with an isomorphism between de Rham bundles:
	\[
	\varphi_L: (C^{-1}_{1} \circ \mathrm{Gr}_{\Fil}(V,\nabla)) \otimes (L^{\otimes p},\nabla_{\mathrm{can}}) \cong (V,\nabla).
	\] We will refer to the isomorphism
	$\varphi_L$ as the twisted $\varphi$-structure. The general construction of twisted Fontaine-Faltings modules and twisted periodic Higgs-de Rham flows are given in Section~\ref{section TFFMES} and Section~\ref{section TPHDF} (over $X_n/W_n(k)$, and multi-periodic case).

	\begin{thm}[Theorem~\ref{equiv:TFF&THDF}]\label{thm:second_thm}
		Let $\X$ be a smooth proper scheme over $W$. For each integer $0 \leq a \leq p-2$ and each $f \in \mathbb{N}$, there is an equivalence of categories between the category of all twisted $f$-periodic Higgs-de Rham flows over $X_n$ of level $\leq a$ and the category of strict $p^n$-torsion twisted Fontaine-Faltings modules over $X_n$ of Hodge-Tate weight $\leq a$ with an endomorphism structure of $W_n(\mathbb{F}_{p^f})$.
	\end{thm}
	
  Theorem \ref{thm:second_thm} can be   generalized to the logarithmic case.[Theorem~\ref{equiv:logTFF&THDF}] \\[.2cm]
  The next goal is to associate a $\mathrm{PGL}_n$-representation of $\pi^\text{\'et}_1$ to a twisted (logarithmic) Fontaine-Faltings module. To do so, we need to generalize Faltings' work.  Following Faltings \cite{Fal89}, we construct a functor $\mathbb{D}^P$ in section~\ref{section FDP}, which associates to a twisted (logarithmic) Fontaine-Faltings module a $\mathrm{PGL}_n$ representation of the \'etale fundamental group.
	\begin{thm}[Theorem~\ref{ConsFunc:D^P}]
		Let $\X$ be a smooth proper geometrically connected scheme over $W$ with a simple normal crossing divisor $\D\subset \X$ relative to $W$. Suppose $\mathbb F_{p^f}\subset k$. 
		Let $M$ be a twisted logarithmic Fontaine-Faltings module over $\X$ (with pole along $\D$) with endomorphism structure of $W(\F_{p^f})$. Applying $\mathbb D^P$-functor, one gets a projective representation
		\[\rho : \pi^\text{\'et}_1(X_K^o) \to \mathrm{PGL}(\mathbb{D}^P(M)),\]
		where $X_K^o$ is the generic fiber of $\X^o=\X\setminus \D$.
	\end{thm}
	In Section~\ref{section SRSPHdRF},  we study several properties of this functor $\mathbb{D}^P$.  For instance, we prove that a projective subrepresentation of $\mathbb{D}^P(M)$   corresponds to a sub-object $N \subset M$ such that $\mathbb{D}^P(M/N)$ is isomorphic to this subrepresentation. Combining this with Theorem~\ref{equiv:TFF&THDF}, we infer that a projective representation coming from a stable twisted periodic Higgs bundle $(E,\theta)$ with $(\mathrm{rank}(E),\deg_H(E))=1$ must be irreducible.\\

The next theorem gives a $p$-adic analogue of the existence
of projective flat Yang-Mills-Higgs connection in terms of semistability of Higgs bundles and triviality of the discriminant.

\begin{thm}[Theorem~\ref{Main: preperiod}] A semistable Higgs bundle over $X_1$  initials a twisted preperiodic Higgs-de Rham flow if and only if it is semistable and has trivial discriminant.\end{thm}
Consequently we obtain the existence of non-trivial representations of \'etale fundamental group in terms of the existence of semistable graded Higgs bundles.\\[.2cm]
\begin{defi} A representation $\pi^\text{\'et}_1(X^o_K)\to \PGL_r(\mathbb{F}_q)$ is called geometrically absolutely irreducible if
its pull-back to the geometric fundamental group 
$$\bar\rho: \pi^\text{\'et}_1( {X^o}_{\bar {\bQ}_p})\to  \PGL_r(\bF_{q})$$
 is absolutely   irreducible, i.e. it is irreducible as a $\PGL_r(\bar{\bF}_p)$-representation.  
\end{defi}
\begin{thm}[Theorem~\ref{Mainthm}]Let $k$ be a finite field of characteristic $p$. Let $\X$ be a smooth proper geometrically connected scheme over $W(k)$ together with a smooth log structure $\D/W(k)$ and let $\X^o=\X\setminus \D.$
 Assume that there exists a semistable graded logarithmic Higgs bundle 
		$(E,\theta)/(\X,\D)_1$  with $r:=\mathrm{rank}( E) \leq p-1,$  discriminant $\Delta_H(E)=0$, $r$ and $\deg_H(E)$ are coprime.  Then there exists a positive integer $f$ and a geometrical absolutely irreducible  $\PGL_r(\F_{p^f})$-representation $\rho$ of $\pi^\text{et}_1(X^o_{K'})$, where $\X^o=\X\setminus \D$ and $K'=W(k\cdot\F_{p^f})[1/p]$.
\end{thm}
The proof of Theorem $0.5$ will be divided into two parts. We first show the existence of the irreducible projective representation of $\pi^\text{\'et}_1(X^o_{K'})$, in section $3$ (see Theorem~\ref{Mainthm}).
The proof for the geometric irreducibility of $\rho$  will be postponed to Section-5.  \\[.1cm]
The second main result of this paper, the so-called \emph{ base changing of the projective Fontaine-Faltings module and twisted Higgs-de Rham flow over a very ramified valuation ring  $V$}  is introduced in Section-5.  We show that there exists an equivalent functor from the category of twisted periodic Higgs-de Rham flow over $ \sX_{\pi,1}$  to the category of twisted Fontaine-Faltings modules over $ \sX_{\pi,1},$ where  $ \sX_{\pi,1}$ is the closed fiber of the formal completion of the base change of $\mathcal{X}$ to the PD-hull of $V$.
As a consequence,  we prove the second statement of Theorem 0.5 on the geometric absolute irreducibility of $\rho$ in Subsection 5.4 (see Theorem~\ref{ramified_Thm}).\\[.1cm]
We like to emphasize that the Fontaine-Faltings module and Higgs-de Rham flow over a very ramified valuation ring $V$  introduced here shall be a crucial step toward to constructing  $p$-adic Hitchin-Simpson correspondence between the category of de Rham representations of
 $\pi_1^\text{\'et}(X_{V[1/p]})$  and the category of periodic Higgs bundles over a potentially semistable reduction $\mathcal{X}_{V}$.\\[.2cm]
As the third ingredient of this paper, we investigate the dynamic of Higgs-de Rham flows on the projective line with marked points in Section~\ref{section CCREFGpCHB}.  
Taking the moduli space $M$ of graded stable Higgs bundles of rank-$2$  and degree $1$  over $\P^1$  with logarithmic structure on $m\geq 4$ marked points we show that the self-map induced by Higgs-de Rham flow stabilizes the component  $M(1,0)$ of $M$  of maximal dimension $m-3$ as a rational and dominant map.
Hence by Hrushovski's theorem \cite{Hru}  the subset of periodic Higgs bundles is Zariski dense in $M(1,0)$. In this way, we produce infinitely many geometrically absolutely irreducible $\mathrm{PGL}_2(\F_{p^f})$-crystalline representations.   By Theorem~\ref{Mainthm},   all these representations lift to $\mathrm{PGL}_2(\mathbb Z_p^{ur})$-crystalline representations. In Proposition~\ref{strong_irred} we show that all those lifted representations are strongly irreducible.\\
 For the case of four marked points $\{0,1,\infty,\lambda\} $ we state an explicit formula for the self-map
 and use it to study the dynamics of Higgs-de Rham flows for $p=3$ and several values of $\lambda$. \\[.1cm]
 Much more exciting, we claim that  (Conjecture \ref{conj-1}) the self-map on the moduli space $M(1,0)$ induced by the Higgs-de Rham flow for $\mathbb{P}^1\supset \{0,\,1,\infty,\lambda\}$ coincides with the multiplication
 by $p$ map on the associated elliptic curve  
  defined as the double cover $\pi: \mathcal{C}_\lambda\to\mathbb{P}^1$ 
and   ramified on
 $\{0,\,1,\infty,\lambda\}$.  We have checked this conjecture holds true for $p\leq 50.$   It really looks surprised that the self-map coming from nonabelian $p$-adic Hodge theory has really something to do with the addition law on an elliptic curve. \\[.1cm]
For   $\ell$-adic representations Kontsevich has observed a relation between the set of isomorphic classes of $\text{GL}_2(\bar{\mathbb Q}_l)$-local systems over $\mathbb{P}^1\setminus \{0,\,1,\infty,\lambda\}$ over $\mathbb{F}_q$ and the set of rational points on $C_\lambda$ over $\mathbb{F}_q$  via the work of Drinfeld on the Langlands program over function field. It looks quite mysterious. There should exist a relation between periodic Higgs bundles in the $p$-adic world and the Hecke-eigenforms in the $\ell$-adic world via Abe's solution of Deligne conjecture on $\ell$-to-$p$ companions. We plan to carry out this program in a further coming paper joint with J.Lu and X.Lu \cite{preparation}.\\[.2cm]
 In the last subsection \ref{proj F-units}, we consider a smooth projective curve $\X$  over $W(k)$  of genus $g\geq2$. In the Appendix of \cite{Osserman}, de Jong and Osserman have shown that the subset of twisted periodic vector bundles over $X_1$   in the moduli space of semistable vector bundles over $X_1$ of any rank and any degree is always Zariski dense. By applying our main theorem for twisted periodic Higgs de Rham flows with zero Higgs fields, which should be regarded as projective \'etale trivializable vector bundles in the projective version of Lange-Stuhler's theorem (see~\cite{LangeStuhe}), they all correspond to $\mathrm{PGL}_r(\F_{p^f})$- representations of $\pi^\text{\'et}_1(X_1)$.
 Once again we show that they all lift to $\mathrm{PGL}_r(\Z_p^\mathrm{ur})$ of  $\pi^\text{\'et}_1(X_1)$. 
 It should be very interesting to make a comparison between the lifting theorem obtained here lifting  $\mathrm{GL}_r(\F_{p^f})$-representations of $\pi^\text{\'et}_1(X_1)$  to $\mathrm{GL}_r(\Z_p^\mathrm{ur})$-representation of $\pi^\text{\'et}_1({X_1}_{\bar{\F}_p})$ and the lifting theorem developed by Deninger-Werner~\cite{DW}. In their paper, they have shown that any vector bundle over $\X/W$  which is \'etale trivializable over $X_1$  lifts to a $\mathrm{GL}_r(\mathbb{C}_p)$-representation of $\pi^\text{\'et}_1(X_{\overline K})$.\\ 

%	Finally, we give some example in section~\ref{section CCREFGpCHB} to show how our machinery works in a concrete situation. We study the moduli space of graded stable Higgs bundles of rank $2$ degree $1$ over $\mathbb{P}^1$, with $m$ logarithmic poles $\{x_1,x_2,\cdots,x_m\}$. The self-map on this moduli scheme induced by the twisted Higgs-de Rham flow is a rational map. And the periodic points of this self-map is Zariski dense. Hence in the component of moduli space of dimensional $m-3$ the $p$-adic Hitchin-Simpson correspondence discussed above, gives us infinitely many irreducible crystalline projective representations of the fundamental group.

	\section{Twisted Fontaine-Faltings modules}\label{section TFFM}
	In this section, we will recall the definition of Fontaine-Faltings modules in~\cite{Fal89} and generalize it to the twisted version. 

   \subsection{Fontaine-Faltings modules}\label{section FFM}

	Let $X_n$ be a smooth and proper variety over $W_n(k)$.
	And $(V,\nabla)$ is a \emph{de Rham sheaf} (i.e. a sheaf with an integrable connection) over $X_n$. In this paper, a filtration $\Fil$ on $(V,\nabla)$ will be called a \emph{Hodge filtration of level in $[a,b]$} if the following conditions hold:
	\begin{itemize}
		\item[-] $\Fil^i V$'s are locally split sub-sheaves of $V$, with 
		\[V=\Fil^aV\supset \Fil^{a+1}V \supset\cdots \supset \Fil^bV\supset \Fil^{b+1}V=0,\]
		and locally on all open subsets $U\subset X_n$, the graded factor\\ $\Fil^i V(U)/\Fil^{i+1} V(U)$ are finite direct sums of $\mathcal O_{X_n}(U)$-modules of form $\mathcal O_{X_n}(U)/p^e$. 
		\item[-] $\Fil$ satisfies Griffiths transversality with respect to the connection $\nabla$.
	\end{itemize} 
	In this case, the triple $(V,\nabla,\Fil)$ is called a \emph{filtered de Rham sheaf}. One similarly gives the conceptions of \emph{ (filtered) de Rham modules over a $W$-algebra}. 
	
	\subsubsection{Fontaine-Faltings modules over a small affine base.}
	Let $\U=\mathrm{Spec}R$ be a small affine scheme ( which means there exist an \'etale map  $$W_n[T_1^{\pm1},T_2^{\pm1},\cdots, T_{d}^{\pm1}]\rightarrow \sO_{X_n}(U),$$ see \cite{Fal89}) over $W$ and   $\Phi:\widehat{R}\rightarrow\widehat{R}$ be a lifting of the absolute Frobenius on $R/pR$, where $\widehat{R}$ is the $p$-adic completion of $R$. A Fontaine-Faltings module over $\U$ of Hodge-Tate weight in $[a,b]$ is a quadruple $(V,\nabla,\Fil,\varphi)$, where 
	\begin{itemize}
		\item[-] $(V,\nabla)$ is a de Rham $R$-module;
		\item[-] $\Fil$ is a Hodge filtration on $(V,\nabla)$ of level in $[a,b]$; 
		%	\item[-] $(V,\Fil)$ is a filtered $R$-module with split embeddings
		%	\[V=\Fil^a V \supset \Fil^{a+1} V\supset \cdots \supset \Fil^b V\supset \Fil^{b+1} V=0,\]
		%	and the graded factor $\Fil^i V/\Fil^{i+1} V$ are finite direct sums of $R$-modules of form $R/p^eR$. 
		%	\item[-] $\nabla$ is an integrable connection on $V$ satisfying the Griffiths transversality. i.e. for all $i\in \Z$
		%	\[\nabla(\Fil^i(V))\subset \Fil^{i-1}(V)\otimes \Omega_\U^1.\]
		\item[-] $\varphi$ is an $R$-linear isomorphism \[\varphi:F^*_{\widehat{\U},\Phi}\widetilde{V}=\widetilde{V}\otimes_{\Phi}\widehat{R} \longrightarrow V,\]
		where $F^*_{\widehat{\U},\Phi}=\mathrm{Spec}(\Phi)$, $\widetilde{V}$ is the quotient $\bigoplus\limits_{i=a}^b\Fil^i/\sim$ with $x\sim py$ for any $x\in\Fil^iV$ and $y$ is the image of $x$ under the natural inclusion $\Fil^iV\hookrightarrow\Fil^{i-1}V$. 
		% We denote by $[x]_i$ the image of $x\in \Fil^i V$ under the canonical map $\Fil^iV\rightarrow \widetilde{V}$.
		\item[-] The relative Frobenius $\varphi$ is horizontal with respect to the connections $F^*_{\widehat{\U},\Phi}\widetilde{\nabla}$ on $F^*_{\widehat{\U},\Phi}\widetilde{V}$ and $\nabla$ on $V$, 
		%where $\widetilde{\nabla}([x]_i):=[\nabla(x)]_{i-1}$ be a $p$-connection over $\widetilde{V}$.
		i.e. the following diagram commutes:
		\begin{equation*}
		\xymatrix{F^*_{\widehat{\U},\Phi}\widetilde{V} \ar[r]^{\varphi} \ar[d]^{F^*_{\widehat{\U},\Phi}\widetilde{\nabla}} & V\ar[d]^{\nabla} \\
			F^*_{\widehat{\U},\Phi}\widetilde{V}\otimes \Omega_{\U/W}^1 \ar[r]^{\quad \varphi\otimes \mathrm{id}}  & V\otimes \Omega_{\U/W}^1}
		\end{equation*}  
	\end{itemize}
	
	Let $M_1=(V_1,\nabla_1,\Fil_1,\varphi_1)$ and $M_2=(V_2,\nabla_2,\Fil_2,\varphi_2)$ be two Fontaine-Faltings modules over $\U$ of Hodge-Tate weight in $[a,b]$. The homomorphism set between $M_1$ and $M_2$ constitutes by those morphism $f:V_1\rightarrow V_2$ of $R$-modules, satisfying:
	\begin{itemize}
		\item[-] $f$ is strict for the filtrations. i.e. $f^{-1}(\Fil^iV_2)=\Fil^iV_1$.
		\item[-] $f$ is a morphism of de Rham modules. i.e. $(f\otimes \mathrm{id})\circ\nabla_1=\nabla_2\circ f$.
		\item[-] $f$ commutes with the $\varphi$-structures. i.e. $(\widetilde{f}\otimes \mathrm{id})\circ\varphi_1=\varphi_2\circ f$, where $\widetilde{f}$ is the image of $f$ under Faltings' tilde functor.
	\end{itemize}
	Denote by  $\mathcal {MF}_{[a,b]}^{\nabla,\Phi}(\U/W)$ the category of all Fontaine-Faltings modules over $\U$ of Hodge-Tate weight in $[a,b]$.

	\paragraph{\emph{The gluing functor.}} In the following, we recall the gluing functor of Faltings. In other words, up to a canonical equivalence of categories, the category $\mathcal {MF}_{[a,b]}^{\nabla,\Phi}(\U/W)$ does not depend on the choice of $\Phi$. More explicitly, the equivalent functor is given as follows.
	
	Let $\Psi$ be another lifting of the absolute Frobenius. For any filtered de Rham module $(V,\nabla,\Fil)$, Faltings~\cite[Theorem~2.3]{Fal89} shows that there is a canonical isomorphism by Taylor formula
	\[\alpha_{\Phi,\Psi}:	F^*_{\widehat{\U},\Phi}\widetilde{V}\simeq F^*_{\widehat{\U},\Psi}\widetilde{V},\]
	which is parallel with respect to the connection, satisfies the cocycle conditions and induces an equivalent functor of categories 
	\begin{equation}
	\xymatrix@R=0mm{ \mathcal {MF}_{[a,b]}^{\nabla,\Psi}(\U/W)\ar[r]  & \mathcal {MF}_{[a,b]}^{\nabla,\Phi}(\U/W)\\
		(V,\nabla,\Fil,\varphi)\ar@{|->}[r] &  (V,\nabla,\Fil,\varphi\circ\alpha_{\Phi,\Psi})\\}
	\end{equation}
	%Now, denote the category of Fontaine-Faltings modules over $R$ by
	%\[\mathcal {MF}_{[a,b]}^{\nabla}(\U/W) =\coprod_{\Psi} \mathcal {MF}_{[a,b]}^{\nabla,\Psi}(\U/W),\]
	%and for any two Fontaine-Faltings modules $M_1=(V_1,\nabla_1,\Fil_1,\varphi_1)$ and  $M_2=(V_2,\nabla_2,\Fil_2,\varphi_2)$,
	%\[\mathrm{Hom}(M_1,M_2):=\mathrm{Hom}(\iota_\Phi(M_1), \iota_\Phi(M_2)),\]
	%one can check that definition of the homomorphism set is independent the choice of $\Phi$.  Explicitly, for any morphism $f: V_1\rightarrow V_2$ of sheaves, the  $f$ is contained $\mathrm{Hom}(M_1,M_2)$ if and only if $f$ is strict respect to the filtrations and the following two diagrams commute
	%\begin{equation}\label{diag:1}
	%\xymatrix{V_1\ar[r]^{f} \ar[d]^{\nabla_1} & V_2\ar[d]^{\nabla_2} \\
	%	V_1\otimes\Omega_R^1\ar[r]^{f\otimes{id}} & V_2\otimes\Omega_R^1\\}
	%\end{equation}
	%
	%\begin{equation}
	%\xymatrix@C=2cm{
	%	\widetilde{V}_1\otimes_{\Phi_1} R 
	%	\ar[r]^{\widetilde{f}\otimes_{\Phi_1} \mathrm{id}_{R}} \ar[d]^{\varphi_1} 
	%	& \widetilde{V}_2\otimes_{\Phi_1} R\ar[r]^{\alpha_{\Phi_1,\Phi_2}} 
	%	& \widetilde{V}_2\otimes_{\Phi_2} R \ar[d]^{\varphi_2} \\
	%	V_1\ar[r]^f & V_2\ar[r]^{\mathrm{id}_{V_2}} & V_2\\}
	%\end{equation} 

	\subsubsection{Fontaine-Faltings modules over global base.} 
	Let $I$ be the index set of all pairs $(\U_i,\Phi_i)$. The $\U_i$ is a small affine open subset of $\X$, and $\Phi_i$ is a lift of the absolute Frobenius on $\mathcal O_\X(\U_i)\otimes_W k$. Recall that the category $\mathcal {MF}_{[a,b]}^{\nabla}(\X/W)$ is constructed by gluing those categories $\mathcal {MF}_{[a,b]}^{\nabla,\Phi_i}(\U_i/W)$. Actually $\mathcal {MF}_{[a,b]}^{\nabla,\Phi_i}(\U_i/W)$ can be described more precisely as below.
	
	A Fontaine-Faltings module over $\X$ of Hodge-Tate weight in $[a,b]$ is a tuple $(V,\nabla,\Fil,\{\varphi_i\}_{i\in I})$ over $\X$, i.e. a filtered de Rham sheaf $(V,\nabla,\Fil)$ together with  $\varphi_i: \widetilde{V}(\U_i)\otimes_{\Phi_i} \widehat{\mathcal O_\X(\U_i)}\rightarrow V(\U_i)$ such that
	\begin{itemize}
		\item[-] $M_i:=(V(\U_i),\nabla,\Fil,\varphi_i)\in \mathcal {MF}_{[a,b]}^{\nabla,\Phi_i}(\U_i/W)$.
		\item[-]  For all $i,j\in I$, on the overlap open set $\U_i\cap \U_j$, local Fontaine-Faltings modules $M_i\mid_{\U_{i}\cap \U_j}$ and  $M_j\mid_{\U_{i}\cap \U_j}$ are associated to each other by the equivalent functor respecting these two liftings $\Phi_i$ and $\Phi_j$. In other words, the following diagram commutes
		\begin{equation}
		\xymatrix@C=2cm{  
			\widetilde{V}(\U_{ij})\otimes_{\Phi_i} \widehat{\mathcal O_\X(\U_i)}
			\ar[r]^{\alpha_{\Phi_i,\Phi_j}} \ar[d]^{\varphi_i}  
			& \widetilde{V}(\U_{ij})\otimes_{\Phi_j} \widehat{\mathcal O_\X(\U_i)}
			\ar[d]^{\varphi_j} \\
			V(\U_{ij})\ar[r]^{\mathrm{id}} & V(\U_{ij})\\}
		\end{equation} 
	\end{itemize}
	Morphisms between Fontaine-Faltings modules are those between sheaves and locally they are morphisms between local Fontaine-Faltings modules. More precisely, for a morphism $f$ of the underlying sheaves of two Fontaine-Faltings modules over $\X$, the map $f$ is called a morphism of Fontaine-Faltings modules if and only if  $f(\U_i) \in \mathrm{Mor}\left(\mathcal {MF}_{[a,b]}^{\nabla,\Phi_i}(\U_i/W)\right)$, for all $i\in I$.

	Denote by $\mathcal {MF}_{[a,b]}^{\nabla}(\X/W)$ the category of all Fontaine-Faltings modules over $\X$ of Hodge-Tate weight in $[a,b]$. And denote by $\mathcal {MF}_{[a,b]}^{\nabla}(X_{n+1}/W_{n+1})$ the sub-category of $\mathcal {MF}_{[a,b]}^{\nabla}(\X/W)$ consisted of strict $p^n$-torsion Fontaine-Faltings modules over $\X$ of Hodge-Tate weight in $[a,b]$.

	\subsection{Inverse Cartier functor}\label{section ICF}
	
	For a Fontaine-Faltings module $(V,\nabla,\Fil,\{\varphi_i\}_{i\in I})$, we call $\{\varphi_i\}_i$ the $\varphi$-structure of the Fontaine-Faltings module. In this section, we first recall a global description of the $\varphi$-structure via the inverse Cartier functor over truncated Witt rings constructed by Lan, Sheng, and Zuo~\cite{LSZ13a}.
	
	Note that the inverse Cartier functor $C_1^{-1}$ (the characteristic $p$ case) is introduced in the seminal work of Ogus-Vologodsky \cite{OgVo07}. Here we sketch an explicit construction of $C_1^{-1}$ presented in \cite{LSZ13a}.
	Let $(E,\theta)$ be a nilpotent Higgs bundle over $X_1$ of exponent$\leq p-1$. Locally we have
	\begin{itemize}
		\item[-] $V_i=F_{\U_{i}}^*(E\mid_{\U_i})$,
		\item[-] $\nabla_i = \mathrm{d} + \frac{\mathrm{d} \tilde{F}_{\tilde{\U}_i}} {[p]} (F_{\U_{i}}^*\theta\mid_{\U_i}): V_i\rightarrow V_i\otimes \Omega_{\U_i}^1$, 
		\item[-] $G_{ij} = \mathrm{exp}(h_{ij}(F_{\U_{i}}^*\theta\mid_{\U_i})): V_i\mid_{\U_{ij}}\rightarrow V_j\mid_{\U_{ij}}$, 
	\end{itemize}
	where $F_{\U_i}$ is the absolute Frobenius on $\U_i$ and $h_{ij}:F_{\U_{i,1}}^*\Omega^1_{\U_{ij}}\rightarrow \sO_{\U_{ij}}$ is the homomorphism given by the Deligne-Illusie's Lemma~\cite{JJIC02}. Those local data $(V_i,\nabla_i)$'s can be glued into a global sheave $H$ with integrable connection $\nabla$ via the transition maps $\{ G_{ij} \}$ (Theorem 3 in \cite{LSZ12a}).  The inverse Cartier functor on $(E,\theta)$ is 
	\[C^{-1}_1(E,\theta):=(V,\nabla).\]
	
	\begin{rmk}
		Note that the inverse Cartier transform $C^{-1}_1$ also has the logarithmic version. When the log structure is given by a simple normal crossing divisor, an explicit construction of the log inverse Cartier functor is given in the Appendix of \cite{LSYZ14}. 
	\end{rmk}
	As mentioned in the introduction, we need to generalize  $C^{-1}_1$ to the inverse Cartier transform over the truncated Witt ring for Higgs bundles over $X_n/W_m(k)$. We briefly recall the construction in section $4$ of \cite{LSZ13a}.

	\subsubsection{Inverse Cartier functor over truncated Witt ring}
	Let $S=\mathrm{Spec(W(k))}$ and $F_S$ be the Frobenius map on $S$. Let $ X_{n+1}\supset X_n$ be a $W_{n+1}$-lifting of smooth proper varieties. Recall that the functor $C^{-1}_n$ is defined as the composition of $\mathcal C^{-1}_n$ and the base change $F_S:X_n'=X_n\times_{F_S} S \rightarrow X_n$ (by abusing notation, we still denote it by $F_S$). The functor $\mathcal C^{-1}_n$  is defined as the composition of two functors $\mathcal{T}_n$ and $\mathcal{F}_n$. 
	In general, we have the following diagram and its commutativity follows easily from the construction of those functors.
	\begin{equation}\label{diag:C^{-1}}
	\xymatrix{
		& \mathcal{H}(X_n)\ar[drr]_(0.3){C_n^{-1}}|!{[r];[ddr]}\hole  \ar[r]^{F_S^*}\ar[dd]^{\mathcal T_n}
		& \mathcal{H}(X'_n) \ar[dr]^{\mathcal C_n^{-1}} \ar[dd]_{\mathcal T_n} 
		&\\
		\mathrm{MCF}_{p-2}(X_n)\ar[dr]_{\widetilde{(\cdot)}}\ar[ur]^{\overline{\mathrm{Gr}}} 
		&&& \mathrm{MIC}(X_n)\\
		&\widetilde{\mathrm{MIC}}(X_n) \ar[r]_{F_S^*}  \ar@{..>}[urr]^(0.3){\{F_\U^*\}_\U}|!{[r];[uur]}\hole & \widetilde{\mathrm{MIC}}(X'_n)\ar[ur]_{\quad\mathcal F_n=\{F_{\U/S}^*\}_\U} &\\
	}
	\end{equation}
	These categories appeared in the diagram are explained as following:
	\begin{itemize}
		\item $\mathrm{MCF}_{a}(X_n)$ is the category of filtered de Rham sheaves over $X_n$ of level in $[0,a]$.
		\item  $\mathcal{H}(X_n)$ (resp. $\mathcal{H}(X'_n)$)  is the category of tuples $(E,\theta,\bar{V},\bar{\nabla},\overline{Fil},\overline{\psi})$, where 
		\begin{itemize}
			\item[-] $(E,\theta)$ is a graded Higgs module
                    \footnote{A Higgs bundle $(E,\theta)$ is called graded if $E$ can be written as direct sum of sub bundles $E^{i}$ with $\theta(E^i)\subset E^{i-1}\otimes\Omega^1$. Obviously, a graded Higgs bundle is also nilpotent.}
                   over $X_n$ (resp. $X_n'=X_n\otimes_\sigma W$) of exponent $\leq p-2$;
			\item[-] $(\bar{V},\bar{\nabla},\overline{Fil})$ is a filtered de Rham sheaf over $X_{n-1}$ (resp. over $X_{n-1}'$);
			\item[-] and  $\overline{\psi}: Gr_{\bar{Fil}}(\bar{V},\bar{\nabla}) \simeq (E, \theta)\otimes\Z/p^{n-1}\Z$ is an isomorphism of Higgs sheaves over $X_n$ (resp. $X_n'$).
		\end{itemize}    
		\item $\widetilde{\mathrm{MIC}}(X_n)$ (resp. $\widetilde{\mathrm{MIC}}(X'_n)$) is the category of sheaves over $X_n$ (resp. $X'_n$) with integrable $p$-connection . 
		\item $\mathrm{MIC}(X_n)$ (resp. $\mathrm{MIC}(X'_n)$) is the category of de Rham sheaves over $X_n$ (resp. $X'_n$). 
	\end{itemize}
	\paragraph{\emph{Functor $\overline{\mathrm{Gr}}$.}} For an object $(V,\nabla,\Fil)$ in $\mathrm{MCF}_{p-2}(X_n)$, the functor $\overline{\mathrm{Gr}}$ is given by 
	\[\overline{\mathrm{Gr}}(V,\nabla,\Fil)=(E,\theta,\overline{V},\overline{\nabla},\overline{Fil},\overline{\psi}),\]
	where $(E,\theta)=\mathrm{Gr}(V,\nabla,\Fil)$ is the graded sheaf with Higgs field,   $(\overline{V},\overline{\nabla},\overline{Fil})$ is the modulo $p^{n-1}$-reduction of $(V,\nabla,\Fil)$ and $\overline{\psi}$ is the identifying map  $\mathrm{Gr}(\overline{V})\cong E\otimes \Z/p^{n-1}\Z$.\\
	\paragraph{\emph{Faltings tilde functor $\widetilde{(\cdot)}$.}} For an object $(V,\nabla,\Fil)$ in $\mathrm{MCF}_{p-2}(X_n)$, the $\widetilde{(V,\nabla,\Fil)}$ will be denoted as the quotient $\bigoplus\limits_{i=0}^{p-2}\Fil^i/\sim$ with $x\sim py$ for any $x\in\Fil^iV$ and $y$ the image of $x$ under the natural inclusion $\Fil^iV\hookrightarrow\Fil^{i-1}V$.  
	
	\paragraph{\emph{The construction of functor $\mathcal{T}_n$.}} Let $(E,\theta,\bar{V},\bar{\nabla},\overline{Fil},\psi)$ be an object in $\mathcal{H}(X_n)$ (resp. $\mathcal{H}(X'_n)$). Locally on an affine open subset $\U\subset \X$ (resp. $\U\subset \X'$), there exists $(V_\U,\nabla_\U,\Fil_\U)$ (Lemma 4.6 in~\cite{LSZ13a}), a filtered de Rham sheaf, such that 
	\begin{itemize}
		\item[-] $(\bar{V},\bar{\nabla},\overline{Fil})\mid_\U \simeq (V_\U,\nabla_\U,\Fil_\U) \otimes \Z/p^{n-1}\Z$;
		\item[-] $(E,\theta)\mid_\U \simeq \mathrm{Gr}(V_\U,\nabla_\U,\Fil_\U)$. 
	\end{itemize}
	The tilde functor associates to $(V_\U,\nabla_\U,\Fil_\U)$ a sheaf with $p$-connection over $\U$. By gluing those sheaves with $p$-connections over all $\U$'s (Lemma 4.10 in~\cite{LSZ13a}), one gets a global sheaf with $p$-connection over $X_n$ (resp. $X'_n$). Denote it by 
	\[\mathcal T_n(E,\theta,\bar{V},\bar{\nabla},\overline{Fil},\psi).\]
	
	\paragraph{\emph{The construction of functor $\mathcal F_n$.}} For small affine open subset $\U$ of $\X$, there exists endomorphism $F_\U$ on $\U$ which lifts the absolute Frobenius on $\U_k$ and is compatible with the Frobenius map $F_S$ on $S=\mathrm{Spec}(W(k))$. Thus there is a map $F_{\U/S}:\U\rightarrow \U'=\U\times_{F_S} S$  satisfying  $F_\U= F_S\circ F_{\U/S}$.
	\begin{equation}
	\xymatrix{
		\U \ar@/^15pt/[drr]^{F_\U}  \ar[dr]^{F_{\U/S}} \ar@/_15pt/[ddr] &&\\
		& \U' \ar[r]^{F_S}\ar[d] & \U\ar[d]\\
		& S\ar[r]^{F_S} & S\\}
	\end{equation}
	
	Let $(\widetilde{V}',\widetilde{\nabla}')$ be an object in $\widetilde{MIC}(X'_n)$. Locally on $\U$, applying functor $F_{\U/S}^*$, we get a de Rham sheaf over $\U$
	\[F_{\U/S}^*(\widetilde{V}'\mid_{\U'},\widetilde{\nabla}'\mid_{\U'}).\]
	By Taylor formula, up to a canonical isomorphism, it does not depends on the choice of $F_\U$. In particular, on the overlap of two small affine open subsets, there is an canonical isomorphism of two de Rham sheaves. By gluing those isomorphisms, one gets a de Rham sheaf over $X_n$, we denote it by
	\[\mathcal F_n(\widetilde{V}',\widetilde{\nabla}').\]

	\subsection{Global description of the $\varphi$-structure in Fontaine-Faltings modules (via the inverse Cartier functor).}\label{section GDphiFFM}
	Let $(V,\nabla,\Fil)\in \mathrm{MFC}_{p-2}(X_n)$ be a filtered de Rham sheaf over $X_n$ of level in $[0,p-2]$. From the commutativity of diagram~(\ref{diag:C^{-1}}), for any $i\in I$, one has
	\begin{equation}
	\begin{split}
	C_n^{-1}\circ\overline{\mathrm{Gr}}(V,\nabla,\Fil)\mid_{\U_i} &= \mathcal F_n\circ\mathcal T_n\circ F_S^*\circ \overline{\mathrm{Gr}}(V,\nabla,\Fil)\mid_{\U_i}\\
	&=\mathcal F_n\circ F_S^* (\widetilde{V},\widetilde{\nabla})\mid_{\U_i}\\
	&\simeq F_{\U_i}^*(\widetilde{V}\mid_{\U_i},\widetilde{\nabla}\mid_{\U_i}).
	\end{split}
	\end{equation}
	Here $F_{\U_i}=\mathrm{Spec}(\Phi_i): \U_i\rightarrow \U_i$ is the lifting of the absolute Frobenius on $\U_{i,k}$. As the $\mathcal F_n$ is glued by using the Taylor formula, for any $i,j\in I$, one has the following commutative diagram
	\begin{equation}
	\xymatrix{
		F_{\U_i}^*(\widetilde{V}\mid_{\U_i\cap \U_j},\widetilde{\nabla}\mid_{\U_i\cap \U_j})
		\ar[r]^(0.47){\sim} \ar[d]^{\alpha_{\Phi_i,\Phi_j}}
		& C^{-1}_n\circ \overline{\mathrm{Gr}}(V,\nabla,\Fil) \mid_{\U_i\cap \U_j} \ar@{=}[d]\\
		F_{\U_j}^*(\widetilde{V}\mid_{\U_i\cap \U_j},\widetilde{\nabla}\mid_{\U_i\cap \U_j}) \ar[r]^(0.47){\sim}
		& C^{-1}_n\circ \overline{\mathrm{Gr}}(V,\nabla,\Fil)\mid_{\U_i\cap \U_j}\\ 
	}
	\end{equation}
	To give a system of compatible $\varphi$-structures (for all $i\in I$)
	\[\varphi_{i}: F_{\U_i}^*(\widetilde{V}\mid_{\U_i},\widetilde{\nabla}\mid_{\U_i}) \rightarrow (V\mid_{\U_i},\nabla\mid_{\U_i}),\]
	it is equivalent to give an isomorphism 
	\[\varphi:C^{-1}_n\circ \overline{\mathrm{Gr}}(V,\nabla,\Fil) \rightarrow (V,\nabla).\]
	In particular, we have the following results
	\begin{lem}[Lemma 5.6 in~\cite{LSZ13a}] \label{lem:anotherDiscribtionFM}
		To give a Fontaine-Faltings module in $\mathcal{MF}_{[0,p-2]}^\nabla(\X/W)$, it is equivalent to give a tuple $(V,\nabla,\Fil,\varphi)$, where
		\begin{itemize}
			\item[-] $(V,\nabla,\Fil)\in \mathrm{MCF}_{p-2}(X_n)$  is a filtered de Rham sheaf over $X_n$ of level in $[0,p-2]$, for some positive integer $n$;
			\item[-] $\varphi: C_n^{-1}\circ \overline{\mathrm{Gr}} (V,\nabla,\Fil)\rightarrow (V,\nabla)$ is an isomorphism of de Rham sheaves.
		\end{itemize} 
	\end{lem}

	\subsection{Fontaine-Faltings modules with endomorphism structure.}\label{section FFMES}
	Let $f$ be a positive integer. We call $(V,\nabla,\Fil,\varphi,\iota)$ a Fontaine-Faltings module over $\X$ with endomorphism structure of $W(\F_{p^f})$ whose Hodge-Tate weights lie in $[a,b]$, if $(V,\nabla,\Fil,\varphi)$ is an object in $\mathcal {MF}_{[a,b]}^{\nabla}(\X/W)$ and  
	\[\iota: W(\F_{p^f})\rightarrow \mathrm{End}_{\mathcal{MF}}(V,\nabla,\Fil,\varphi)\]
	is a continuous ring homomorphism. We call $\iota$ an endomorphism structure of $W(\F_{p^f})$ on $(V,\nabla,\Fil,\varphi)$. Let's denote by $\mathcal {MF}_{[a,b],f}^{\nabla}(X/W)$ the category of Fontaine-Faltings module with endomorphism structure of $W(\F_{p^f})$ whose Hodge-Tate weights lie in $[a,b]$. And denote by $\mathcal {MF}_{[0,p-2],f}^{\nabla}(X_{n+1}/W_{n+1})$ the subcategory of $\mathcal {MF}_{[0,p-2],f}^{\nabla}(\X/W)$ consisted by strict $p^n$-torsion objects.  
	\begin{lem}\label{lem:anotherDiscribtionFMwithEnd} Assume $f$ is a positive integer with $\F_{p^f}\subset k$. Then giving an object in  
		$\mathcal {MF}_{[0,p-2],f}^{\nabla}(\X/W)$ is equivalent to give $f$-ordered objects \[ (V_i,\nabla_i,\Fil_i)\in \mathrm{MCF}_{p-2}(X_n), \qquad i=0,1,\cdots,f-1\]
		(for some $n\in \N$) together with isomorphisms of de Rham sheaves
		\[\varphi_i:C_n^{-1}\circ\overline{\mathrm{Gr}}(V_i,\nabla_i,\Fil_i) \rightarrow (V_{i+1},\nabla_{i+1}), \quad \text{ for } 0\leq i\leq f-2,\]
		and 
		\[\varphi_{f-1}:C_n^{-1}\circ\overline{\mathrm{Gr}}(V_{f-1},\nabla_{f-1},\Fil_{f-1}) \rightarrow (V_{0},\nabla_{0}).\]
	\end{lem}
	\begin{proof} Let $(V,\nabla,\Fil,\varphi,\iota)$ be an object in $\mathcal {MF}_{[0,p-2],f}^{\nabla}(\X/W)$.
		Let $\sigma$ be the Frobenius map on $W(\F_{p^f})$ and let $\xi$ be a generator of $W(\F_{p^f})$ as a $\Z_p$-algebra. Then $\iota(\xi)$ is an endomorphism of the Fontaine-Faltings module $(V,\nabla,\Fil,\varphi)$. Since $\F_{p^f}\subset k$, all conjugate elements of $\xi$ are of form $\sigma^i(\xi)$, which are contained in $w=w(K)$. The filtered de Rham sheaf $(V,\nabla,\Fil)$ can be decomposed into eigenspaces
		\[(V,\nabla,\Fil)=\bigoplus_{i=0}^{f-1}(V_i,\nabla_i,\Fil_i),\]
		where $(V_i,\nabla_i,\Fil_i)=(V,\nabla,\Fil)^{\iota(\xi)=\sigma^i(\xi)}$ is the $\sigma^i(\xi)$-eigenspace of $\iota(\xi)$.
		Applying  $C_n^{-1}\circ\overline{\mathrm{Gr}}$ on both side, we get 
		\[C_n^{-1}\circ\overline{\mathrm{Gr}}(V,\nabla,\Fil)=\bigoplus_{i=0}^{f-1}C_n^{-1}\circ\overline{\mathrm{Gr}}(V_i,\nabla_i,\Fil_i).\]
		Comparing $\sigma^{i+1}(\xi)$-eigenspaces of $\iota(\xi)$ on both side of \[\varphi:C_n^{-1}\circ\overline{\mathrm{Gr}}(V,\nabla,\Fil)\simeq (V,\nabla),\]
		one gets the restrictive isomorphisms
		\[\varphi_i:C_n^{-1}\circ\overline{\mathrm{Gr}}(V_i,\nabla_i,\Fil_i) \rightarrow (V_{i+1},\nabla_{i+1}), \quad \text{for all } 0\leq i\leq f-2,\]
		and 
		\[\varphi_{f-1}:C_n^{-1}\circ\overline{\mathrm{Gr}}(V_{f-1},\nabla_{f-1},\Fil_{f-1}) \rightarrow (V_{0},\nabla_{0}).\]
		Conversely, we can construct the Fontaine-Faltings module with endomorphism structure in an obvious way.
	\end{proof}

	\subsection{Twisted Fontaine-Faltings modules with endomorphism structure.}\label{section TFFMES}
	Let $L_n$ be a line bundle over $X_n$. Then there is a natural connection $\nabla_{\mathrm{can}}$ on $L_n^{p^{n}}$ by 5.1.1 in~\cite{KaNi70}. Tensoring with $(L_n^{p^{n}}, \nabla_{\mathrm{can}})$ induces a self equivalence functor
	on the category of de Rham bundles over $X_n$.
	
	\begin{defi}
		An $L_n$-twisted Fontaine-Faltings module over $X_n$  with endomorphism structure of $W_n(\mathbb F_{p^f})$ whose Hodge-Tate weights lie in $[a,b]$ is a tuple consisting the following data:
		\begin{itemize}
			\item[-] for $0\leq i\leq f-1$, a filtered de Rham bundle $(V_i,\nabla_i,\Fil_i)$ over $X_n$ of level in $[a,b]$;
			\item[-] for $0\leq i\leq f-2$, an isomorphism of de Rham sheaves
			\[\varphi_i:  \mathcal C^{-1}_n\circ \overline{\mathrm{Gr}} (V_i,\nabla_i,\Fil_i)\rightarrow  (V_{i+1},\nabla_{i+1});\]
			\item[-] an isomorphism of de Rham sheaves
			\[\varphi_{f-1}: \left(\mathcal C^{-1}_n\circ \overline{\mathrm{Gr}}(V_{f-1},\nabla_{f-1},\Fil_{f-1})\right) \otimes(L_n^{p^{n}},\nabla_{\mathrm{can}}) \rightarrow  (V_0,\nabla_0).\]
		\end{itemize} 
		We use $(V_i,\nabla_i,\Fil_i,\varphi_i)_{0\leq i<f}$ to denote the $L_n$-twisted Fontaine-Faltings module and use $\mathcal {TMF}_{[a,b],f}^{\nabla}(X_{n+1}/W_{n+1})$ to denote the category of all twisted Fontaine-Faltings modules over $X_n$ with endomorphism structure of $W_n(\mathbb F_{p^f})$ whose Hodge-Tate weights lie in $[a,b]$.
	\end{defi}
	A morphism between two objects $(V_i,\nabla_i,\Fil_i,\varphi_i)_{0\leq i<f}$ and $(V'_i,\nabla'_i,\Fil'_i,\varphi'_i)_{0\leq i<f}$ is an $f$-tuple $(g_0,g_1,\cdots,g_{f-1})$ of morphisms of filtered de Rham sheaves
	\[g_i:(V_i,\nabla_i,\Fil_i)\rightarrow (V'_i,\nabla'_i,\Fil'_i), \quad i=0,1,\cdots,f-1\]
	satisfying 
	\[g_{i+1}\circ \varphi_i=\varphi_i'\circ \left(\mathcal C^{-1}_n\circ\mathrm{Gr}(g_i)\right), \qquad \text{for } 0\leq i\leq f-2,\]
	and 
	\[\left(g_0\otimes \mathrm{id}_{L_n^{p^n}}\right)\circ \varphi_{f-1}=\varphi_{f-1}'\circ \left(\mathcal C^{-1}_n\circ\mathrm{Gr}(g_{f-1})\right).\]
	
	\begin{rmk} $i)$. By Lemma~\ref{lem:anotherDiscribtionFM}, to give an object in $\mathcal{TMF}_{[a,b],1}^{\nabla}(X_{n}/W_{n})$ with $L_n=\mathcal O_{X_n}$ is equivalent to give a strict $p^n$-torsion Fontaine-Faltings module over $X_n$ whose Hodge-Tate weights lie in $[a,b]$.\\
		$ii)$. Suppose $\mathbb F_{p^f}\subset k$. By Lemma~\ref{lem:anotherDiscribtionFMwithEnd}, to give an object in $\mathcal{TMF}_{[a,b],f}^{\nabla}(X_{n}/W_{n})$ with $L_n=\mathcal O_{X_n}$ is equivalent to give a strict $p^n$-torsion Fontaine-Faltings module over $X_n$ with endomorphism structure of $W_n(\mathbb F_{p^f})$ and whose Hodge-Tate weight in $[a,b]$ .\\	
	\end{rmk}
	
	\paragraph{\emph{Local trivialization.}}
	Let $j\in I$. Locally on the small open affine set $\U_j$ ($R_j=\mathcal O_\X(\U_j)$), we choose and fix a lifting $\Phi_j:\widehat{R}_j\rightarrow \widehat{R}_j$ and a trivialization of the line bundle $L_n$ 
	\[ \tau_j:\mathcal O_{X_n}(\U_j) \simeq L_n(\U_j).\]
	It induces a trivialization of de Rham bundle $\tau_j^{p^n}: (\mathcal O_{X_n}(\U),\mathrm{d})  \simeq  (L_n^{p^{n}}(\U),\nabla_{\mathrm{can}})$.
	Let $M=(V_i,\nabla_i,\Fil_i,\varphi_i)_{0\leq i<f}\in \mathcal {TMF}_{[a,b],f}^{\nabla}(X_{n+1}/W_{n+1})$ be an $L_n$-twisted Fontaine-Faltings module over $X_n$  with endomorphism structure of $W_n(\mathbb F_{p^f})$ whose Hodge-Tate weights lie in $[a,b]$. Then one gets a local Fontaine-Faltings module over $R_j$ with endomorphism structure of $W_n(\mathbb F_{p^f})$ whose Hodge-Tate weights lie in $[a,b]$
	\[M(\tau_j)=\Big(\oplus V_i(\U_j), \oplus \nabla_i,\oplus \Fil_i, \sum_{i=0}^{f-2}\varphi_i+\varphi_{f-1}\circ(\mathrm{id}\otimes\tau_j^{p^n})\Big).\]
	We call $M(\tau_j)$ the \emph{trivialization of of $M$ on $\U_j$ via $\tau_j$}.  
 
		\paragraph{\emph{Logarithmic version.}}
		Finally, let us mention that everything in this section extends to the logarithmic context. 
		Let $\X$ be a smooth and proper scheme over $W$ and $\X^o$ is the complement of a simple normal crossing divisor $\D\subset \X$ relative to $W$. Similarly, one constructs the category $\mathcal {TMF}_{[a,b],f}^{\nabla}(X^o_{n+1}/W_{n+1})$ of strict $p^n$-torsion twisted logarithmic Fontaine modules (with pole along $\D\times W_n\subset \X\times W_n$) with endomorphism structure of $W_n(\F_{p^f})$ whose Hodge-Tate weights lie in $[a,b]$.

	\section{Projective Fontaine-Laffaille-Faltings functor}\label{section PFLFF}
	
	\subsection{The Fontaine-Laffaille-Faltings' $\mathbb D$-functor}\label{section FDF}
	
	\paragraph{\emph{The functor $\mathbb{D}_\Phi$.}} Let $R$ be a small affine algebra over $W=W(k)$ with a $\sigma$-linear map $\Phi:\widehat{R}\rightarrow\widehat{R}$ which lifts the absolute Frobenius of $R/pR$.
	If $\Phi$ happens to be \'etale in characteristic $0$, Faltings (page 36 of \cite{Fal89}) constructed a map $\kappa_\Phi:\widehat{R}\rightarrow B^+(\widehat{R})$ which respects Frobenius-lifts. Thus the following diagram commutes
	\begin{equation}
	\xymatrix@C=2cm{\widehat{R} \ar[r]^{\kappa_\Phi} \ar[d]^{\Phi}
		& B^+(\widehat{R})  \ar[d]^{\Phi_B}\\
		\widehat{R} \ar[r]^{\kappa_{\Phi}} & B^+(\widehat{R}).\\
	}
	\end{equation} 
	Here $\Phi_B$ is the Frobenius on $B^+(\widehat{R})$. Denote $D=B^+(\widehat R)[1/p]/B^+(\widehat R)$, which is equipped with the natural $\varphi$-structure and filtration.
	
	Let $M=(V,\nabla,\Fil,\varphi)$ be an object in $\mathcal {MF}_{[a,b]}^{\nabla,\Phi}(\U/W)$. 
	Faltings constructed a functor  $\mathbb{D}_{\Phi}$ by 
	\[\mathbb{D}_{\Phi}(M)=\mathrm{Hom}(V\otimes_{\kappa_{\Phi}} B^+(\widehat R),D),\]
	where the homomorphisms are $B^+(\widehat R)$-linear and respect filtrations and the $\varphi$-structure.
	The action of $\mathrm{Gal}(\widehat{\overline{R}}/\widehat{R})$ on the tensor product $V\otimes_{\kappa_\Phi} B^+(\widehat{R})$ is defined via the connection on $V$, which commutes with the $\varphi$'s and hence induces an action of $\mathrm{Gal}(\widehat{\overline{R}}/\widehat{R})$ on $\mathbb{D}_{\Phi}(M)$. 
	%\begin{lem} There is a canonical isomorphism \[V\otimes_{\kappa_\Phi} B^+(\widehat R) \rightarrow V \otimes_{\kappa_{\Phi_1}} B^+(\widehat R),\]
	%	which induce isomorphism of $\mathrm{Gal}(\overline{\widehat{R}}/\widehat{R})$-modules
	%	\[\mathbb D_{\Psi}(M)\rightarrow\mathbb D_\Phi(\iota_\Phi M).\]
	%\end{lem}
	Since $V$ is a $p$-power torsion finitely generated $R$-module, $\mathbb{D}_{\Phi}(M)$ is a finite $\Z_p$-module. Faltings shows that the functor $\mathbb{D}_\Phi$ from $\mathcal {MF}_{[a,b]}^{\nabla,\Phi}(\U/W)$ to $\mathrm{Rep}_{\Z_p}^{\mathrm{finite}}\left(\mathrm{Gal}\left(\widehat{\overline{R}}/\widehat{R}\right)\right)$, the category of finite $\Z_p$-representation of $\mathrm{Gal}\left(\widehat{\overline{R}}/\widehat{R}\right)$, 
	is fully faithful and its image is closed under subobjects and quotients.

	\paragraph{\emph{The functor $\mathbb{D}$.}}  Recall that $I$ is the index set of all pair $(\U_i,\Phi_i)$ of small affine open subset $\U_i$ of $\X$ and lifting $\Phi_i$ of the absolute Frobenius on $\mathcal O_\X(\U_i)\otimes_W k$. For each $i\in I$, the functor $\mathbb D_{\Phi_i}$ associates to any Fontaine-Faltings module over $\X$ a compatible system of \'etale sheaves on $\widehat{\U}_{i,K}$ (the generic fiber of $\widehat{\U}_i$). By gluing and using the results in EGA3, one obtains a locally constant sheaf on $X_K$ and a globally defined functor $\mathbb D$. 
	
	In the following, we give a slightly different way to construct the functor $\mathbb D$. Let $J$ be a finite subset of the index set $I$, such that $\{\U_j\}_{j\in J}$ forms a covering of $\X$. Denote $U_j=(\U_j)_K$ and choose $\overline{x}$ a geometric point of $X_K$ contained in $\bigcap\limits_{j\in J} U_j$. 
	
	Let $(V,\nabla,\Fil,\{\varphi_i\}_{i\in I})$ be a Fontaine-Faltings module over $\X$. For each $j\in J$, the functor $\mathbb D_{\Phi_j}$  gives us a finite $\Z_p$-representation of $\pi^\text{\'et}_1(\widehat{U}_j,\overline{x})$. Recall that the functor $\mathbb D_\Phi$does not depends on the choice of $\Phi$, up to a canonical isomorphism. In particular, for all $j_1, j_2\in J$, there is a natural isomorphism of $\pi^\text{\'et}_1(\widehat{U}_{j_1}\cap \widehat{U}_{j_2},s)$-representations 
	\[\mathbb D(V(\U_{j_1}\cap \U_{j_2}),\nabla,\Fil,\varphi_{j_1})\simeq \mathbb D(V(\U_{j_1}\cap \U_{j_2}),\nabla,\Fil,\varphi_{j_2}).\]
	By Theorem~\ref{Mainthm:rep}, all representations $\mathbb D(V(\U_{j}),\nabla,\Fil,\varphi_{j})$'s descend to a $\Z_p$-representations of $\pi^\text{\'et}_1(X_K,\overline{x})$. Up to a canonical isomorphism, this representation does not depend on the choice of $J$ and $s$. This representation is just $\mathbb D(V,\nabla,\Fil,\{\varphi_i\}_{i\in I})$ and we construct the Fontaine-Laffaille-Faltings' $\mathbb D$-functor in this way.
	
	\begin{thm}[Faltings]\label{faltingslocal} 
		The functor $\mathbb D$ induces an equivalence of the category $\mathscr {MF}_{[0,p-2]}^\nabla(\X/W)$ with the full subcategory of finite $\Z_p[[\pi^\text{\'et}_1(X_K)]]$-modules whose objects are dual-crystalline representations. This subcategory is closed under sub-objects and quotients.
	\end{thm}

	\paragraph{\emph{The extra  $W(\mathbb F_{p^f})$-structure}} Suppose $\F_{p^f}\subset k$. Let $(V,\nabla,\Fil,\varphi,\iota)$ be an object in $\mathcal {MF}_{[0,p-2],f}^{\nabla}(X_{n+1}/W_{n+1})$. Since the functor $\mathbb D$ is fully faithful, we get an
	extra  $W(\mathbb F_{p^f})$-structure on $\mathbb D(V,\nabla,\Fil,\varphi)$, via the composition
	\[W(\F_{p^f})\overset{\iota}{\longrightarrow} \mathrm{End}_{\mathcal{MF}} (V,\nabla,\Fil,\varphi) \overset{\sim}{\longrightarrow} \mathrm{End}(\mathbb{D}(V,\nabla,\Fil,\varphi)).\]
	Since $V$ is strictly $p^n$-torsion, the $W_n(\mathbb F_{p^f})$-module $\mathbb{D}(V,\nabla,\Fil,\varphi)$ is free with a linear action of $\pi^\text{\'et}_1(X_K)$. We write this $W_n(\F_{p^f})$-representation as 
	\[\mathbb D(V,\nabla,\Fil,\varphi,\iota).\]

	\subsection{The category of projective representations}\label{section CPR}
	
	\paragraph{\emph{The categories $\mathrm{Rep}_{\mathcal O}(G)$ and $\mathrm{Rep}^{\mathrm{free}}_{\mathcal O}(G)$.}} Let $\mathcal O$ be a commutative topological ring with identity and let $G$ be a topological group. Note that all morphisms of topological groups and all actions of groups are continuous in this section.
	Denote by $\mathrm{Rep}_{\mathcal O}(G)$ the category of all finitely generated $\mathcal O$-modules with an action of $G$ and denote by $\mathrm{Rep}^{\mathrm{free}}_{\mathcal O}(G)$ the subcategory of all free $\mathcal O$-modules of finite rank with an action of $G$. 
	
	\paragraph{\emph{The categories $\mathrm{PRep}_{\mathcal O}(G)$ and $\mathrm{PRep}^{\mathrm{free}}_{\mathcal O}(G)$.}} 
	For a finitely generated $\mathcal O$-module ${\mathbb V}$, we denote by $\mathrm{PGL}_{\mathcal O}({\mathbb V})$ the quotient group $\mathrm{GL({\mathbb V})}/\mathcal O^\times$. If $\rho: G\rightarrow \mathrm{PGL}_{\mathcal O}({\mathbb V})$ is a group morphism, then there exists a group action of $G$ on the quotient set ${\mathbb V}/\mathcal O^\times$ defined by $g([v]):=[\rho(g)v]$ for any $g\in G$ and $v\in \mathbb V$. In this case, we call the pair $({\mathbb V},\rho)$ a \emph{projective $\mathcal O$-representation of $G$}.  A morphism of projective $\mathcal O$-representations from $({\mathbb V}_1,\rho_1)$ to $({\mathbb V}_2,\rho_2)$ is an $\mathcal O$-linear morphism $f:{\mathbb V}_1\rightarrow {\mathbb V}_2$ such that the quotient map from ${\mathbb V_1}/{\mathcal O}^\times$ to ${\mathbb V_2}/{\mathcal O}^\times$ induced by $f$ is a morphism of $G$-sets. Denote $\mathrm{PRep}_{\mathcal O}(G)$ the category of finite projective $\mathcal O$-representations of $G$. Denote by $\mathrm{PRep}^{\mathrm{free}}_{\mathcal O}(G)$ the subcategory with $\mathbb V$ being a free $\mathcal O$-module.

	\subsection{Gluing representations and projective representations}\label{section GRPR}
	
	Let $S$ be an irreducible scheme. We fix a geometric point $s$ of $S$. In this section, ${U}$ is an open subset of $S$ containing $s$.  
	\begin{prop}[SGA$1$ \cite{SGA1}, see also Proposition 5.5.4 in~\cite{Sza09}] \label{surjectivity} The open immersion $U\to S$ induces a surjective morphism of fundamental groups 
		\[\rho^S_{U}:\pi^\text{\'et}_1({U}, {s}) \twoheadrightarrow \pi^\text{\'et}_1(S, {s}).\]
	\end{prop}
	Thus, there is a natural restriction functor $\mathrm{res}$ from the category of $\pi^\text{\'et}_1(S, {s})$-sets to the category of $\pi^\text{\'et}_1({U}, {s})$-sets, which is given by \[\mathrm{res}(\rho)=\rho\circ \rho^S_{U}.\]
    	
	\begin{cor}\label{Cor:compHom} The restriction functor $\mathrm{res}$ is fully faithful.
		%	Let $\Sigma$ and $\Sigma'$ be two $\pi^\text{\'et}_1(S, {s})$-sets.  Then 
		%	\[\mathrm{Hom}_{\pi^\text{\'et}_1(S,s)}(\Sigma,\Sigma')=\mathrm{Hom}_{\pi^\text{\'et}_1({U},s)}(\Sigma,\Sigma').\]
	\end{cor}
	The proof of this corollary directly follows from the surjectivity proved in Proposition \ref{surjectivity} and Lemma $52.4.1$ in \cite[Tag 0BN6]{stacks-project}.
%	\begin{proof} Let $\Sigma$, $\Sigma'$ be two $\pi^\text{\'et}_1(S,s)$-sets.  The restriction induces a nature injection 
%		\[\mathrm{Hom}_{\pi^\text{\'et}_1(S,s)}(\Sigma,\Sigma')\subset \mathrm{Hom}_{\pi^\text{\'et}_1({U},s)}(\Sigma,\Sigma').\]  
%		Let $f$ be any elements in $\mathrm{Hom}_{\pi^\text{\'et}_1({U},s)}(\Sigma',\Sigma)$. For any $\sigma'\in\Sigma$ and any $g\in \pi^\text{\'et}_1(S,s)$, we have 
%		\[g\circ f(\sigma')=\widehat{g}\circ f(\sigma')=f\circ\widehat{g}(\sigma')=f\circ g(\sigma')\]
%		where $\widehat{g}$ is some lifting of $g$ under surjective map $\rho_S^\U$.
%		So we have $f\in \mathrm{Hom}_{\pi^\text{\'et}_1(S,s)}(\Sigma',\Sigma)$. Hence the restriction functor is fully faithful.
%	\end{proof}

	Let $\widetilde{S}$ be a finite \'etale covering of $S$. Then there is a natural action of $\pi^\text{\'et}_1(S,s)$ on the fiber $F_s(\widetilde{S})$. 
	\begin{prop}\label{main:prop} $i).$ The fiber functor $F_s$
		induces an equivalence from the category of finite \'etale covering of $S$ to the category of finite $\pi^\text{\'et}_1(S,{s})$-sets. 
		
		$ii).$ The functor $F_{s}$ is compatible with the restrictions of covering to open set ${U}\subset S$ and restrictions of $\pi^\text{\'et}_1(S, {s})$-sets to $\pi^\text{\'et}_1({U},s)$-sets by $\rho_{U}^S$.
	\end{prop}
	See Proposition~$52.3.10$ in \cite[Tag 0BN6]{stacks-project} for a proof of the first statement. The second one follows the very definition, one can find the proof in 5.1 of~\cite{Mur67}
	
	As a consequence, one has the following result, which should be well-known for experts. %We still give a proof for the reader's convenience.
	\begin{cor}[Rigid]\label{lem:rigid} The restriction functor $(\cdot)\mid_{U}$ from 
		the category of finite \'etale coverings of $S$ to the category of finite \'etale coverings of $U$
		is fully faithful. Suppose that there is an isomorphism $f_{U}: \widetilde{S'}\mid_{U} \rightarrow \widetilde{S}\mid_{U}$ of finite \'etale coverings of ${U}$, for some finite \'etale coverings $\widetilde{S}$ and $\widetilde{S'}$ of $S$. Then there is a unique isomorphism $f_S: \widetilde{S'} \rightarrow \widetilde{S}$ of finite \'etale coverings of $S$,  such that $f_{U}=f_S\mid_{U}$.
	\end{cor}
%\err{	\begin{proof}
%		According to Proposition~\ref{main:prop}, we have the following commutative diagram, with two bijective horizontal maps $F_s$.
%		\begin{equation}\label{diag:3}
%		\xymatrix{
%			\mathrm{Hom}_{S}(\widetilde{S'},\widetilde{S}) \ar[d]^{(\cdot)\mid_{U}} \ar[r]^{F_s}
%			& \mathrm{Hom}_{\pi^\text{\'et}_1(S,s)}(\Sigma',\Sigma) \ar[d]^{\mathrm{res}}\\
%			\mathrm{Hom}_{{U}}(\widetilde{S'}\mid_{U},\widetilde{S}\mid_{U}) \ar[r]^{F_s}
%			& \mathrm{Hom}_{\pi^\text{\'et}_1({U},s)}(\Sigma',\Sigma) \\}
%		\end{equation}
%		The fully faithful of the restriction functor follows from Corollary~\ref{Cor:compHom}. Under a fully faithful functor, a morphism is an isomorphism if and only if its image under this functor is an isomorphism. So the corollary~\ref{lem:rigid} follows.
%	\end{proof}}
	
	In the following, we fix a finite index set $J$ and an   open covering $\{{U}_j\}_{j\in J}$ of $S$ with $s\in \bigcap\limits_j {U}_j$. Then for any $j\in J$, the inclusion map ${U}_j\rightarrow S$ induces a surjective group morphism of fundamental groups
	\[\tau_j:\pi^\text{\'et}_1({U}_j,s)\twoheadrightarrow \pi^\text{\'et}_1(S,s).\]
	Denote  ${U}_{J_1}:={U}_{j_1 j_2\cdots j_r}:={U}_{j_1}\cap {U}_{j_2}\cap\cdots \cap {U}_{j_r}$ for any $J_1=\{j_1,j_2,\cdots,j_r\}\subset J$. Similarly, for any $J_1 \subset J_2 \subset J$, we have a surjective group morphism of fundamental groups
	\[\tau_{J_2}^{J_1}:\pi^\text{\'et}_1({U}_{J_2},s)\twoheadrightarrow \pi^\text{\'et}_1({U}_{J_1},s).\] 
	Now we can view every $\pi^\text{\'et}_1({U}_{J_1},s)$-set as a $\pi^\text{\'et}_1({U}_{J_2},s)$-set through this group morphism. Since we already have the rigidity of finite \'etale coverings, one can use it to glue these local $\pi_1$-sets together.
	
	\begin{thm}\label{Mainthm:rep} Let $(\Sigma_j,\rho_j)$ be a finite $\pi^\text{\'et}_1({U}_j,s)$-set for each $j\in J$. Suppose for each pair $i,j\in J$, there exists an isomorphism of $\pi^\text{\'et}_1({U}_{ij},s)$-sets $\eta_{ij}:\Sigma_i\simeq \Sigma_j$. Then every $\Sigma_j$ descends to a $\pi^\text{\'et}_1(S,s)$-set $(\Sigma_j,\widehat{\rho_j})$ uniquely.
		Moreover, the image of $\rho_j$ equals that of $\widehat{\rho_j}$.
	\end{thm}

\subsection{Comparing representations associated to local Fontaine-Faltings modules underlying isomorphic filtered de Rham sheaves}\label{section LFFMHDRF}
In this section, we compare several representations associated to local Fontaine-Faltings modules underlying isomorphic filtered de Rham sheaves. To do so, we first introduce a local Fontaine-Faltings module, which corresponds to a $W_n(\F_{p^f})$-character of the local fundamental group.
We will then use this character to measure the difference between the associated representations.

Let $R$ be a small affine algebra over $W(k)$ and denote $R_n=R/p^nR$ for all $n\geq 1$. Fix a lifting $\Phi:\widehat{R}\rightarrow \widehat{R}$ of the absolute Frobenius on $R/pR$. Recall that $\kappa_\Phi:\widehat{R}\rightarrow B^+(\widehat{R})$ is the lifting of $B^+(\widehat{R})/F^1B^+(\widehat{R})\simeq \widehat{R}$ with respect to the $\Phi$. Under such a lifting, the Frobenius $\Phi_B$ on $B^+(\widehat{R})$ extends to $\Phi$ on $\widehat{R}$.
	
	\paragraph{\emph{Element $a_{n,r}$.}}
	Let $f$ be an positive integer. For any $r\in \widehat{R}^\times$, we construct a Fontaine-Faltings module of rank $f$ as following. Let 
	\[V=R_n e_0\oplus R_n e_1\oplus \cdots \oplus R_n e_{f-1}\]
	be a free $R_n$-module of rank $f$. The integrable connection $\nabla$ on $V$  is defined by formula
	\[\nabla(e_i)=0,\]
	and the filtration $\Fil$ on $V$ is the trivial one. Applying the tilde functor and twisting by the map $\Phi$, one gets 
	\[\widetilde{V}\otimes_\Phi \widehat{R}=
	R_n \cdot (\widetilde{e}_0\otimes_\Phi 1) 
	\oplus R_n \cdot (\widetilde{e}_1\otimes_\Phi 1) 
	\oplus \cdots 
	\oplus R_n \cdot (\widetilde{e}_{f-1}\otimes_\Phi 1),\]
	where the connection on $\widetilde{V}\otimes_\Phi \widehat{R}$ is determined by
	\[\nabla(\widetilde{e}_i\otimes_\Phi 1)=0.\]
	Denote by $\varphi$ the $R_n$-linear map from $(\widetilde{V}\otimes_\Phi \widehat{R},\nabla)$ to $(V,\nabla)$ 
	\[\varphi(\widetilde{e}_0\otimes_\Phi 1,\widetilde{e}_1\otimes_\Phi 1,\cdots,\widetilde{e}_{f-1}\otimes_\Phi 1)=(e_0,e_1,\cdots,e_{f-1})
	\left(\begin{array}{ccccc}
	0 &  & &  & r^{p^n}\\
	1 & 0 & & & \\
	& 1 & 0 & & \\
	&   & \ddots &  \ddots  & \\
	&   &        & 1 & 0 \\ 
	\end{array}\right).\]
	The $\varphi$ is parallel due to $\mathrm{d}(r^{p^n})\equiv 0\pmod{p^n}$. By lemma~\ref{lem:anotherDiscribtionFM}, the tuple $(V,\nabla,\Fil,\varphi)$ forms a Fontaine-Faltings module. Applying Fontaine-Laffaille-Faltings' functor $\mathbb D_\Phi$, one gets a finite $\Z_p$-representation of $\mathrm{Gal}(\widehat{\overline{R}}/\widehat{R})$, which is a free $\Z/p^n\Z$-module of rank $f$. 
	
	\begin{lem}\label{a_n,r}  Let $n$ and $f$ be two positive integers and let $r$ be an invertible element in $R$. 
		Then there exists an $a_{n,r} \in B^+(\widehat{R})^\times$ such that
		\begin{equation}\label{a_nr}
		\Phi_B^f(a_{n,r})\equiv \kappa_{\Phi}(r)^{p^n}\cdot a_{n,r} \pmod{p^n}.
		\end{equation}
	\end{lem}
	\begin{proof}
		Since $\mathbb{D}_{\Phi}(V,\nabla,\Fil,\varphi)$ is free over $\Z/p^n\Z$ of rank $f$. one can find an element $g$ with order $p^n$. Recall that $\mathbb{D}_{\Phi}(V,\nabla,\Fil,\varphi)$ is the sub-$\Z_p$-module of $\mathrm{Hom}_{B^+(\widehat{R})}(V\otimes_{\kappa_{\Phi}} B^+(\widehat{R}),D)$ consisted by elements respecting the filtration and $\varphi$. In particular, the following diagram commutes
		\begin{equation}
		\xymatrix@C=1.5cm{
			\left(\widetilde{V}\otimes_{\kappa_\Phi} B^+(\widehat{R})\right)\otimes_{\Phi} B^+(\widehat{R}) 
			\ar[r]^(0.6){g\otimes_{\Phi} \mathrm{id}}\ar@{=}[d]
			&D\otimes_{\Phi} B^+(\widehat{R})\ar[dd]^{\simeq }\\ 
			\left(\widetilde{V}\otimes_{\Phi} \widehat{R}\right)\otimes_{\kappa_\Phi} B^+(\widehat{R}) 
			\ar[d]^{ \varphi\otimes \mathrm{id}}
			&\\
			V\otimes_{\kappa_\Phi}B^+(\widehat{R})\ar[r]^g
			&D\\ 
		}
		\end{equation}
		Comparing images of $(\widetilde{e}_i\otimes_{\kappa_\Phi} 1)\otimes_\Phi 1$ under the diagram, we have 
		\[\Phi(g(v_{i}\otimes_{\kappa_\Phi} 1))=g(v_{i+1}\otimes_{\kappa_\Phi} 1), \quad \text{ for all } 0\leq i\leq f-2;\]
		and 
		\[\Phi(g(v_{f-1}\otimes_{\kappa_\Phi} 1))=\kappa_\Phi(r)^{p^n}\cdot g(v_{0}\otimes_{\kappa_\Phi} 1).\]
		So we have 
		\[\Phi^f(g(v_{0}\otimes_{\kappa_\Phi} 1))=\kappa_\Phi(r)^{p^n}\cdot g(v_{0}\otimes_{\kappa_\Phi} 1).\]
		Since the image of $g$ is $p^n$-torsion, $\mathrm{Im}(g)$ is contained in $D[p^n]=\frac{1}{p^n}B^+(\widehat{R})/B^+(\widehat{R})$, the $p^n$-torsion part of $D$. Choose a lifting $a_{n,r}$ of $g(e_0\otimes_{\kappa_\Phi} 1)$ under the surjective map $B^+(\widehat{R}) \overset{\frac{1}{p^n}}{\longrightarrow}  D[p^n]$.
		Then the equation~(\ref{a_nr}) follows. Similarly, one can define $a_{n,r^{-1}}$ for $r^{-1}$. By equation~(\ref{a_nr}), we have
		\[\Phi^f(a_{n,r}\cdot a_{n,r^{-1}})=a_{n,r}\cdot a_{n,r^{-1}}.\]
		Thus $a_{n,r}\cdot a_{n,r^{-1}}\in W(\F_{p^f})$. Since both $a_{n,r}$ and $a_{n,r^{-1}}$ are not divided by $p$ (by the choice of $g$), we know that $a_{n,r}\cdot a_{n,r^{-1}}\in W(\F_{p^f})^\times$. The invertibility of $a_{n,r}$ follows.
	\end{proof}
	
	\paragraph{\emph{Comparing representations.}}
	Let $n$ and $f$ be two positive integers.  For all $0\leq i\leq f$, let $(V_i,\nabla_i,\Fil_i)$ be filtered de Rham $R_n=R/p^nR$-modules of level $a$ ($a\leq p-1$). We write $V=\bigoplus\limits_{i=0}^{f-1} V_i$, $\nabla=\bigoplus\limits_{i=0}^{f-1} \nabla_i$ and $\Fil=\bigoplus\limits_{i=0}^{f-1} \Fil_i$ for short. Let
	\begin{equation}
	\begin{split}
	\varphi_i: &\  C^{-1}\circ \overline{\mathrm{Gr}}(V_i,\nabla_i,\Fil_i)\simeq (V_{i+1},\nabla_{i+1}), \quad 0\leq i\leq f-2\\
	\varphi_{f-1}: &\  C^{-1}\circ \overline{\mathrm{Gr}}(V_{f-1},\nabla_{f-1},\Fil_{f-1})\simeq (V_{0},\nabla_{0})\\
	\end{split}
	\end{equation} 
	be isomorphisms of de Rham $R$-modules. Let $r$ be an element in $R^\times$. Since $\mathrm{d}(r^{p^n})=0\pmod{p^n}$, the map $r^{p^n}\varphi_{f-1}$ is also an isomorphism of de Rham $R_n$-modules. Thus 
	\[M=(V,\nabla,\Fil,\varphi) \text{ and } M'=(V,\nabla,\Fil,\varphi')\]
	are Fontaine-Faltings modules over $R_n$, where  $\varphi=\sum\limits_{i=0}^{f-1}\varphi_i$ and $\varphi'=\sum\limits_{i=0}^{f-2}\varphi_i+r^{p^n}\varphi_{f-1}$. 
	
	\begin{prop} $i)$. There are $W_n(\F_{p^f})$-module structures on $\mathbb D_{\Phi}(M)$ and $\mathbb D_{\Phi}(M')$. And the actions of $\mathrm{Gal}(\widehat{\overline{R}}/\widehat{R})$ are semi-linear.\\	
		$ii)$. The multiplication of $a_{n,r}$ on  $\Hom_{B^+(\widehat{R})}\left(V \otimes_{\kappa_\Phi} B^+(\widehat{R}),D\right)$
		induces a $W_n(\F_{p^f})$-linear map between these two submodules 
		\[\mathbb D_{\Phi}(M)\overset{\sim}{\longrightarrow} \mathbb D_{\Phi}(M').\]
	\end{prop}
	\begin{proof} $i)$. Let $g : \bigoplus\limits_{i=0}^{f-1} V_i \otimes_{\kappa_\Phi} B^+(\widehat{R})\rightarrow D$ be an element in $\mathbb D_{\Phi}(M)$. For any $a\in W_n(\F_{p^f})$, denote
    \begin{equation}\label{action of W_n(F_p^f)}
            a\star g:=\sum_{i=0}^{f-1} \sigma^{i}(a) g_i
    \end{equation} 
		where $g_i$ is the restriction of $g$ on the $i$-th component $V_i \otimes_{\kappa_{\Phi}} B^+(\widehat{R})$. One checks that $a\star g$ is also contained in $\mathbb D_{\Phi}(M)$. Thus $\star$ defines a $W_n(\F_{p^f})$-module structure on $\mathbb D_{\Phi}(M)$.  Let $\delta$ be an element in $\mathrm{Gal}(\widehat{\overline{R}}/\widehat{R})$. Then 
		\begin{equation}
		\begin{split}
		\delta(a\star g) & = \delta\circ\left(\sum_{i=0}^{f-1} \sigma^{i}(a) g_i\right)\circ \delta^{-1}\\
		&=\sum_{i=0}^{f-1} \sigma^{i}(\delta(a)) \delta\circ g_i \circ \delta^{-1}\\ 
		&=\delta(a)\star\delta(g)
		\end{split}
		\end{equation}
		In this way, $\mathbb D_{\Phi}(M)$ forms a $W_n(\F_{p^f})$-module with a continuous semi-linear action of $\pi^\text{\'et}_1(U_K)$. For the $W_n(\F_{p^f})$-module structure on $\mathbb D_{\Phi}(M')$, the action of $W_n(\F_{p^f})$ on $\mathbb D_{\Phi}(M')$  is defined in the same manner as in (\ref{action of W_n(F_p^f)}).	\\	
		$ii)$. Recall that $\mathbb{D}_\Phi(M)$ (resp. $\mathbb{D}_\Phi(M')$) is defined to be the set of all morphisms in $\Hom_{B^+(\widehat{R})}\left(V \otimes_{\kappa_\Phi} B^+(\widehat{R}),D\right)$ compatible with the filtration and $\varphi$ (resp. $\varphi'$). Since $\mathbb{D}_\Phi(M)$ and $\mathbb{D}_\Phi(M')$ have the same rank and multiplication by $a_{n,r}$ map on $\Hom_{B^+(\widehat{R})}\left(V \otimes_{\kappa_\Phi} B^+(\widehat{R}),D\right)$ is injective, we only need to show that $a_{n,r}\cdot f\in \mathbb D_{\Phi}(M')$ for all $f\in \mathbb D_{\Phi}(M)$. Suppose $f:V\otimes_{\kappa_\Phi} B^+(\widehat{R})\rightarrow D$ is an element in $\mathbb D_{\Phi}(M)$, which means that $f$ satisfies the following two conditions:
		
		1). $f$ is strict for the filtrations. i.e. 
		\[\sum_{\ell_1+\ell_2=\ell}\Fil^{\ell_1} V\otimes_{\kappa_\Phi}   \Fil^{\ell_2} B^+(\widehat{R})=f^{-1}(\Fil^{\ell} D).\]
		
		2).$f \otimes _\Phi \mathrm{id} =f\circ (\varphi\otimes_{\kappa_\Phi} \mathrm{id})$. i.e. the following diagram commutes
		\begin{equation}
		\xymatrix@C=1.5cm{
			\left(\widetilde{V}\otimes_{\kappa_\Phi} B^+(\widehat{R})\right)\otimes_{\Phi} B^+(\widehat{R}) 
			\ar[r]^(0.6){ f\otimes_{\Phi} \mathrm{id}}\ar@{=}[d]
			& D\otimes_{\Phi} B^+(\widehat{R})\ar[dd]^{\simeq }\\ 
			\left(\widetilde{V}\otimes_{\Phi} \widehat{R}\right)\otimes_{\kappa_\Phi} B^+(\widehat{R}) 
			\ar[d]^{ \varphi_{n,r}\otimes \mathrm{id}}
			&\\
			V\otimes_{\kappa_\Phi}B^+(\widehat{R})\ar[r]^f
			&D\\ 
		}
		\end{equation}
		Since $a_{n,r}\in B^+(\widehat{R})^\times\subset\Fil^0 B^+(\widehat{R})\setminus \Fil^1 B^+(\widehat{R})$,  we have $a_{n,r}\cdot \Fil^\ell D=\Fil^\ell D$, and thus 
		\[\sum_{\ell_1+\ell_2=\ell}\Fil^{\ell_1} V\otimes_{\kappa_\Phi}   \Fil^{\ell_2} B^+(\widehat{R})=f^{-1}(\Fil^{\ell} D)=(a_{n,r}\cdot f)^{-1}(\Fil^{\ell} D).\]
		Simultaneously, we have
		\begin{equation*}
		\begin{split}
		(a_{n,r}\cdot f)\otimes _\Phi \mathrm{id}
		=&
		f\otimes_\Phi \Phi(a_{n,r})\cdot \mathrm{id}\\
		=& f\otimes_\Phi a_{n,r}\cdot \kappa_\Phi(r)^{p^n}\cdot \mathrm{id}=( a_{n,r}\cdot \kappa_\Phi(r)^{p^n})\cdot \Big(f\otimes_\Phi\mathrm{id}\Big) \\
		=& a_{n,r}\cdot \kappa_\Phi(r)^{p^n}\cdot\Big(f\circ (\varphi\otimes_{\kappa_\Phi} \mathrm{id})\Big) \\
		=& (a_{n,r}\cdot f)\circ (r^{p^n}\varphi\otimes_{\kappa_\Phi} \mathrm{id})\\
		\end{split}
		\end{equation*}
		So by definition $a_{n,r}\cdot f\in \mathbb D_{\Phi}(M')$.
	\end{proof}
	\begin{cor}\label{cor_compare_proj} Suppose that $\F_{p^f}\subset k$. 
		The map from $\mathbb D_{\Phi}(M)$ to $\mathbb D_{\Phi}(M')$ is an isomorphism of projective $W_n(\F_{p^f})$-representations of $\mathrm{Gal}(\overline{\widehat{R}}/\widehat{R})$. In particular, we have an bijection of  $\mathrm{Gal}(\overline{\widehat{R}}/\widehat{R})$-sets
		\[ \mathbb D_{\Phi}(M)/W_n(\F_{p^f})^\times\rightarrow  \mathbb D_{\Phi}(M')/W_n(\F_{p^f})^\times.\]
	\end{cor}

	\subsection{The functor $\mathbb D^P$}\label{section FDP}
	In this section, we   assume $f$ to be a positive integer with $\mathbb F_{p^f}\subset k$. Let $\{\U_j\}_{j\in J}$ be a finite small affine open covering of $\X$. Let ${U}_j=(\U_j)_K$. For every $j\in J$, fix $\Phi_j$ as a lifting of the absolute Frobenius on $\U_j\otimes_W k$. Fix $\overline{x}$ as a geometric point in ${U}_J=\bigcap\limits_{j\in J} U_j$ and fix $j_0$ an element in $J$. 
	
	Let $(V,\nabla,\Fil,\varphi,\iota)$ be a Fontaine-Faltings module over $X_n$ with an endomorphism structure of $W(\F_{p^f})$ whose Hodge-Tate weights lie in $[0,p-2]$. Locally, Applying Fontaine-Laffaille-Faltings' functor $\mathbb D_{\Phi_j}$, one gets a finite $W_n(\mathbb F_{p^f})$-representation $\varrho_j$ of $\pi^\text{\'et}_1(U_j,\overline{x})$. Faltings shows that there is an isomorphism $\varrho_{j_1}\simeq \varrho_{j_2}$ of $\Z/p^n\Z$-representations of $\pi^\text{\'et}_1(U_{j_1j_2},\overline{x})$. By Lan-Sheng-Zuo~\cite{LSZ13a}, this isomorphism is $W_n(\mathbb F_{p^f})$-linear. By Theorem~\ref{Mainthm:rep}, these $\varrho_j$'s uniquely descend to a  $W_n(\mathbb F_{p^f})$-representation of $\pi^\text{\'et}_1(X_K,\overline{x})$. Thus one reconstructs the $W_n(\mathbb F_{p^f})$-representation $\mathbb D(V,\nabla,\Fil,\varphi,\iota)$ in this way. 
	
	Now we construct functor $\mathbb D^P$ for twisted Fontaine-Faltings modules, in a similar way. 
	Let $(V_i,\nabla_i,\Fil_i,\varphi_i)_{0\leq i<f}\in\mathcal {TMF}_{[0,p-2],f}^{\nabla}(X_{n+1}/W_{n+1})$ be an $L_n$-twisted Fontaine-Faltings module over $X_n$  with endomorphism structure of $W_n(\mathbb F_{p^f})$ whose Hodge-Tate weights lie in $[0,p-2]$. For each $j\in J$, choosing a trivialization $M(\tau_j)$ and applying Fontaine-Laffaille-Faltings' functor $\mathbb D_{\Phi_j}$, we get a $W_n(\mathbb F_{p^f})$-module together with a linear action of $\pi^\text{\'et}_1(U_j,\overline{x})$. Denote its projectification by $\varrho_j$. By Corollary~\ref{cor_compare_proj}, there is an isomorphism  $\varrho_{j_1}\simeq \varrho_{j_2}$ as projective $W_n(\mathbb F_{p^f})$-representations of $\pi^\text{\'et}_1(U_{j_1j_2},\overline{x})$. In what follows, we will show that these $\varrho_j$'s uniquely descend to a projective $W_n(\mathbb F_{p^f})$-representation of $\pi^\text{\'et}_1(X_K,\overline{x})$ by using Theorem~\ref{Mainthm:rep}. 
	
	In order to use Theorem~\ref{Mainthm:rep}, set $\Sigma_j$ to be the quotient  $\pi^\text{\'et}_1(U_j,\overline{x})$-set $$\mathbb{D}_{\Phi_j}(M(\tau_j))/W_n(\mathbb F_{p^f})^\times.$$ Obviously 
	the kernel of the canonical group morphism 
	\[\mathrm{GL}\big(\mathbb{ D}_{\Phi_{j}}(M(\tau_j))\big)\rightarrow \mathrm{Aut}(\Sigma_j)\]
	is just $W_n(\mathbb F_{p^f})^\times$, we identify the image of this morphism with 
	\[\mathrm{PGL}\big(\mathbb{D}_{\Phi_{j}}(M(\tau_j))\big)=\mathrm{GL}\big(\mathbb{D}_{\Phi_{j}}(M(\tau_j))\big)/W_n(\mathbb F_{p^f})^\times.\]
	Let's denote by $\rho_j$ the composition of $\varrho_j$ and $\mathrm{GL}\big(\mathbb{ D}_{\Phi_{j}}(M(\tau_j))\big)\rightarrow \mathrm{Aut}(\Sigma_j)$ for all $j\in J$.
	\begin{equation*}
	\xymatrix{
		\pi^\text{\'et}_1(U_{j_0},\overline{x})
		\ar@{->>}[d] \ar[r]^(0.3){\varrho_{j_0}} \ar[rrd]_(0.3){\rho_{j_0}}|!{[r];[dr]}\hole 
		& \mathrm{GL}\big(\mathbb{ D}_{\Phi_{j_0}}(M(\tau_{j_0}))\big) 
		\ar@{->>}[d] \ar[dr]
		& \\
		\pi^\text{\'et}_1(X_K,\overline{x})\ar@{.>}[r]^(0.3){\widehat{\rho}_{j_0}} \ar@/_15pt/@{.>}[rr]_{\widehat{\rho}_{j_0}}& \mathrm{PGL}\big(\mathbb{D}_{\Phi_{j_0}}(M(\tau_{j_0}))\big)
		\ar@{>->}[r]
		& \mathrm{Aut}(\Sigma_{j_0})\\ }
	\end{equation*} 
	By Corollary~\ref{cor_compare_proj}, the restrictions of $(\Sigma_{j_1},\rho_{j_1})$ and $(\Sigma_{j_2},\rho_{j_2})$ on $\pi^\text{\'et}_1(U_{j_1j_2},\overline{x})$ are isomorphic for all $j_1,j_2\in J$. Hence by Theorem~\ref{Mainthm:rep}, the map $\rho_{j_0}$ descends to some $\widehat{\rho}_{j_0}$ and the image of $\widehat{\rho}_{j_0}$ is contained in $\mathrm{PGL}\big(\mathbb{D}_{\Phi_{j_0}}(M(\tau_{j_0}))\big)$. 
	So the projective $W_n(\F_{p^f})$-representation $(\mathbb{D}_{\Phi_{j_0}}(M(\tau_{j_0})),\rho_{j_0})$ of $\pi^\text{\'et}_1(U_{j_0},\overline{x})$ descends to projective representation $(\mathbb{D}_{\Phi_{j_0}}(M(\tau_{j_0})),\widehat\rho_{j_0})$ of $\pi^\text{\'et}_1(X_K,\overline{x})$. Up to a canonical isomorphism, this projective representation does not depends on the choices of the covering $\{\U_{j}\}_{{j}\in J}$, the liftings $\Phi_j$'s and $j_0$. And we denote this projective $W_n(\mathbb F_{p^f})$-representation of $\pi^\text{\'et}_1(X_K,\overline{x})$ by 
	\[\mathbb D^P\Big((V_i,\nabla_i,\Fil_i,\varphi_i)_{0\leq i<f}\Big).\]

	Similarly as Faltings' functor $\mathbb{D}$  in \cite{Fal89}, our construction of the $\mathbb D^P$ functor can also be extended to the logarithmic version.  More precisely, let $\X$ be a smooth and proper scheme over $W$ and let $\X^o$ be the complement of a simple normal crossing divisor $\D\subset \X$ relative to $W$. Similarly, by replacing $X_K$ and $U_{j}$ with $X_K^o=X_K$ and $U_{j}^o$, we construct the functor 
	\begin{equation}
	\mathbb D^P: \mathcal {TMF}_{[0,p-2],f}^{\nabla}(X^o_{n+1}/W_{n+1}) \rightarrow  \mathrm{PRep}_{W_n(\mathbb F_{p^f})}^{\mathrm{free}}(\pi^\text{\'et}_1(X_K^o))
	\end{equation}
	from the category of strict $p^n$-torsion twisted logarithmic Fontaine modules (with pole along $\D\times W_n\subset \X\times W_n$) with endomorphism structure of $W_n(\F_{p^f})$ whose Hodge-Tate weights lie in $[0,p-2]$ to the category of free $W_n(\F_{p^f})$-modules with projective actions of $\pi^\text{\'et}_1(X_K^o)$.
 
	Summarizing this section, we get the following result.
	\begin{thm}\label{ConsFunc:D^P} Let $M$ be a twisted logarithmic Fontaine-Faltings module over $\X$ (with pole along $\D$) with endomorphism structure of $W(\F_{p^f})$. Te $\mathbb D^P$-functor associates to $M$ and its endomorphism structure n a projective representation
		\[\rho : \pi^\text{\'et}_1(X_K^o) \to \mathrm{PGL}(\mathbb{D}^P(M)),\]
		where $X_K^o$ is the generic fiber of $\X^o=\X\setminus \D$.
	\end{thm}

\section{Twisted periodic Higgs-de Rham flows}\label{section (T)PHDF}
In this section, we will recall the definition of periodic Higgs-de Rham flows and generalize it to the twisted version.

\subsection{Higgs-de Rham flow over $X_n\subset X_{n+1}$}\label{section HDF}
Recall~\cite{LSZ13a} that a \emph{Higgs-de Rham flow} over $X_n\subset X_{n+1}$ is a sequence consisting of infinitely many alternating terms of filtered de Rham bundles and Higgs bundles
\[\left\{
(V,\nabla,\Fil)^{(n-1)}_{-1},
(E,\theta)_{0}^{(n)},  
(V,\nabla,\Fil)_{0}^{(n)},
(E,\theta)_{1}^{(n)},  
(V,\nabla,\Fil)_{1}^{(n)},
 \cdots\right\},\]
which are related to each other by the following diagram inductively
\begin{equation*}\tiny
\xymatrix@W=10mm@C=-3mm@R=5mm{
	&&&& (V,\nabla,\Fil)_{0}^{(n)} \ar[dr]^{\mathrm{Gr}} 
	&& (V,\nabla,\Fil)_{1}^{(n)} \ar[dr]^{\mathrm{Gr}} 
	&\\
	&&&  (E,\theta)_{0}^{(n)} \ar[ur]^{\mathcal C^{-1}_n} \ar@{..>}[dd] 
	&& (E,\theta)_{1}^{(n)} \ar[ur]^{\mathcal C^{-1}_n} 
	&& \cdots\\
	(V,\nabla,\Fil)^{(n-1)}_{-1} \ar[dr]^{\mathrm{Gr}}	
	&&&&&&&\\
	&\mathrm{Gr}\left((V,\nabla,\Fil)^{(n-1)}_{-1}\right)\ar[rr]^{\sim}_{\psi}
	&&(E,\theta)_{0}^{(n)}\pmod{p^{n-1}}
	&&&&\\
}
\end{equation*}
where
\begin{itemize}
\item[-]  $(V,\nabla,\Fil)^{(n-1)}_{-1}$ is a filtered de Rham bundle over $X_{n-1}$ of level in $[0,p-2]$;
\item[-] $(E,\theta)_0^{(n)}$ is a lifting of the graded Higgs bundle $\mathrm{Gr}\left((V,\nabla,\Fil)^{(n-1)}_{-1}\right)$ over $X_n$, $(V,\nabla)_0^{(n)}:=C^{-1}_n ((E,\theta)_0^{(n)},(V,\nabla,\Fil)^{(n-1)}_{-1} ,\psi)$ and $\Fil^{(n)}_0$ is a Hodge filtration on $(V,\nabla)_0^{(n)}$ of level in $[0,p-2]$;
%\item[-] $(E,\theta)_1^{(n)}:=\mathrm{Gr}\left((V,\nabla,\Fil)_0^{(n)}\right)$ and     
%$(V,\nabla)_1^{(n)}:=C^{-1}_n \left(
%(E,\theta)_1^{(n)},
%(V,\nabla,\Fil)^{(n-1)}_{0},
%\mathrm{id}
%\right)$, where $(V,\nabla,\Fil)^{(n-1)}_{0}$ is the reduction of $(V,\nabla,\Fil)^{(n)}_{0}$ on $X_{n-1}$. And $\Fil^{(n)}_1$ is a Hodge filtration on $(V,\nabla)_1^{(n)}$;
\item[-] Inductively, for $m\geq1$, $(E,\theta)_m^{(n)}:=\mathrm{Gr}\left((V,\nabla,\Fil)_{m-1}^{(n)}\right)$ and     
$(V,\nabla)_m^{(n)}:=C^{-1}_n \left(
(E,\theta)_{m}^{(n)},
(V,\nabla,\Fil)^{(n-1)}_{m-1},
\mathrm{id}
\right)$. Here $(V,\nabla,\Fil)^{(n-1)}_{m-1}$ is the reduction of $(V,\nabla,\Fil)^{(n)}_{m-1}$ on $X_{n-1}$. And $\Fil^{(n)}_m$ is a Hodge filtration on $(V,\nabla)_m^{(n)}$.
\end{itemize}
\begin{rmk}
In case $n=1$, the data of $(\overline{V},\overline{\nabla},\overline{\Fil})_{-1}^{(n-1)}$ is empty. The Higgs-de Rham flow can be rewritten in the following form
\[\left\{ (E,\theta)_{0}^{(1)},  
(V,\nabla,\Fil)_{0}^{(1)},
(E,\theta)_{1}^{(1)},  
(V,\nabla,\Fil)_{1}^{(1)},
\cdots\right\}.\]
In this way, the diagram becomes
\begin{equation*}\tiny
\xymatrix@W=5mm@C=2mm{
&(V,\nabla,\Fil)_{0}^{(1)} \ar[dr]^{\mathrm{Gr}} 
	&& (V,\nabla,\Fil)_{1}^{(1)} \ar[dr]^{\mathrm{Gr}} 
	&\\
(E,\theta)_{0}^{(1)} \ar[ur]^{\mathcal C^{-1}_1}  
	&& (E,\theta)_{1}^{(1)} \ar[ur]^{\mathcal C^{-1}_1} 
	&& \cdots\\
}
\end{equation*}
\end{rmk}

 In the rest of this section, we will give the definition of twisted periodic Higgs-de Rham flow (section~\ref{section TPHDF}), which generalizes the periodic Higgs-de Rham flow in~\cite{LSZ13a}.

\subsection{Twisted periodic Higgs-de Rham flow and equivalent categories}\label{section TPHDF}
	Let $L_n$ be a line bundle over $X_n$. For all $1\leq \ell <n$, denote $L_\ell=L_n\otimes_{\mathcal O_{X_n}}\mathcal O_{X_\ell}$ the reduction of $L_n$ on $X_\ell$. In this subsection, let $a\leq p-2$ be a positive integer. We will give the definition of $L_n$-twisted Higgs-de Rham flow of level in $[0,a]$.
\subsubsection{Twisted periodic Higgs-de Rham flow over $X_1$.} 
\begin{defi}
	Let $f$ be a positive integer. An \emph{$f$-periodic  $L_1$-twisted Higgs-de Rham flow over $X_1\subset X_2$} of level in $[0,a]$, is a Higgs-de Rham flow 	over $X_1$ 
	\[\left\{ (E,\theta)_{0}^{(1)},  
	(V,\nabla,\Fil)_{0}^{(1)},
	(E,\theta)_{1}^{(1)},  
	(V,\nabla,\Fil)_{1}^{(1)},
	\cdots\right\}\]
together with isomorphisms $\phi^{(1)}_{f+i}:(E,\theta)_{f+i}^{(1)}\otimes (L_1^{p^i},0)\rightarrow (E,\theta)_i^{(1)}$ of Higgs bundles for all $i\geq0$ 
\begin{equation*}
\tiny\xymatrix@W=2cm@C=-13mm{
	&\left(V,\nabla,\Fil\right)_{0}^{(1)} \ar[dr]^{\mathrm{Gr}}
	&&\left(V,\nabla,\Fil\right)_{1}^{(1)}\ar[dr]^{\mathrm{Gr}}
	&&\cdots \ar[dr]^{\mathrm{Gr}}
	&&\left(V,\nabla,\Fil\right)_{f}^{(1)}\ar[dr]^{\mathrm{Gr}}  %\ar@/_20pt/[llllll]_{\mathcal C^{-1}_1(\phi^{(1)}_{f})}
	&&\left(V,\nabla,\Fil\right)_{f+1}^{(1)}\ar[dr]^{\mathrm{Gr}}
	\ar[dr]^{\mathrm{Gr}}%\ar@/_20pt/[llllll]_{\mathcal C^{-1}_1(\phi^{(1)}_{f+1})}
	&&\cdots\\%\ar@/_20pt/[llllll]_{\cdots}\\
	\left(E,\theta\right)_{0}^{(1)}\ar[ur]_{\mathcal C^{-1}_{1}}
	&&\left(E,\theta\right)_{1}^{(1)}\ar[ur]_{\mathcal C^{-1}_{1}}
	&& \cdots \ar[ur]_{\mathcal C^{-1}_{1}}
	&&\left(E,\theta\right)_{f}^{(1)}\ar[ur]_{\mathcal C^{-1}_{1}} \ar@/^20pt/[llllll]^{\phi_f^{(1)}}|(0.33)\hole
	&&\left(E,\theta\right)_{f+1}^{(1)}\ar[ur]_{\mathcal C^{-1}_{1}} \ar@/^20pt/[llllll]^{\phi_{f+1}^{(1)}}|(0.33)\hole
	&& \cdots \ar@/^20pt/[llllll]^{\cdots}\ar[ur]_{\mathcal C^{-1}_{1}}\\} 
\end{equation*} 
And for any $i\geq 0$ the isomorphism
\begin{equation*}
 C^{-1}_1(\phi^{(1)}_{f+i}): (V,\nabla)_{f+i}^{(1)}\otimes (L_1^{p^{i+1}},\nabla_{\mathrm{can}})\rightarrow (V,\nabla)_{i}^{(1)},
\end{equation*} 
		strictly respects filtrations $\Fil_{f+i}^{(1)}$ and $\Fil_{i}^{(1)}$. Those $\phi^{(1)}_{f+i}$'s are relative to each other by formula
\[\phi^{(1)}_{f+i+1}=\mathrm{Gr}\circ C^{-1}_1(\phi^{(1)}_{f+i}).\]
	\end{defi}
Denote the category of all twisted $f$-periodic Higgs-de Rham flow over $X_1$ of level in $[0,a]$ by $\mathcal{HDF}_{a,f}(X_{2}/W_2)$.
	
	\subsubsection{Twisted periodic Higgs-de Rham flow $X_n\subset X_{n+1}$.}
	Let $n\geq2$  be an integer and $f$ be a positive integer. And $L_n$ is a line bundle over $X_n$. Denote by $L_{\ell}$ the reduction of $L_n$ modulo $p^\ell$.
We define the category $\mathcal{THDF}_{a,f}(X_{n+1}/W_{n+1})$ of all $f$-periodic twisted Higgs-de Rham flow over $X_n\subset X_{n+1}$ of level in $[0,a]$ in the following inductive way.

\begin{defi}
An $L_n$-twisted $f$-periodic Higgs-de Rham flow over $X_n\subset X_{n+1}$ is a Higgs-de Rham flow
\[\Big\{(V,\nabla,\Fil)_{n-2}^{(n-1)},(E,\theta)_{n-1}^{(n)},(V,\nabla,\Fil)_{n-1}^{(n)},(E,\theta)_{n}^{(n)},\cdots\Big\}_{/X_n\subset X_{n+1}}\]
which is a lifting of an $L_{1}$-twisted $f$-periodic Higgs-de Rham flow
\[\Big\{(E,\theta)_{0}^{(1)},(V,\nabla,\Fil)_{0}^{(1)},(E,\theta)_{1}^{(1)},(V,\nabla,\Fil)_{1}^{(1)},\cdots;\phi^{(1)}_{\bullet}\Big\}_{/X_{1}\subset X_{2}}\]
It is constructed by the following diagram for $2\leq \ell \leq n$, inductively
\begin{equation*} \tiny
\xymatrix@W=2cm@C=-8mm{
	&&\Big(V,\nabla,\Fil\Big)_{\ell-1}^{(\ell)}
	\ar[dr]^{\mathrm{Gr}}\ar@{.>}[ddd]
	&&\cdots
	\ar[dr]^{\mathrm{Gr}}\ar@{.>}[ddd]
	&&\Big(V,\nabla,\Fil\Big)_{\ell+f-2}^{(\ell)}
	\ar[dr]^{\mathrm{Gr}}\ar@{.>}[ddd]
	\\%%%%%%%%%%%%%%%%%%%%%%%%%%%%%%%%%%%%%%%%%%%%%%%%%%%%%%%%%%%
	&\left(E,\theta\right)_{\ell-1}^{(\ell)}
	\ar[ur]_{\mathcal C^{-1}_n}\ar@{.>}[ddd]_(0.3){\mod p^{\ell-1}}
	&&\left(E,\theta\right)_{\ell}^{(\ell)}
	\ar[ur]_{\mathcal C^{-1}_n}\ar@{.>}[ddd]
	&&\cdots
	\ar[ur]_{\mathcal C^{-1}_n}\ar@{.>}[ddd]
	&&\left(E,\theta\right)_{\ell+f-1}^{(\ell)}  
	\ar@{.>}[ddd] \ar@/^20pt/[llllll]^(0.44){\phi_{\ell+f-1}^{(\ell)}}
	\\%%%%%%%%%%%%%%%%%%%%%%%%%%%%%%%%%%%%%%%%%%%%%%%%%%%%%%%%%%%
	&&&&&&&\\
	%%%%%%%%%%%%%%%%%%%%%%%%%%%%%%%%%%%%%%%%%%%%%%%%%%%%%%%%%%%
	\Big(V,\nabla,\Fil\Big)_{\ell-2}^{(\ell-1)}
	\ar[dr]^{\mathrm{Gr}}
	&&\Big(V,\nabla,\Fil\Big)_{\ell-1}^{(\ell-1)}
	\ar[dr]^{\mathrm{Gr}}
	&&\cdots
	\ar[dr]^{\mathrm{Gr}} 
	&&\Big(V,\nabla,\Fil\Big)_{\ell+f-2}^{(\ell-1)}
	\ar[dr]^{\mathrm{Gr}} %\ar@/_20pt/[llllll]_{\mathcal C^{-1}_{\ell-1}(\phi_{\ell+f-2}^{(\ell-1)})}
	\\%%%%%%%%%%%%%%%%%%%%%%%%%%%%%%%%%%%%%%%%%%%%%%%%%%%%%%%%%%%
	&\quad\left(E,\theta\right)_{\ell-1}^{(\ell-1)}\quad
	\ar[ur]_{\mathcal C^{-1}_{\ell-1}}
	&&\left(E,\theta\right)_{\ell}^{(\ell-1)} 
	\ar[ur]_{\mathcal C^{-1}_{\ell-1}}
	&&\cdots
	\ar[ur]_{\mathcal C^{-1}_{\ell-1}}
	&&\left(E,\theta\right)_{\ell+f-1}^{(\ell-1)} 
	\ar@/^20pt/[llllll]^(0.44){ \phi_{\ell+f-1}^{(\ell-1)}}
	\\}%%%%%%%%%%%%%%%%%%%%%%%%%%%%%%%%%%%%%%%%%%%%%%%%%%%%%%%%%%%
\end{equation*} 
Here 
$\bullet$ $(E,\theta)_{\ell-1}^{(\ell)}/X_\ell$ is a lifting of $(E,\theta)_{\ell-1}^{(\ell-1)}/X_{\ell-1}$, which implies automatically 
	$(V,\nabla)_{\ell-1}^{(\ell)}:=C^{-1}_\ell \left( (E,\theta)_{\ell-1}^{(\ell)}, (V,\nabla,\Fil)^{(\ell-1)}_{\ell-2},
	\mathrm{id}\right)$ is a lifting of $(V,\nabla)_{\ell-1}^{(\ell-1)}$ since $C_\ell^{-1}$ is a lifting of $C_{\ell-1}^{-1}$.\\
$\bullet$ $\Fil_{\ell-1}^{(\ell)}\subset (V,\nabla)_{\ell-1}^{(\ell)}$ is a lifting of the Hodge filtration $\Fil_{\ell-1}^{(\ell-1)}\subset (V,\nabla)_{\ell-1}^{(\ell-1)}$, which implies that $(E,\theta)_{\ell}^{(\ell)}=\mathrm{Gr}\left((V,\nabla,\Fil)_{\ell-1}^{(\ell)}\right)/X_\ell$ is a lifting of $(E,\theta)_{\ell}^{(\ell-1)}/X_{\ell-1}$ and
	$(V,\nabla)_{\ell}^{(\ell)}:=C^{-1}_\ell \left( (E,\theta)_\ell^{(\ell)}, (V,\nabla,\Fil)^{(\ell-1)}_{\ell-1}, \mathrm{\mathrm{id}} \right)$.\\
$\bullet$  Repeating the process above, one gets the data $\Fil_i^{(\ell)}$, $(E,\theta)_{i+1}^{(\ell)}$ and $(V,\nabla)_{i+1}^{(\ell)}$ for all $i\geq \ell$.\\
$\bullet$  Finally, for all $i\geq \ell-1$,  $\phi^{(\ell)}_{i+f}:(E,\theta)_{i+f}^{(\ell)}\otimes (L_\ell^{p^i},0) \rightarrow (E,\theta)_{i}^{(\ell)} $ is a lifting of $\phi^{(\ell-1)}_{i+f}$. And these morphisms are related to each other by formula $\phi_{i+f+1}^{(\ell)}=\mathrm{Gr}\circ C^{-1}_\ell(\phi_{i+f}^{(\ell)})$.
Denote the twisted periodic Higgs-de Rham flow by 
\[\Big\{(V,\nabla,\Fil)_{n-2}^{(n-1)},(E,\theta)_{n-1}^{(n)},(V,\nabla,\Fil)_{n-1}^{(n)},(E,\theta)_{n}^{(n)},\cdots;\phi^{(n)}_{\bullet}\Big\}_{/X_n\subset X_{n+1}}\]

The category of all periodic twisted Higgs-de Rham flow over $X_n\subset X_{n+1}$  of level in $[0,a]$ is denoted by $\mathcal{THDF}_{a,f}(X_{n+1}/W_{n+1})$.	
\end{defi}
\begin{rmk} For the trivial line bundle $L_n$, the definition above is equivalent to the original definition of periodic Higgs-de Rham flow in~\cite{LSZ13a} by using the identification $\phi: (E,\theta)_0=(E,\theta)_f$. 
\end{rmk}
 
Note that we can also define the logarithmic version of the twisted periodic Higgs-de Rham flow since we already have the log version of inverse Cartier transform. $\X$ is a smooth proper scheme over $W$ and $\X^o$ is the complement of a simple normal crossing divisor $\D\subset \X$ relative to $W$. Similarly, one constructs the category $\mathcal{THDF}_{a,f}(X^o_{n+1}/W_{n+1})$ of twisted $f$-periodic logarithmic Higgs-de Rham flows (with pole along $\D\times W_n\subset \X\times W_n$) over $\X\times W_n$ whose nilpotent exponents are $\leq p-2$ .

\subsubsection{Equivalence of categories.} We establish an equivalence of categories between $\mathcal{THDF}_{a,f}(X_{n+1}/W_{n+1})$ and $\mathcal {TMF}_{[0,a],f}^{\nabla}(X_{n+1}/W_{n+1})$.

\begin{thm}\label{equiv:TFF&THDF}
	Let $a \leq p-1$ be a natural number and $f$ be an positive integer. Then there is an equivalence of categories between $\mathcal{THDF}_{a,f}(X_{n+1}/W_{n+1})$ and $\mathcal {TMF}_{[0,a],f}^{\nabla}(X_{n+1}/W_{n+1})$.
\end{thm}

\begin{proof} Let 
	\[\mathscr E=\Big\{(V,\nabla,\Fil)_{n-2}^{(n-1)},(E,\theta)_{n-1}^{(n)},(V,\nabla,\Fil)_{n-1}^{(n)},(E,\theta)_{n}^{(n)},\cdots;\phi^{(n)}_{\bullet}\Big\}_{/X_n\subset X_{n+1}}\]
	 be an $f$-periodic $L_n$-twisted Higgs-de Rham flow over $X_n$ with level in $[0,a]$. Taking out $f$ terms of filtered de Rham bundles
	\[(V,\nabla,\Fil)_0^{(n)}, (V,\nabla,\Fil)_1^{(n)},\cdots,(V,\nabla,\Fil)_{f-1}^{(n)}\] 
	together with $f-1$ terms of identities maps
	\[\varphi_i: C^{-1}_n\circ \mathrm{Gr}\left( (V,\nabla,\Fil)_i^{(n)} \right)=(V,\nabla)_{i+1}^{(n)}, \quad i=0,1,\cdots,f-2,\]
	and $\varphi_{f-1}:= C^{-1}_n\left(\phi_{f}^{(n)}\right)$, one gets a tuple
	\[ \mathcal{IC}(\mathscr E):=\left(V_i^{(n)},\nabla_i^{(n)},\Fil_i^{(n)},\varphi_i\right)_{0\leq i<f},\]
	This tuple forms an $L_n$-twisted Fontaine-Faltings module by definition. It gives us the functor $\mathcal{IC}$ from $\mathcal{THDF}_{a,f}(X_{n+1}/W_{n+1})$ to $\mathcal {TMF}_{[0,a],f}^{\nabla}(X_{n+1}/W_{n+1})$.

	Conversely, let $(V_i,\nabla_i,\Fil_i,\varphi_i)_{0\leq i<f}$ be an $L_n$-twisted Fontaine-Faltings module. For $0\leq i\leq f-2$, we identify $(V_{i+1},\nabla_{i+1})$ with $C^{-1}_n\circ\mathrm{Gr}(V_i,\nabla_i,\Fil_i)$ via $\varphi_i$. We construct the corresponding flow by induction on $n$. 
	
	In case $n=1$, we already have following diagram 
	\begin{equation}\tiny
	\xymatrix@W=2cm@C=-13mm{  
		(V,\nabla,\Fil)_0\ar[dr]|{Gr}  
		&&  (V,\nabla,\Fil)_1\ar[dr]|{Gr}
		&&  \cdots \ar[dr]|{Gr}
		&&  (V,\nabla,\Fil)_{f-1} \ar[dr]|{Gr}	
		&&  (V,\nabla)_f \ar@/_15pt/[llllllll]|{\varphi_{f-1}}\\
		& (E,\theta)_1 \ar[ur]|{C_1^{-1}}  
		&&\cdots \ar[ur]|{C_1^{-1}}  
		&& (E,\theta)_{f-1}  \ar[ur]|{C_1^{-1}}  
		&& (E,\theta)_{f}  \ar[ur]|{C_1^{-1}}  
	}
	\end{equation}
	Denote $(E,\theta)_0=(E,\theta)_f\otimes(L_1,0)$. Then 
	\[C_1^{-1}(E_0,\theta_0)\simeq(V_f,\nabla_f)\otimes(L_1^p,\nabla_{\mathrm{can}})\simeq (V_0,\nabla_0).\]
	By this isomorphism, we identify $(V_0,\nabla_0)$ with $C_1^{-1}(E_0,\theta_0)$. Under this isomorphism, the Hodge filtration $\Fil_0$ induces a Hodge filtration $\Fil_f$ on $(V_f,\nabla_f)$. Take Grading and denote 
	\[(E_{f+1},\theta_{f+1}):=\mathrm{Gr}(V_f,\nabla_f,\Fil_f).\]
	Inductively, for $i>f$, we denote $(V_i,\nabla_i)=C^{-1}_1(E_i,\theta_i)$. By the isomorphism
	\[\left(C^{-1}_1\circ \mathrm{Gr}\right)^{i-f}(\varphi_{f-1}):(V_i,\nabla_i)\otimes (L^{p^{i+1-f}},\nabla_{\mathrm{can}})\rightarrow (V_{i-f},\nabla_{i-f}),\]
	the Hodge filtration $\Fil_{i-f}$ induces a Hodge filtration $\Fil_i$ on $(V_i,\nabla_i)$. Denote 
	\[(E_{i+1},\nabla_{i+1}):=\mathrm{Gr}(V_i,\nabla_i,\Fil_i).\]
	Then we extend above diagram into the following twisted periodic Higgs-de Rham flow over $X_1$
	\begin{equation}\tiny
	\xymatrix@W=2cm@C=-13mm{  
		& (V,\nabla,\Fil)_0\ar[dr]|{Gr}  
		&&  (V,\nabla,\Fil)_1\ar[dr]|{Gr}
		&&  \cdots \ar[dr]|{Gr}
		&&  (V,\nabla,\Fil)_{i} \ar[dr]|{Gr}	
		&&  (V,\nabla)_{i+1} \ar[dr]|{Gr}	 \\
		(E,\theta)_0 \ar[ur]|{C_1^{-1}}  
		&& (E,\theta)_1 \ar[ur]|{C_1^{-1}}  
		&&\cdots \ar[ur]|{C_1^{-1}}  
		&& (E,\theta)_{i}  \ar[ur]|{C_1^{-1}}  
		&& (E,\theta)_{i+1}  \ar[ur]|{C_1^{-1}}  
		&& \cdots 
	}
	\end{equation} 
	
	For $n\geq2$, denote  \[(\overline{V}_{-1},\overline{\nabla}_{-1},\overline{\Fil}_{-1}):=(\overline{V}_{f-1}\otimes L_{n-1}^{p^{n-1}},\overline\nabla_{f-1}\otimes \nabla_{can},\overline\Fil_{f-1}\otimes \Fil_{\mathrm{tri}}),\]
	where $(\overline{V}_{f-1},\overline{\nabla}_{f-1},\overline{\Fil}_{f-1})$ is the modulo $p^{n-1}$ reduction of $({V}_{f-1}, {\nabla}_{f-1}, {\Fil}_{f-1})$. Those $\varphi_i$ reduce to a $\varphi$-structure on $(\overline{V}_{i},\overline{\nabla}_{i},\overline{\Fil}_{i})_{-1\leq i<f-1}$. This gives us a $L_{n-1}$-twisted Fontaine-Faltings module over $X_{n-1}$ 
	\[(\overline{V}_{i},\overline{\nabla}_{i},\overline{\Fil}_{i},\overline{\varphi}_i)_{-1\leq i<f-1}\]
	By induction, we have a twisted periodic Higgs-de Rham flow over $X_{n-1}$  
	\begin{equation*} \tiny
	\xymatrix@W=2cm@C=-13mm{
		&\Big(\overline{V},\overline{\nabla},\overline{\Fil}\Big)_{-1}
		\ar[dr]|{\mathrm{Gr}}
		&&\Big(\overline{V},\overline{\nabla},\overline{\Fil}\Big)_{0}
		\ar[dr]|{\mathrm{Gr}}
		&&\cdots
		\ar[dr]|{\mathrm{Gr}} 
		&&\Big(\overline{V},\overline{\nabla},\overline{\Fil}\Big)_{f-1}
		\ar[dr]|{\mathrm{Gr}} %\ar@/_20pt/[llllll]_{\mathcal C^{-1}_{n-1}(\phi_{n+f-2}^{(n-1)})}
		&&\cdots\\
		\left(\overline{E},\overline{\theta}\right)_{-1}
		\ar[ur]|{\mathcal C^{-1}_{n-1}}
		&&\left(\overline{E},\overline{\theta}\right)_{0}
		\ar[ur]|{\mathcal C^{-1}_{n-1}}
		&&\cdots\ar[ur]|{\mathcal C^{-1}_{n-1}}
		&&\left(\overline{E},\overline{\theta}\right)_{f-1}
		\ar[ur]|{\mathcal C^{-1}_{n-1}}
		&&\left(\overline{E},\overline{\theta}\right)_{f}  \ar[ur]|{\mathcal C^{-1}_{n-1}}\\}
	\end{equation*} 
	where the first $f$-terms of filtered de Rham bundles over $X_{n-1}$ are those appeared in the twisted Fontaine-Faltings module over $X_{n-1}$.
	
	Based on this flow over $X_{n-1}$,   we extend the diagram similarly as the $n=1$ case,
	\begin{equation}\tiny
	\xymatrix@W=2cm@C=-13mm{  
		(V,\nabla,\Fil)_0\ar[dr]|{Gr}  
		&&  (V,\nabla,\Fil)_1\ar[dr]|{Gr}
		&&  \cdots \ar[dr]|{Gr}
		&&  (V,\nabla,\Fil)_{f-1} \ar[dr]|{Gr}	
		&&  (V,\nabla)_f \ar@/_15pt/[llllllll]|{\varphi_{f-1}}\\
		& (E,\theta)_1 \ar[ur]|{C_n^{-1}}  
		&&\cdots \ar[ur]|{C_n^{-1}}  
		&& (E,\theta)_{f-1}  \ar[ur]|{C_n^{-1}}  
		&& (E,\theta)_{f}  \ar[ur]|{C_n^{-1}}  
	}
	\end{equation}
	Now it is a twisted periodic Higgs-de Rham flow over $X_n$. Denote this flow by
	\[\mathcal{GR}\big((V_i,\nabla_i,\Fil_i,\varphi_i)_{0\leq i<f}\big).\]
	It is straightforward to verify $\mathcal{GR}\circ \mathcal{IC}\simeq \mathrm{id}$ and $\mathcal{IC} \circ \mathcal{GR} \simeq \mathrm{id}$.
\end{proof} 
 
	This Theorem can be straightforwardly generalized to the logarithmic case and the proof is similar as that of Theorem~\ref{equiv:TFF&THDF}.
\begin{thm}\label{equiv:logTFF&THDF}
Let $\X$ be a smooth proper scheme over $W$ with a simple normal crossing divisor $\D\subset \X$ relative to $W$. Then for each natural number $f\in \N$, there is an equivalence of categories between $\mathcal{THDF}_{a,f}(X^o_{n+1}/W_{n+1})$ and $\mathcal {TMF}_{[0,a],f}^{\nabla}(X^o_{n+1}/W_{n+1})$
\end{thm} 
 
\subsubsection{A sufficient condition for lifting the twisted periodic Higgs-de Rham flow}

We suppose that the field $k$ is finite in this section. Let $\X$ be a smooth proper variety over $W(k)$ and denote $X_n=X\times_{W(k)} W_n(k)$. Let 
$D_1\subset X_1$ be a $W(k)$-liftable normal crossing divisor over $k$. Let $\D\subset \X$ be a lifting of $D_1$.

\begin{prop}\label{Lifting_PHDF} Let $n$ be an positive integer and let $L_{n+1}$ be a line bundle over $X_{n+1}$. Denote by $L_\ell$ the reduction of $L_{n+1}$ on $X_{\ell}$.  Let 
	\[\Big\{(V,\nabla,\Fil)_{n-2}^{(n-1)},(E,\theta)_{n-1}^{(n)},(V,\nabla,\Fil)_{n-1}^{(n)},(E,\theta)_{n}^{(n)},\cdots;\phi^{(n)}_{\bullet}\Big\}_{/X_n\subset X_{n+1}}\] 
	be an $L_{n}$-twisted periodic Higgs-de Rham flow over $X_{n}\subset X_{n+1}$. Suppose 
	\begin{itemize}
		\item[-] Lifting of the graded Higgs bundle $(E,\theta)_i^{(n)}$ is unobstructed. i.e. there exist a logarithmic graded Higgs bundle $(E,\theta)_i^{(n+1)}$ over $X_{n+1}$, whose reduction on $X_{n}$ is isomorphic to $(E,\theta)_i^{(n)}$.
		\item[-] Lifting of the Hodge filtration $\Fil_i^{(n)}$ is unobstructed. i.e. for any lifting $(V,\nabla)_i^{(n+1)}$ of $(V,\nabla)_i^{(n)}$ over $X_{n+1}$, there exists a Hodge filtration $\Fil_i^{(n+1)}$ on $(V,\nabla)_i^{(n+1)}$, whose reduction on $X_{n}$ is $\Fil_i^{(n)}$. 
	\end{itemize}
	Then every twisted periodic Higgs-de Rham flow over $X_{n}$ can be lifted to a twisted periodic Higgs-de Rham flow over $X_{n+1}$.
\end{prop}

\begin{proof} By assumption, we choose $(E',\theta')_{n}^{(n+1)}$ a lifting of $(E',\theta')_n^{(n)}$. Inductively, for all $i\geq n$, we construct $(V',\nabla',\Fil')_i^{(n+1)}$ and $(E',\theta')_{i+1}^{(n+1)}$ as follows. Denote 
	\[(V',\nabla')_i^{(n+1)}=\mathcal C^{-1}_{n+1}\left((E',\theta')_i^{(n+1)}\right),\]
	which is a lifting of $(V,\nabla)_i^{(n)}$. By assumption, we choose a lifting $\Fil_{i}^{'(n+1)}$  on $(V',\nabla')_i^{(n+1)}$ of the Hodge filtration $\Fil_i^{(n)}$  and denote 
	\[(E',\theta')_{i+1}^{(n+1)}=\mathrm{Gr}(V',\nabla',\Fil')_{i}^{(n+1)},\]
	which is a lifting of $(E,\theta)_{i+1}^{(n)}$. 
	
	From the $\phi$-structure of the Higgs-de Rham flow, for all $m\geq0$ there is an isomorphism  \[(E,\theta)_{n}^{(n)}\simeq (E,\theta)_{n+mf}^{(n)}\otimes L_n^{p^{n-1}+p^{n}+\cdots+p^{n+mf-2}}.\] 
	Twisting $(E',\theta')_{n+mf}^{(n+1)}$ with $L_{n+1}^{p^{n-1}+p^{n}+\cdots+p^{n+mf-2}}$, one gets a lifting of $(E,\theta)_{n}^{(n)}$.\\

By deformation theory, the lifting space of $(E,\theta)_{n}^{(n)}$ is a torsor space modeled by $H^1_{Hig}\left(X_1,\mathrm{End}\left((E,\theta)_{n}^{(1)}\right)\right)$. Therefore, the torsor space of lifting $(E,\theta)_{n}^{(n)}$ as a \emph{graded} Higgs bundle should be modeled by a subspace of $H^1_{Hig}$. We give a description of this subspace as follows. For simplicity of notations, we shall replace $(E,\theta)_{n}^{(1)}$ by $(E,\theta)$ in this paragraph. The decomposition of $E = \bigoplus_{p+q=n}E^{p,q}$ induces a decomposition of $\mathrm{End}(E)$:
\[
(\mathrm{End}(E))^{k,-k} := \bigoplus_{p+q=n} (E^{p,q})^{\vee} \otimes E^{p+k,q-k} 
\]  
Furthermore, it also induces a decomposition of the Higgs complex $\mathrm{End}(E,\theta)$. One can prove that the hypercohomology of the following Higgs subcomplex
\begin{equation}\label{gr-lifting space}
\mathbb{H}^1 ((\mathrm{End}(E))^{0,0} \stackrel{\theta^{\mathrm{End}}}{\longrightarrow} (\mathrm{End}(E))^{-1,1} \otimes \Omega^1 \stackrel{\theta^{\mathrm{End}}}{\longrightarrow} \cdots)
\end{equation}
gives the subspace corresponding to the lifting space of graded Higgs bundles.\\

Thus by the finiteness of the torsor space, there are two integers $m>m'\geq 0$, such that 
	\begin{equation}\label{equ_twisting_1}
		(E',\theta')_{n+mf}^{(n+1)} \otimes L_{n+1}^{p^{n-1}+p^{n}+\cdots+p^{n+mf-2}} \simeq (E',\theta')_{n+m'f}^{(n+1)} \otimes L_{n+1}^{p^{n-1}+p^{n}+\cdots+p^{n+m'f-2}}.
	\end{equation}
	
	By twisting suitable power of the line bundle $L_{n+1}$  we may assume $m'=0$. By replacing the period $f$ with $mf$, we may assume $m=1$.
	For integer $i\in[n,n+f-1]$ we denote
	\[(E,\theta,V,\nabla,\Fil)_i^{(n+1)}:=(E',\theta',V',\nabla',\Fil')_i^{(n+1)}.\]
	Then (\ref{equ_twisting_1}) can be rewritten as 
	\begin{equation}\label{equ_twisting_2}
		\phi_{n+f}^{(n+1)}:(E,\theta)_{n+f}^{(n+1)}\otimes L_{n+1}^{p^{n-1}+p^n+p^{n+f-2}}\rightarrow (E,\theta)_n^{(n+1)}
	\end{equation}
	where $(E,\theta)^{(n+1)}_{n+f}=(E',\theta')^{(n+1)}_{n+f}=\mathrm{Gr}\left((V,\nabla,\Fil)_{n+f-1}^{(n+1)}\right)$.
	Inductively, for all $i\geq n+f$, we construct $(V,\nabla,\Fil)_i^{(n+1)}$, $(E,\theta)_{i+1}^{(n+1)}$ and $\phi_{i+1}^{n+1}$ as follows. Denote 
	\[(V,\nabla)_i^{(n+1)}=\mathcal C^{-1}_{n+1}\left((E,\theta)_i^{(n+1)}\right).\]
	According to the isomorphism 
	\begin{equation}\label{equ_twisting_3}
		\mathcal C^{-1}_{n+1}(\phi_{i}^{(n+1)}):(V,\nabla)_i^{(n+1)}\otimes L_{n+1}^{p^{i-f-1}+p^{i-f}+\cdots+p^{i-2}}\rightarrow (V,\nabla)_{i-f}^{(n+1)},
	\end{equation}
	the Hodge filtration $\Fil_{i-f}^{(n+1)}$ on $(V,\nabla)_{i-f}^{(n+1)}$ induces a Hodge filtration $\Fil_{i}^{(n+1)}$ on $(V,\nabla)_{i}^{(n+1)}$. Denote
	\[(E,\theta)_{i+1}^{(n+1)}=\mathrm{Gr}\left((V,\nabla,\Fil)_{i}^{(n+1)}\right).\] 
	Taking the associated graded objects in equation~(\ref{equ_twisting_3}), one gets a lifting of $\phi_{i+1}^{(n+1)}$
	\begin{equation}\label{equ_twisting_4}
		\phi_{i+1}^{(n+1)}:(E,\theta)_{i+1}^{(n+1)}\otimes L_{n+1}^{p^{i-f-1}+p^{i-f}+p^{i-1}}\rightarrow (E,\theta)_{i+1-f}^{(n+1)}
	\end{equation}
	and a twisted Higgs-de Rham flow over $X_{n+1}\subset X_{n+2}$ 
	\[\Big\{(V,\nabla,\Fil)_{n-1}^{(n)},(E,\theta)_{n}^{(n+1)},(V,\nabla,\Fil)_{n}^{(n+1)},(E,\theta)_{n+1}^{(n+1)},\cdots;\phi^{(n+1)}_{\bullet}\Big\}_{/X_{n+1}\subset X_{n+2}}\] 
	which lifts the given twisted periodic flow over $X_n\subset X_{n+1}$. 
\end{proof}

\begin{rmk}
	In the proof, we see that one needs to enlarge the period for lifting the twisted periodic Higgs-de Rham flow. 
\end{rmk}

	\subsection{Twisted Higgs-de Rham self map on moduli schemes of semi-stable Higgs bundle with trivial discriminant}\label{section CTLBSSHBTD}
	Let $X_1$ be a smooth proper $W_2$-liftable variety over $k$, with $\mathrm{dim} \, X_1=n$. Let $H$ be a polarization of $X_1$. Let $r<p$ be a positive integer and $(E,\theta)_0$ be a semistable graded Higgs bundle over $X_1$ of rank $r$ and with the vanishing discriminant. 
	 \begin{thm}\label{thm:existence_of_HiggsdeRham} There is a Higgs-de Rham flow of Higgs bundles and de Rham bundles over $X_1$ with initial term $(E,\theta)_0$.
	\end{thm}
	 The construction of the Higgs-de Rham flow given by Theorem \ref{thm:existence_of_HiggsdeRham} is made by two steps. \\
	 {\bf Step 1.} there exists a Simpson's graded semistable Hodge filtration $\Fil$ (Theorem A.4 in~\cite{LSZ13a} and Theorem $5.12$ in~\cite{Langer14}), which is the most coarse Griffiths transverse filtration on a semi-stable de Rham module such that the associated graded Higgs sheaf is torsion free and still semi-stable. Denote $(V,\nabla)_0:=C^{-1}_1(E_0,\theta_0)$ and $\Fil_0$  the Simpson's graded semistable Hodge filtration on $(V,\nabla)_0$. Denote $(V,\nabla)_1:=C^{-1}_1(E_1,\theta_1)$ and $\Fil_1$  the Simpson's graded semistable Hodge filtration on $(V,\nabla)_1$. Repeating this process, we construct a Higgs-de Rham flow of torsion free Higgs sheaves and de Rham sheaves over $X_1$ with initial term  $(E,\theta)$
	\begin{equation}\label{HDF_ass(E,Theta)}\tiny
	\xymatrix@W=10mm@C=0mm{
		&(V,\nabla,\Fil)_{0} \ar[dr]^{\mathrm{Gr}} 
		&& \cdots  \ar[dr]^{\mathrm{Gr}} 
		&& (V,\nabla,\Fil)_{f-1} \ar[dr]^{\mathrm{Gr}} 
		&& \cdots
		\\
		(E,\theta)_{0} \ar[ur]^{\mathcal C^{-1}_1}  
		&& (E,\theta)_{1}\ar[ur]^{\mathcal C^{-1}_1} 
		&& \cdots \ar[ur]^{\mathcal C^{-1}_1} 
		&& (E,\theta)_{f}\ar[ur]^{\mathcal C^{-1}_1} \\
	}
	\end{equation}  
	Since the Simpson's graded semistable Hodge filtration is unique, this flow is also uniquely determined by $(E,\theta)_0$.\\
   {\bf Step 2.}  The Higgs sheaves and de Rham sheaves
	appearing in the Higgs-de Rham flow are locally free.  Thanks to the recent paper by A. Langer \cite{Langer19}. The local freeness follows from Theorem 2.1 and Corollary 2.9 in his paper.\\[.2cm]
	The purpose of this subsection is to find a canonical choice of the twisting line bundle $L$ such that this Higgs-de Rham flow is twisted preperiodic.

 Firstly, we want to find a positive integer $f_1$ and a suitable twisting line bundle $L_1$ such that $(E'_{f_1},\theta'_{f_1}):=(E_{f_1},\theta_{f_1})\otimes (L_1,0)$ satisfies the following conditions 
\begin{equation}\label{chernclass}
\begin{split}
\text{(i).} & \quad c_1(E'_{f_1}) = c_1(E_0),\\
\text{(ii).} & \quad  c_2(E'_{f_1}) \cdot [H]^{n-2} = c_2(E_0) \cdot [H]^{n-2}.\\
\end{split}   
\end{equation}
Under these two condition, both $(E,\theta)_0$ and $(E,\theta)_{f_1}$ are contained in the moduli scheme $M^{ss}_{Hig}(X_1/k,r,a_1,a_2)$ constructed by Langer in~\cite{Langer04} classifying all semistable Higgs bundles over $X_1$ with some fixed topological invariants (which will be explained later).
Following~\cite{Langer04}, we introduce $\mathcal{S}'_{X_1/k}(d;r,a_1,a_2,\mu_{max})$ the family of Higgs bundles over $X_1$ such that $(E,\theta)$ is a member of the family, where $E$ is of rank $d$, $\mu_{max}(E,\theta) \leq \mu_{max}$,$a_0(E)=r$,$a_1(E)=a_1$ and $a_2(E) \geq a_2$. Here $\mu_{max}(E,\theta)$ is the slope of the maximal destabilizing sub sheaf of $(E,\theta)$, and $a_i(E)$ are defined by
		\[
		\chi (X_{1,\bar{k}},E(m)) = \Sigma^d_{i=0} a_i(E) {m+d-i \choose d-i}.
		\]
		By the results of Langer, the family $\mathcal{S}'_{X_1/k}(d;r,a_1,a_2,\mu_{max})$ is bounded (see Theorem $4.4$ of \cite{Langer04}). So $M^{ss}_{Hig}(X_1/k,r,a_1,a_2)$ is the moduli scheme which corepresents this family. Note that $a_i(E) = \chi(E|_{\bigcap_{j \leq d-i} H_j})$ where $H_1,\dots,H_d \in |\sO(H)|$ is an $E$-regular sequence (see \cite{HL}). Using Hirzebruch-Riemann-Roch theorem, one finds that $a_1(E)$,$a_2(E)$ will be fixed if $c_1(E)$ and $c_2(E) \cdot [H]^{n-2}$ are fixed.

	\begin{prop}\label{prop:find L1}
		Assume discriminant of $E_0$ (with respect to the polarization $H$) $\Delta(E_0):= (c_2(E_0)-\frac{r-1}{2r}c_1(E_0)^2)\cdot[H]^{n-2}$ equals to zero.
		 Let $f_1$ be the minimal positive integer with $r\mid p^{f_1}-1$, and let $L_1=\det(E_0)^{\frac{1-p^{f_1}}{r}}$. Then the two conditions in (\ref{chernclass}) are satisfied.
	\end{prop}
	\begin{proof}
		Since $c_1(C^{-1}_1(E_0,\theta_0)) = pc_1(E_0)$ and $c_1(L_1) = \frac{1-p^{f_1}}{r}\cdot c_1(E_0)$, we have 
		\[c_1(E'_{f_1}) = rc_1(L_1) + c_1\left((Gr\circ C^{-1}_1)^{f_1}(E_0,\theta_0)\right)=\left(r\cdot\frac{1-p^{f_1}}{r} +p^{f_1}\right)c_1(E_0).\]
		One gets Condition (i). Note that the discriminant $\Delta$ is invariant under twisting line bundles, and $\Delta(C^{-1}_1(E_0,\theta_0)) = p^2\Delta(E_0)$, one gets
		\[\Delta(E'_{f_1}) = \Delta(Gr\circ C^{-1}_1(E_0,\theta_0)) = p^{2f_2}\Delta(E_0) =0.\]
		So we have $c_2(E_0) \cdot [H]^{n-2} = c_1(E_0)^2\cdot [H]^{n-2}$ and $c_2(E'_{f_1}) \cdot [H]^{n-2} = c_1(E'_{f_1})^2\cdot [H]^{n-2}$. Since $c_1(E'_{f_1}) = c_1(E_0)$, we already get Condition (ii).   
	\end{proof}
	
	\begin{cor-defi}\label{def:selfmap}
		There is a self-map $\varphi$ on the set of $k$-points of $M^{ss}_{Hig}(X_1/k,r,a_1,a_2)$ defined by the twisted Higgs-de Rham flow, which sends a Higgs bundle $(E,\theta)_0$ to the Higgs bundle $(E,\theta)_{f_1}\otimes (\det(E_0)^\frac{1-p^{f_1}}{r},0)$. Here  $f_1$ is the minimal positive integer with $r\mid p^{f_1}-1$. 
	\end{cor-defi}

\begin{rmk}
In fact one can show that the self-map is a constructible map, i.e. there is a stratification of the moduli scheme such that there is a Simpson graded semistable Hodge filtration attached to the universal de Rham bundle restricted to each constructible subset. Taking the associated graded objects and twisting suitable line bundles, one gets a morphism from each constructible subset to the moduli scheme itself. For more explicit construction one can see the special case in \autoref{section SMMSHBPMP}.
\end{rmk}

	\begin{prop}\label{Self-map}
		Suppose that discriminant of $E_0$ equals to zero and there exists $f_2$ a positive integer with $\varphi^{f_2}(E_0,\theta_0) \simeq (E_0,\theta_0)$. Then the Higgs-de Rham flow~(\ref{HDF_ass(E,Theta)}) is $\det(E_0)^{\frac{p^f-1}{r}}$-twisted $f$-periodic, where  $f=f_1f_2$.   
	\end{prop}
	\begin{proof}  Inductively, one shows that 
\begin{equation} 
	\varphi^{m} (E_0,\theta_0) =  (E,\theta)_{mf_1} \otimes (\det(E_0)^\frac{1-p^{mf_1}}{r},0). 
\end{equation} 
Since $\varphi^{f_2} (E_0,\theta_0)\simeq (E_0,\theta_0)$, there is an isomorphism of Higgs bundles
		\[\phi_f:(E_f,\theta_f)\otimes (\det(E_0)^{\frac{p^f-1}{r}},0)\rightarrow (E_0,\theta_0).\]
By formula $\phi_i=\big(\mathrm{Gr}\circ\mathcal C^{-1}_1\big)^{i-f}(\phi_f)$ for all $i\geq f$, we construct the twisted $\phi$-structure. Under this  $\phi$-structure the Higgs-de Rham flow is $\det(E_0)^{\frac{p^f-1}{r}}$-twisted $f$-periodic.
\end{proof}
	
	\begin{thm} \label{Main: preperiod}
		A semistable Higgs bundle over $X_1$ with trivial discriminant is preperiodic after twisting. Conversely, a twisted preperiodic Higgs bundle is semistable with a trivial discriminant.   
	\end{thm}
	\begin{proof}
		For a Higgs bundle $(E,\theta)$ in $M^{ss}_{Hig}(X_1/k,r,a_1,a_2)$, we consider the iteration of the self-map $\varphi$. Since $M^{ss}_{Hig}(X_1/k,r,a_1,a_2)$ is of finite type over $k$ and has only finitely many $k$-points, there must exist a pair of integers $(e,f_2)$ such that $\varphi^e(E,\theta) \cong \varphi^{e+f_2}(E,\theta)$. By Proposition \ref{Self-map}, we know that $(E,\theta)$ is preperiodic after twisting.\\
		Conversely, let $(E,\theta)$ be the initial term of a twisted $f$-preperiodic Higgs-de Rham flows. We show that it is semistable. Let $(F,\theta) \subset (E,\theta)$ be a proper sub bundle. Denote $(F^{(1)}_i,\theta^{(1)}_i)$ and $(E^{(1)}_i,\theta^{(1)}_i)$ are the terms appearing in the Higgs-de Rham flows.  By the preperiodicity, there exists a line bundle $L$ and an isomorphism $\phi : (E_e,\theta_e) \cong (E_{e+f},\theta_{e+f}) \otimes (L,0)$. Calculating the slope on both side, one get $\mu(L)=(1-p^f)\mu(E_e)$.
		Iterating $m$ times of this isomorphism $\phi$, one get
		\[\phi^m : (E_e,\theta_e) \cong (E_{e+mf},\theta_{e+mf}) \otimes (L^{1+p^f+\cdots+p^{(m-1)f}},0).\]
		So $\left(\phi^m\right)^{-1} \left(F_{e+mf}\otimes L^{1+p^f+\cdots+p^{m-1}f} \right)$ forms a sub sheaf of $E_e$ of slope
		\[p^{mf}\mu(F_e)+(1+p^f+\cdots+p^{(m-1)f})\mu(L)=p^{mf}\big(\mu(F_e)-\mu(E_e)\big)+\mu(E_e).\]
		So $\mu(F_e)\leq \mu(E_e)$ (otherwise there are subsheaves of $E_e$ with unbounded slopes, but this is impossible). So we have 
		\[\mu(F)=\frac{1}{p^e}\mu(F_e)\leq\frac{1}{p^e}\mu(E_e)=\mu(E).\] 
		This shows that $(E,\theta)$ is semistable. The discriminant equals zero follows from the fact that $\Delta(C^{-1}_1(E,\theta)) = p^2 \Delta(E)$.
\end{proof}

\begin{cor}\label{slop sub THDF}
	Let  $(E,\theta)\supset (F,\theta)$ be the initial terms of a twisted periodic Higgs-de Rham flow and a sub twisted periodic Higgs-de Rham flow. Then 
	\[\mu(F)=\mu(E).\]
\end{cor}

\subsection{Sub-representations and sub periodic Higgs-de Rham flows}\label{section SRSPHdRF}
In this section, we assume $\F_{p^f}$ is contained in $k$. Recall that the functor $\mathbb D^P$ is contravariant and sends quotient object to subobject, i.e. for any sub twisted Fontaine-Faltings module $N\subset M$ with endomorphism structure, the projective representation $\mathbb D^P(M/N)$ is a projective subrepresentation of $\mathbb D^P(M)$. Conversely,  we will show that every projective subrepresentation comes from this way. By the equivalence of the category of twisted Fontaine-Faltings modules and the category of twisted periodic Higgs-de Rham flows, we construct a twisted periodic sub Higgs-de Rham flow for each projective subrepresentation. 

Let $\X$ be a smooth proper $W(k)$-variety. Denote by $X_{n}$ the reduction of $\X$ on $W_n(k)$ . Let $\{\U_i\}_{i\in I}$ be a finite covering of small affine open subsets and we choose a geometric point $x$ in $\bigcap\limits_{i\in I} U_{i,\overline{K}}$.  

\begin{prop}\label{subrep-1} Let $M$ be an object in $\mathcal{TMF} _{[a,b],f} ^{\nabla}(X_{2}/W_{2})$. Suppose we have a projective $\mathbb F_{p^f}$-subrepresentation of $\pi^\text{\'et}_1(X_K)$ ${\mathbb V} \subset \mathbb{D}^P(M)$, then there exists a subobject $N$ of $M$ such that ${\mathbb V}$ equals to $\mathbb{D}^P(M/N)$.
\end{prop}

\begin{proof} Recall that the functor $\mathbb D^P$ is defined by gluing representations of $\Delta_i=\pi^\text{\'et}_1(U_{i,K},x)$ into a projective representation of $\Delta=\pi^\text{\'et}_1(X_K,x)$. Firstly, we show that the projective subrepresentation $\mathbb V$ is actually corresponding to some local subrepresentations. Secondly, since the Fontaine-Laffaille-Faltings' functor $\mathbb D$ is fully faithful, there exist local Fontaine-Faltings modules corresponding to those subrepresentations. Thirdly, we glue those local Fontaine-Faltings modules into a global twisted Fontaine-Faltings module. 

For $i\in I$, we choose a trivialization  $M_i=M(\tau_i)$  of $M$ on $\U_i$, which gives a local Fontaine-Faltings module with endomorphism structure on $\U_i$. By definition of $\mathbb D^P$, those representations $\mathbb{D}_{\U_i}(M_i)$ of $\Delta_i$ are glued into the projective representation $\mathbb{D}^P(M)$. In other words, we have the following commutative diagram of $\Delta_{ij}=\pi^\text{\'et}_1(U_{i,K}\cap U_{j,K},x)$-sets
\begin{equation}
\xymatrix@W=14mm@C=0mm@R=5mm{
  & \mathbb{D}_{\U_i}(M_i)/\mathbb F_{p^f}^\times \ar[dd]^{a_{1,r}}	\\
 \mathbb{D}^P(M)/\mathbb F_{p^f}^\times \ar[ur]\ar[dr] 
 &\\
 & \mathbb{D}_{\U_j}(M_j)/\mathbb F_{p^f}^\times\\
 }
\end{equation}
Here $r$ is the difference of the trivializations of the twisting line bundle on $\U_i$ and $\U_j$. And $a_{1,r}$ is the elements given in Lemma~\ref{a_n,r}.

Assume that $\mathbb V$ is a projective $\mathbb F_{p^f}$-subrepresentation of $\mathbb{D}^P(M)$ of $\pi^\text{\'et}_1(X_K,x)$, i.e.  $\mathbb V/\mathbb F_{p^f}^\times$ is a $\pi^\text{\'et}_1(X_K)$-subset of $\mathbb D^P(M)/\mathbb F_{p^f}^\times$. Then ${\mathbb V}_i$, the image of $\mathbb V$ under the map $\mathbb{D}^P(M) \to \mathbb{D}_{\U_i}(M_i)$, is a projective $\mathbb F_{p^f}$-subrepresentation of $\mathbb{D}_{\U_i}(M_i)$. So we have the following commutative diagram of $\Delta_{ij}$-sets
\begin{equation}\label{diag:5.2}
\xymatrix@W=14mm@C=0mm@R=5mm{
	& \mathbb V_i/\mathbb F_{p^f}^\times \ar@{>->}[rr] \ar[dd]^(0.3){a_{1,r}}|(0.5)\hole
	&& \mathbb{D}_{\U_i}(M_i)/\mathbb F_{p^f}^\times \ar[dd]^(0.3){a_{1,r}}	\\
	\mathbb V/\mathbb F_{p^f}^\times \ar@{>->}[rr] \ar[ur]\ar[dr]
	&& \mathbb{D}^P(M)/\mathbb F_{p^f}^\times \ar[ur]\ar[dr] 
	&\\
	& \mathbb V_j/\mathbb F_{p^f}^\times \ar@{>->}[rr]
	&& \mathbb{D}_{\U_j}(M_j)/\mathbb F_{p^f}^\times\\
}
\end{equation} 
Notice that $\mathbb{D}_{\U_i}(M_i) /\mathbb F_{p^f}^\times$ is the projectification of the $\mathbb F_{p^f}$-representation $\mathbb{D}_{\U_i}(M_i)$ of $\Delta_i$. So $\mathbb V_i\subset \mathbb D_{\U_i}(M_i)$ is actually a $\mathbb F_{p^f}$-subrepresentation of $\Delta_i$. 
	
Since the image of the contravariant functor $\mathbb D_{\U_i}$ is closed under subobjects, there exists $N_i\subset M_i$ as a sub Fontaine-Faltings module with endomorphism structure of $\F_{p^f}$, such that
		\[\mathbb V_i=\mathbb D_{\U_i}(M_i/N_i).\]	
On the overlap $\U_i\cap \U_j$, those two Fontaine-Faltings module $M_i$ and $M_j$ have the same underlying filtered de Rham sheaf. We can twist the $\varphi$-structure of $M_i$ to get $M_j$ by the element $r$. Doing the same twisting on $N_i$, we get a sub-Fontaine-Faltings module $N_i'$ of $M_j$. By the functoriality of $\mathbb{D}$, one has the following commutative diagram
\begin{equation}
\xymatrix{
	\mathbb D(M_i/N_i)  \ar@{>->}[r]\ar@{..>}[d]^{\exists}  & \mathbb D(M_i) \ar@{->>}[r] \ar[d]^{a_{1,r}}  & \mathbb D(N_i) \ar[d]^{a_{1,r}} \\
	\mathbb D(M_j/N_i') \ar@{>->}[r]  & \mathbb D(M_j) \ar@{->>}[r]   & \mathbb D(N_i') \\	
	}
\end{equation}
So we have $\mathbb D(M_j/N_i')= a_{1,r} \mathbb D(M_i/N_i)=a_{1,r}\mathbb V_i$. On the other hand,  one has $\mathbb D(M_j/N_j)=\mathbb V_j=a_{1,r}\mathbb V_i$ by diagram~(\ref{diag:5.2}). Thus one has $\mathbb D(M_j/N_i')=\mathbb D(M_j/N_j)$. Since $\mathbb{D}$ is fully faithful and contravariant,  $N_i'=N_j$. In particular, on the overlap $\U_i\cap \U_j$ the local Fontaine-Faltings modules $N_i$ and $N_j$ have the same underlying subbundle. 
By gluing those local subbundles together, we get a subbundle of the underlying bundle $M$. 
The connection, filtration and the $\varphi$-structure can be restricted locally on this subbundle, so does it globally. And we get the desired sub-Fontaine-Faltings module. 
\end{proof}

Let $\mathscr E$ be a twisted $f$-periodic Higgs-de Rham flow. Denote by $M=\mathcal{IC}(E)$ the Fontaine module with the endomorphism structure corresponding to $\mathscr E$. By the equivalence of the category of twisted Fontaine-Faltings modules and the category of periodic Higgs-de Rham flow, one get the following result.
\begin{cor}\label{subHDF}
Suppose $\mathbb V\subset \mathbb D^P(M)$ is a non-trivial projective $\mathbb F_{p^f}$-subrepresentation. Then there exists a non-trivial sub-twisted periodic Higgs-de Rham flow of $\mathscr E$ which corresponds to $\mathbb{D}^P(M)/\mathbb{V}$.
\end{cor}

After Corollary~\ref{subHDF} we arrive at the main theorem 0.6 stated in the introduction. However we prove a weaker form of Theorem 0.6 in the below. The proof of the stronger form 
will be postponed in the section 5.
\begin{thm}\label{Mainthm} Let $k$ be a finite field of characteristic $p$. Let $\X$ be a smooth proper scheme over $W(k)$ together with a smooth log structure $\D/W(k)$. Assume that there exists a semistable graded logarithmic Higgs bundle 
$(E,\theta)/(\X,\D)_1$  with discriminant $\Delta_H(E)=0$, $\mathrm{rank}(E)<p$ and $(\mathrm{rank}(E),\deg_H(E))=1$. Then there exists a positive integer $f$ and an absolutely irreducible projective $\F_{p^f}$-representation $\rho$ of $\pi^\text{\'et}_1(X^o_{K'})$, where $\X^o=\X\setminus \D$ and $K'=W(k\cdot\F_{p^f})[1/p]$. 	
\end{thm}
\begin{proof} We only show the result for $\D=\emptyset$, as the proof of the general case is similar.  By Theorem~\ref{Main: preperiod}, there is a twisted preperiodic Higgs-de Rham flow with initial term $(E,\theta)$. Removing finitely many terms if necessary, we may assume that it is twisted $f$-periodic, for some positive integer $f$. By using Theorem~\ref{equiv:TFF&THDF} and applying functor $\mathbb D^P$, one gets a $\mathrm{PGL}_{\mathrm{rank}(E)}(\mathbb F_{p^f})$-representation $\rho$ of $\pi^\text{\'et}_1(X^o_{K'})$.  

Since $(\mathrm{rank}(E),\deg_H(E))=1$, the semi-stable bundle $E$ is actually stable. According to Corollary~\ref{slop sub THDF}, there is no non-trivial sub twisted periodic Higgs-de Rham flow. By Corollary~\ref{subHDF}, there is no non-trivial projective subrepresentation of $\rho$, so that $\rho$ is irreducible.
\end{proof}

\begin{rmk} For simplicity, we only consider results on $X_1$. Actually, all results in this section can be extended to the truncated level.
\end{rmk}

\section{Constructing crystalline representations of \'etale fundamental groups of $p$-adic curves via Higgs bundles}\label{section CCREFGpCHB}

As an application of the main theorem (Theorem~\ref{Mainthm}), we construct irreducible $\mathrm{PGL}_2$ crystalline representations of $\pi^\text{\'et}_1$ of the projective line removing $m$ ($m\geq 4$) marked points.
Let $M$ be the moduli space of semistable graded Higgs bundles of rank $2$ degree $1$ over $\P^1/W(k)$, with logarithmic Higgs fields which have $m$ poles $\{x_1,x_2, \dots,x_m\}$ (actually stable, since the rank and degree are coprime to each other).
The main object of this section is to study the self-map $\varphi$ (Corollary-Definition~\ref{def:selfmap}) on $M$. In section~\ref{section DMS}, we decompose $M$ into connected components. In section~\ref{section SMMSHBPMP}, we show that the self-map is rational and dominant on the component of $M$ with maximal dimension. In section~\ref{section EFSMFMP}, we give the explicit formula in case of $m=4$.

\subsection{Connected components of the moduli space $M$}\label{section DMS}
First, let's investigate the geometry of $M$. For any $[(E,\theta)] \in M$, $E \cong \sO(d_2)\oplus \sO(d_1)$ with $d_1 +d_2=1$($d_2<d_1$).
And the graded semi-stable Higgs bundle with nilpotent non-zero Higgs field
\[\theta:\sO(d_1) \longrightarrow  \sO(d_2) \otimes \Omega^1_{\P^1}(m)\]

By the condition $\theta \neq 0$ in $\Hom_{\sO_{\P^1}} (\sO(d_1),\sO(d_2+m-2))$, we have $d_1 \leq d_2 +m-2$. Combining with the assumption $d_1 +d_2=1$($d_2<d_1$), one get $m \geq 3$ and 
\[(d_1,d_2)=(1,0),(2,-1),\cdots, \text{ or } (\left[\frac{m-1}{2}\right],\left[\frac{4-m}{2}\right]),\]
where $\left[\cdot\right]$ is the greatest integer function. Therefore, $M$ admits a decomposition
\[
M= \coprod_{(d_2,d_1)} M(d_2,d_1)
\]
where $M(d_2,d_1)$ is isomorphic to  
\[\P\left(\mathrm{Hom}_{\mathcal O_{\P^1}}\left(\mathcal O(d_1),\mathcal O(d_2)\otimes \Omega_{\P^1}^1(\log \D)
\right)\right)\simeq \P \Big(\mathrm{H}^0(\P^1,\mathcal O(d_2-d_1+m-2))\Big)\]
 (note that in this case two Higgs bundles are isomorphic if the Higgs fields differ by a scalar).
For $m=3,4$, the decomposition is trivial because $(d_2,d_1)=(0,1)$ is the only choice. But for $m \geq 5$, there are more choices. The following table presents the information of $M(d_2,d_1)$:
 \begin{center}
 	\begin{tabular}{|c|c|c|c|c|c|c|c|c}
 		\hline
 		\diagbox{$(d_1,d_2)$}{$M(d_2,d_1)$}{$m$}  &3&4&5&6&7&8&9& $\cdots$ \\
 		\hline
 		$(1,0)$  & $\mathbb P^0$ & $\mathbb P^1$ & $\mathbb P^2$ &\ $\mathbb P^3$ & $\mathbb P^4$ &$\mathbb P^5$ & $\mathbb P^6$  & $\cdots$\\
 		\hline
 		$(2,-1)$ &&& $\mathbb P^0$ & $\mathbb P^1$ & $\mathbb P^2$ & $\mathbb P^3$ & $\mathbb P^4$ & $\cdots$  \\
 		\hline
 		$(3,-2)$ &&&&& $\mathbb P^0$ & $\mathbb P^1$ & $\mathbb P^2$ & $\cdots$ \\
 		\hline
 		$\vdots$ &&&&&&& $\ddots$ & $\ddots$ \\
 	\end{tabular} 	
 \end{center} 

%\[
%\begin{array}{c|c|c}
%\hline
%m & (d_1,d_2) & component \\
%\hline
%3 & (1,0)    & \P^0       \\ 
%\hline
%4 & (1,0)    & \P^1       \\
%\hline
%5 & (1,0)    & \P^2       \\ \cline{2-3}
%& (2,-1)    & \P^0       \\
%\hline
%6 & (1,0)    & \P^3       \\ \cline{2-3}
%& (2,-1)    & \P^1       \\
%\hline
%7 & (1,0)    & \P^4       \\ \cline{2-3}
%& (2,-1)   & \P^2       \\\cline{2-3}
%& (3,-2)   & \P^0       \\
%\hline
%\vdots & \vdots & \vdots
%\end{array}
%\]

\subsection{Self-maps on moduli spaces of Higgs bundles on	$\mathbb P^1$ with marked points}\label{section SMMSHBPMP}
Let $p$ be an odd prime number. Since the rank $r=2$ for any element in $M$, by Corollary-Definition~\ref{def:selfmap} we know that $f_1=1$ and $L_1=\mathcal O_{\P^1}(\frac{1-p}{2})$. In other words, the self-map is given by 
\[\varphi: (E,\theta)\mapsto \left(\mathrm{Gr}\circ C^{-1}_1(E,\theta)\right)\otimes \mathcal O_{\P^1}(\frac{1-p}{2}),\]
where the filtration on $C^{-1}_1(E,\theta)$ is the Simpson's graded semistable Hodge filtration.  
Let's denote $(V, \nabla)=C^{-1}_1(E,\theta)$, which is a rank $2$ degree $p$ stable de Rham bundle over $\P^1$. Using Grothendieck's theorem, one gets $V \cong \sO(l_1) \oplus \sO(l_2)$ with $l_1 + l_2 = p$ (assume $l_1 <l_2$). In this case, the Simpson's graded semistable Hodge filtration is just the natural filtration $(\sO(l_2) \subset V)$.

 Since $(V,\nabla)$ is stable, $\sO(l_2)$ cannot be $\nabla$-invariant, which means the Higgs field
\[
Gr \nabla :  \sO(l_2) \longrightarrow \sO(l_1) \otimes \Omega^1_{\P^1}(m) \cong \sO(l_1+m-2)
\]
is nontrivial. Thus, $l_2 \leq l_1+m-2$. Combining with the fact $l_1 + l_2 = p$ and $\ell_1<\ell_2$, one gets
\[(l_1,l_2) = ({\frac{p-1}{2}},{\frac{p+1}{2}}),({\frac{p-3}{2}},{\frac{p+3}{2}}),\cdots, \text{ or } (\left[\frac{p-m+3}{2}\right],\left[\frac{p+m-2}{2}\right]).\]
For $m \geq 5$, the jumping phenomena appears, i.e. there exists $[(E,\theta)] \in M(d_2,d_1)$
such that the type of $\left(\mathrm{Gr}\circ C^{-1}_{1}(E,\theta)\right)\otimes \mathcal O(\frac{1-p}{2})$ is different from $(d_2,d_1)$.

Next we shall characterize the jumping locus on $M(d_2,d_1)$. Define a $\Z$-valued function $l$ on $M(d_2,d_1)$: for each $[(E,\theta)] \in M(d_2,d_1)$, set $l([(E,\theta)])=l([\theta]) := l_2$.

\begin{lem}
	The function $l$ on $M(d_2,d_1)$ is upper semicontinuous.   
\end{lem}
\begin{proof}
	Define $\sU_n := \{ [\theta] \in \P H^0(\sO(d_2-d_1+m-2))\, |\, l([\theta]) \leq n\}$. One only need to prove that $\sU_n$ is Zariski open in $\P^{d_2-d_1+m-2}$ for all $n \in \Z$.
	Recall the proof of Grothendieck's theorem, for $(V_{\theta},\nabla) := C^{-1}_{1,2}(\sO(d_2) \oplus \sO(d_1),\theta)$ one defines
	\[
	m := \mathrm{min} \{\lambda \in \Z \, | \, H^0(\P^1 , V_{\theta}(\lambda)) \neq 0\}
	\]
	and gets the splitting $V_{\theta} \cong \sO(-m) \oplus \sO(p+m)$ ($p+m \leq -m$). Therefore, $l([\theta])=-m$. Since $[\theta] \in \sU_n$, one gets $-m \leq n$. But this means $-n-1 < -n \leq m$. Thus $H^0(\P^1 , V_{\theta}(-n-1))=0$.\\
	By the semicontinuity of the rank of the direct image sheaf, we know that 
	$H^0(\P^1 , V_{\theta'}(-n-1))=0$ for $\theta'$ in a neighborhood of $\theta$. This means $l([\theta']) \leq n$ in a neighborhood.
	Therefore, $\sU_n$ is Zariski open for each $n \in \Z$.
\end{proof}

\subsubsection*{Construction of the universal Simpson graded semistable Hodge filtration and the rational self-map}
Now we consider the first component of moduli scheme $M(1,0)$ and the universal Higgs bundle $(E^u,\theta^u))$  on $\mathbb{P}^1 \times \sU_{\frac{p+1}{2}}$:
\[
E^u= (\sO_{\mathbb{P}^1} \oplus \sO_{\mathbb{P}^1}(1)) \otimes \sO_{\mathbb{P}^1 \times \sU_{\frac{p+1}{2}}}, \,\,\,\theta^u|_{\mathbb{P}^1 \times \{x\}}=\theta_x \in \mathrm{Hom}_{\mathcal O_{\P^1}}\left(\mathcal O(1),\mathcal O \otimes \Omega_{\P^1}^1(\log \D)
\right) 
\]
for $x \in \sU_{\frac{p+1}{2}}$. Applying the inverse Cartier functor, we get the universal de Rham bundle $(V^u,\nabla^u)$ on $\mathbb{P}^1 \times M_{dR}({\frac{p-1}{2}},{\frac{p+1}{2}})$. Here $M_{dR}({\frac{p-1}{2}},{\frac{p+1}{2}})$ is the corresponded component of the moduli space of semistable de Rham bundles with rank 2 degree $p$, i.e. $[(V,\nabla)]$ with $V \cong \sO({\frac{p-1}{2}})\oplus \sO({\frac{p+1}{2}})$. For each $s \in M_{dR}({\frac{p-1}{2}},{\frac{p+1}{2}})$, we know that $\sO({\frac{p+1}{2}}) \hookrightarrow V_s$ gives a Hodge filtration. In order to find a Hodge filtration on $(V^u,\nabla^u)$, we shall construct a subsheaf $\mF \subset V^u$ such that $\mF_s \cong \sO({\frac{p+1}{2}})$. 
We have the following diagram
\[
\xymatrix{
 & \mathbb{P}^1 \times M_{dR}({\frac{p-1}{2}},{\frac{p+1}{2}}) \ar[ld]_{p_1} \ar[rd]^{p_2}  & \\
\mathbb{P}^1 & & M_{dR}({\frac{p-1}{2}},{\frac{p+1}{2}})
}
\]
Define $\mL:= {p_2}_*({p_1}^*\sO_{\mathbb{P}^1}(-\frac{p+1}{2}) \otimes V^u)$. For each $s \in M_{dR}({\frac{p-1}{2}},{\frac{p+1}{2}})$, $\mL_s = H^0(\mathbb{P}^1, \sO_{\mathbb{P}^1}(-\frac{p+1}{2}) \otimes V^u_s)$. By the definition of $\sU_{\frac{p+1}{2}}$ and $M_{dR}({\frac{p-1}{2}},{\frac{p+1}{2}})$, we know that $V^u_s \cong \sO({\frac{p-1}{2}})\oplus \sO({\frac{p+1}{2}})$. Thus $\mL$ is a line bundle on $M_{dR}({\frac{p-1}{2}},{\frac{p+1}{2}})$ by Grauert's theorem (see {\cite[Corollary~12.9]{hartshorne}}). Now we define $\mF:= {p_1}^*\sO({\frac{p+1}{2}}) \otimes {p_2}^*\mL$ as a line bundle on $\mathbb{P}^1 \times M_{dR}({\frac{p-1}{2}},{\frac{p+1}{2}})$.\\
Then there is a canonical nonzero morphism from $\mF$ to $V^u$:
\[
\mF= {p_1}^*\sO({\frac{p+1}{2}}) \otimes {p_2}^*{p_2}_*({p_1}^*\sO(-\frac{p+1}{2}) \otimes V^u) \stackrel{\neq0}{\longrightarrow} {p_1}^*\sO({\frac{p+1}{2}}) \otimes {p_1}^*\sO(-\frac{p+1}{2}) \otimes V^u \cong V^u.
\]
Thus the image $\text{Im}(\mF \to V^u)$ is a sub line bundle of $V^u$ on a Zariski dense open subset $W$ of $M_{dR}({\frac{p-1}{2}},{\frac{p+1}{2}})$, which gives the Hodge filtration of $V^u$ on $W$.\\[.2cm]

By the discussion above, $U:=C(W)$ is a Zariski open set of $M(1,0)$, where $C$ is the morphism induced by the Cartier functor. All Higgs bundles $(E,\theta)$ in $U$ will be sent back to $M(1,0)$ by applying the inverse Cartier transform, taking the quotient of $\mF_{[C^{-1}(E,\theta)]} \cong \sO(\frac{p+1}{2})$ and tensoring with $\mathcal O(\frac{1-p}{2})$. This process actually gives us a functor, which we denote as $\mathrm{Gr_{{\frac{p+1}{2}}}}\circ C_1^{-1}(\cdot)\otimes \mathcal O(\frac{1-p}{2})$.\\
Then we want to represent this functor as a rational self-map on the moduli scheme $M(1,0)$.
\begin{lem}\label{lem:rational map}
	The functor $\mathrm{Gr_{{\frac{p+1}{2}}}}\circ C_1^{-1}(\cdot)\otimes \mathcal O(\frac{1-p}{2})$ induces a rational map
	\[
	\varphi: M(1,0) \dashrightarrow M(1,0).
	\] 
\end{lem}
\begin{proof}
Let $\underline{M}(1,0)$ denote the moduli functor of semistable graded Higgs bundles of type $(1,0)$ (see \autoref{section DMS} for details), which is represented by the scheme $M(1,0)$. And $\underline{U}$ denotes the subfunctor corresponding to $U$.\\
Note that the functor $\mathrm{Gr_{{\frac{p+1}{2}}}}\circ C_1^{-1}(\cdot)\otimes \mathcal O(\frac{1-p}{2})$   gives a natural transform between these two moduli functors $\underline{U}$ and $\underline{M}(1,0)$. Since $\underline{M}(1,0)$ is represented by $M(1,0)$, one gets the following diagram
\[
\xymatrix{
\underline{U} \ar[d] \ar@{-->}[dr] & \\
\underline{M}(1,0) \ar[r]    &      \Hom_k (\cdot, M(1,0))
}
\]
By the universal property of the coarse moduli scheme, one get a natural transform
\[
\Hom_k (\cdot , U)  \longrightarrow \Hom_k (\cdot ,M(1,0))
\]
Take $Id \in \Hom_k (U,U)$, the natural transform will give the $k$-morphism
\[
U  \longrightarrow M(1,0)
\]
One can easily check that this map is induced by the self-map.   
\end{proof}
\begin{rmk}
We only deal with the first strata $\sU_{\frac{p+1}{2}}$ here. Actually the argument above can be applied for each strata $\sU_{k+1}/\sU_k$, for $k= {\frac{p+1}{2}}, {\frac{p+3}{2}}, {\frac{p+5}{2}}, \cdots$. The restriction of the self-map on each strata is a rational map from $\sU_{k+1}/\sU_k$ to $M(k-{\frac{p-1}{2}}, {\frac{p+1}{2}}-k)$. Therefore, the self-map is a contructible map.
\end{rmk}

Now we want to prove:
\begin{lem}
	\label{dominant}
	The rational map $\varphi$ is dominant.
\end{lem}
\begin{proof}
	We prove this lemma by induction on the number $m$ of the marked points.	For $m=3$, the lemma trivially holds since $M$ is just a point. Now suppose the statement is true for the case of $m-1$ marked points. We want to prove $\varphi$ is dominant for the case of $m$ marked points.
	Set $Z := \overline{\mathrm{Im}(\varphi)} \subset M(1,0)$ and we want to prove $Z = M(1,0)$. Suppose $Z$ is a proper subscheme of $M(1,0) \cong \P^{m-3}$. Then $\mathrm{dim} \, Z \leq m-4$. Denote $M(\hat{x_i})$ to be the moduli space of semistable graded Higgs bundles of rank $2$ degree $1$ over $\P^1$, with nilpotent logarithmic Higgs fields which have $m-1$ poles $\{x_1, \dots,\hat{x_i},\dots, x_m\}$. Then one can define a natural embedding $M(\hat{x_i}) \hookrightarrow M$ by forgetting one marked point $x_i$. Therefore,
	\[
	\bigcup_i \varphi(M(\hat{x_i};1,0)) \subset Z
	\]
	where $M(\hat{x_i};1,0)$ is the component of $M(\hat{x_i})$ with maximal dimension.
	Then we know that  $M(\hat{x_i};1,0) \cong \P^{m-4}$. So $\mathrm{dim} \, Z = m-4$ by the assumption that $\varphi$ is dominant for $m-1$ case. And $Z$ has more than one irreducible component. But this is impossible since $Z$ is the Zariski closure of $\varphi (M(1,0)) \cong \varphi (\P^{m-3})$, which is irreducible.    
\end{proof}

Now we can state and prove the main result of this section:
\begin{thm}
	\label{Periodic}
	The set of periodic points of $\varphi$ is Zariski dense in $M(1,0)$. Combining this with Proposition \ref{Self-map}, one gets infinitely many irreducible crystalline projective representations of the fundamental group.  
\end{thm}
To prove this we need a theorem of Hrushovski :

\begin{thm}[Hrushovski~\cite{Hru}, see also Theorem $3.7$ in \cite{ES}]
	\label{Hrushovski}
	Let $Y$ be an affine variety over $\F_q$, and let $\Gamma \subset (Y \times_{\F_q}Y) \otimes_{\F_q} \bar{\F}_q$ be an irreducible sub variety over $\bar{\F_q}$. Assume the two projections $\Gamma \to Y$ are dominant. Then, for any closed sub variety $W \subsetneq Y$, there exists $x \in Y(\bar{\F}_q)$ such that $(x, x^{q^m}) \in \Gamma$ and $x \notin W$ for large enough natural number $m$.
\end{thm}

\begin{proof}[proof of Theorem \ref{Periodic}]
	For each Zariski open subset $\U \subset M(1,0)$, we need to find a periodic point $x$ of $\varphi$ such that $x \in \U$. We take $Y$ to be an affine neighborhood of $M(1,0)$. And $\Gamma$ is the intersection of $\Gamma_{\varphi}\otimes_{\F_q} \bar{\F}_q$ and $(Y \times_{\F_q}Y) \otimes_{\F_q} \bar{\F}_q$. $W$ is defined to be the union of $(M(1,0) \setminus \U) \cap Y$ and the indeterminacy of $\varphi$.
	By Lemma \ref{dominant}, the projections $\Gamma \to Y$ are dominant. So we can apply Theorem \ref{Hrushovski} and find a point $x \in Y(\bar{\F}_q)$ such that $(x, x^{q^m}) \in \Gamma$ and $x \notin W$ for some $m$. Therefore, $x \in \U$, $\varphi$ is well-defined at $x$ and $\varphi(x) = x^{q^m}$ ($Y \subset \mathbb{A}^r$ , so $x$ can be written as $(x_1,\dots x_r) \in \mathbb{A}^r(\bar{\F}_q)$ and $x^{q^m} := (x^{q^m}_1,\dots x^{q^m}_r)$). The rational map $\varphi$ is well-defined at $x$ means that $\varphi$ is also well-defined at $x^{q^N}$ for any $N \in \mathbb{N}$.  Then we have
	\[
	\varphi (\varphi(x)) = \varphi (x^{q^m}) = \varphi(x)^{q^m}=x^{q^{2m}}
	\]
	Thus $\varphi^N(x) = x^{q^{Nm}}=x$ for $N$ large enough. That means, $x$ is a periodic point of $\varphi$.
\end{proof}

\subsection{An explicit formula of the self-map in the case of four marked points.}\label{section EFSMFMP}

\def\fai{\varphi_{\lambda,p}}
 In this section, we given an explicit formula of the self-map in case of $m=4$ marked point. Using M\"obius transformation on $\mathbb P^1$, we may assume these $4$ points are of form $\{0,1,\infty,\lambda\}$.
 By section~\ref{section DMS}, the moduli space $M$ is connected and isomorphic to $\P^1$, where the isomorphism is given by sending $(E,\theta)$ to the zero locus 
$(\theta)_0\in \mathbb P^1$. 
To emphasize the dependence of the self-map on $\lambda$ and $p$, we rewrite the self-map by  $\fai$. By calculation, details are given in the appendix section~\ref{Calculate_Selfmap}, we get 
\begin{equation}
\fai(z)=\frac{z^p}{\lambda^{p-1}}\cdot\left( \frac{f_\lambda(z^p)}{g_\lambda(z^p)}\right)^2,
\end{equation}
where $f_\lambda(z^p)$ is the determinant of matrix
\begin{equation*}\small
\left(\begin{array}{ccccc}  \frac{\lambda^p(1-z^p)-(\lambda^p-z^p)\lambda^{2}}{2}&
\frac{\lambda^p(1-z^p)-(\lambda^p-z^p)\lambda^{3}}{3}  &\cdots&
\frac{\lambda^p(1-z^p)-(\lambda^p-z^p)\lambda^{(p+1)/2}}{(p+1)/2} \\ \frac{\lambda^p(1-z^p)-(\lambda^p-z^p)\lambda^{3}}{3} &
\frac{\lambda^p(1-z^p)-(\lambda^p-z^p)\lambda^{4}}{4} &\cdots&
\frac{\lambda^p(1-z^p)-(\lambda^p-z^p)\lambda^{(p+3)/2}}{(p+3)/2} \\   \vdots&\vdots&\ddots&\vdots\\      \frac{\lambda^p(1-z^p)-(\lambda^p-z^p)\lambda^{(p+1)/2}}{(p+1)/2} &
\frac{\lambda^p(1-z^p)-(\lambda^p-z^p)\lambda^{(p+3)/2}}{(p+3)/2} &\cdots&
\frac{\lambda^p(1-z^p)-(\lambda^p-z^p)\lambda^{p-1}}{p-1} \\  
\end{array} \right)
\end{equation*}

and $g_\lambda(z^p)$ is the determinant of matrix
\begin{equation*}\small
\left(\begin{array}{ccccc}  \frac{\lambda^p(1-z^p)-(\lambda^p-z^p)\lambda^1}{1}&
\frac{\lambda^p(1-z^p)-(\lambda^p-z^p)\lambda^{2}}{2}  &\cdots&
\frac{\lambda^p(1-z^p)-(\lambda^p-z^p)\lambda^{(p-1)/2}}{(p-1)/2} \\ \frac{\lambda^p(1-z^p)-(\lambda^p-z^p)\lambda^{2}}{2} &
\frac{\lambda^p(1-z^p)-(\lambda^p-z^p)\lambda^{3}}{3} &\cdots&
\frac{\lambda^p(1-z^p)-(\lambda^p-z^p)\lambda^{(p+1)/2}}{(p+1)/2} \\   \vdots&\vdots&\ddots&\vdots\\      \frac{\lambda^p(1-z^p)-(\lambda^p-z^p)\lambda^{(p-1)/2}}{(p-1)/2} &
\frac{\lambda^p(1-z^p)-(\lambda^p-z^p)\lambda^{(p+1)/2}}{(p+1)/2} &\cdots&
\frac{\lambda^p(1-z^p)-(\lambda^p-z^p)\lambda^{p-2}}{p-2} \\  
\end{array} \right).
\end{equation*}
By calculation, for $p=3$ one has
\[\varphi_{\lambda,3}(z)= z^3 \left(\frac{z^3+\lambda(\lambda+1)}{(\lambda+1)z^3+\lambda^2}\right)^2
\] 
and $\varphi_{\lambda,3}(z)=z^{3^2}$ if and only if $\lambda=-1$; for $p=5$, one has
\[\varphi_{\lambda,5}(z)= z^5\left(\frac{z^{10}-\lambda(\lambda+1)(\lambda^2-\lambda+1)z^5+\lambda^4(\lambda^2-\lambda+1)}{(\lambda^2-\lambda+1)z^{10}-\lambda^2(\lambda+1)(\lambda^2-\lambda+1)z^5+\lambda^6 }\right)^2,\]

and $\varphi_{\lambda,5}(z)=z^{5^2}$ if and only if  $\lambda$ is a $6$-th primitive root of unit; for $p=7$ one has 
\begin{equation*}
\begin{split}
\varphi_{\lambda,7}(z)=& z^7\left(
 \begin{array}{l}   
z^{21}+2\lambda(\lambda+1)(\lambda^2+\lambda+1)(\lambda^2+3\lambda+1)z^{14}\\ 
\ + \lambda^4(\lambda+1)^2(\lambda^2+\lambda+1)(\lambda^2+1)z^{7} + \lambda^9(\lambda+1)(\lambda^2+\lambda+1)\\\hline
(\lambda+1)(\lambda^2+\lambda+1)z^{21}
	+\lambda^2(\lambda+1)^2(\lambda^2+\lambda+1)(\lambda^2+1)
	z^{14}\\
 \ 	+
	2\lambda^6(\lambda+1)(\lambda^2+\lambda+1)(\lambda^2+3\lambda+1)z^{7}
	+\lambda^{12}\\
\end{array}
\right)^2
\end{split}
\end{equation*}
and $\varphi_{\lambda,7}(z)=  z^{7^2}$ if and only if $(\lambda+1)(\lambda^2+\lambda+1)=0$.

We regard $\varphi_{\lambda,p}$ as a self-map on $\P^1$, which is rational and of degree $p^2\neq 1$. Thus it has $p^2+1$ fixed $\overline{k}$-points counting with multiplicity. Suppose the conjecture~\ref{conj-1} holds, then the multiplicity of each fixed point equals to $1$.
Let $(E,\theta)/\P_{k'}^1$ be a fixed point of $\varphi_{\lambda,p}$ defined over some extension field $k'$ of $k$. Then in the language of Higgs-de Rham flow, $(E,\theta)$ is the initial term of a twisted $1$-periodic Higgs-de Rham flow over $\P^1_{k'}$.

{\subsection {Lifting of twisted periodic logarithmic Higgs-de Rham flow on the projective line with marked points and strong irreducibility}
Here we just consider $1$-periodic case, for the higher-periodic case the treatment is similar.  First of all,  Higgs bundles considered here are given by logarithmic 1-forms on the punctured projective line vanishing at one point. They lift unobstructed to $W_n(k)$. Secondly, the obstruction group of lifting Hodge filtration in this case is $H^1(\bP^1_k, \mO(-1))=0$.
  Hence those two conditions required in Proposition~\ref{Lifting_PHDF}  hold true and one lifts $(E,\theta)$ to a twist periodic Higgs bundle over $\bP^1_{W_2}$. Recall the proof of Proposition~\ref{Lifting_PHDF}, one constructs a self-map on the torsor space of all liftings of $(E,\theta)$, and the fixed points of this self-map correspond to those liftings of the twisted $1$-periodic Higgs-de Rham flow. Fix a point $x_0$ in the torsor space, we identify the torsor space (\ref{gr-lifting space}) with $k$. Let $x$ be any point in the torsor space,  denote $z=x-x_0\in k$, by Corollary~\ref{InvCar_Torsor_map} and Proposition~\ref{prop:torsor_Grading}, there exists an element $a\in k$ such that 
$az^p=\mathrm{Gr}\circ C^{-1}(x)-\mathrm{Gr}\circ C^{-1}(x_0)$. Denote $b=\mathrm{Gr}\circ C^{-1}(x_0)-x_0\in k$, then the self-map on this torsor space is of form
\[z\mapsto az^p+b,\] 
where $a,b\in k$. 

\textbf{Case 1: $a=0$.}  Then $z=b$ is the unique fixed point of the self-map. In other words, there is a unique twisted periodic lifting of the given twisted $1$-periodic Higgs-de Rham flow
 over $\bP^1_{W_2(k)}$. 

\textbf{Case 2: $a\neq 0$.} Let $z_0\in \overline{k}$ be a solution of $z=az^p+b$. Then $\Sigma=\{i\cdot a^{-\frac{1}{p-1}}+z_0 \mid i\in \bF_p\}$ is the set of all solutions.
If $a\neq 0$ is not a $(p-1)$-th power of any element in $k^\times$, then $\#(\Sigma\cap k)\leq 1$. In other words there is at most one twisted $1$-periodic lifting over $\bP^1_{W_2(k)}$ of the given twisted $1$-periodic Higgs-de Rham flow. If $a\neq 0$ is a $(p-1)$-th power of some element in $k^\times$, then $\#(\Sigma\cap k)=0$ or $p$. In other words, if the twisted $1$-periodic Higgs-de Rham flow is liftable then there are exactly $p$ liftings over $\bP^1_{W_2(k)}$.
If we consider the lifting problem over an extension $k'$ of $k$, which contains $\Sigma$, then there are exactly $p$ liftings of the twisted $1$-periodic Higgs-de Rham flow over $\bP^1_{W_2(k')}$.
Repeating the same argument for lifting over truncated Witt ring of higher order we lift twisted periodic Higgs-de Rham flows over $W(\bar {\bF}_p)$. We prove
\begin{thm}\label{lifting-thm}
Any periodic Higgs bundle in $ M(1,0)_{\mathbf{F}_q}$ lifts to a periodic Higgs bundle in $M(1,0)_{\mathbb{Z}_p^\text{ur}}.$
\end{thm}
 We recall the notion of strong irreducibility of representations, which is introduced in \cite{LSYZ14}, Proposition 1.4.
Let $\rho: \pi^\text{\'et}_1(\mX^0_K)\to \mathrm{PGL_r}(\mathbb{Z}_p^{ur})$ be a representation and $\bar{\rho}$ be the restriction of $\rho$ to the geometric fundamental group $\pi^\text{\'et}_1(\mX^0_{\bar K}).$ We say $\rho$ is strongly irreducible if
for any surjective and generically finite logarithmic morphism 
$$f: \mathcal Y_{\bar K}\to \mX_{\bar K},$$
the pull-back representation $f^*(\bar \rho)$ is irreducible.\\[.2cm]
\begin{prop}\label{strong_irred} Let 
$$\rho :\pi^\text{\'et}_1((\bP^1\setminus \mathcal D)_{\mathbb Q_p^{ur}})
\to \mathrm{PGL_2}(\mathbb Z_p^{ur})$$
be a representation corresponding to a lifted twisted periodic logarithmic Higgs bundle $(E,\theta)=(\mathcal O(1)\oplus \mathcal O, \theta)$ over  $(\bP^1, \mathcal D)_{\mathbb Z_p^{ur}}$. Then $\rho$ is strongly irreducible.
\end{prop}
\begin{proof}
Denote $\mY=\bP^1=\mX$. We take the double cover of $\bP^1$   
\[\sigma:\mY \to \mX\] 
defined by $z\mapsto z^2$, which is ramified on $\{0,\infty\} \subset \mathcal D$. 
Taking the logarithmic structure $\mathcal D':=\sigma^*(\mathcal D)$ on $\mY$, 
then $\sigma$ is a logarithmic \'etale morphism with respect to the logarithmic structures $\mathcal D$ and $\mathcal D'$. 

The logarithmic inverse Cartier transforms on both logarithmic curves are compatible with respect to $\sigma^*$. 
We may choose compatible local Frobenius liftings on both logarithmic curves. Let $\mU$ be a small affine open subset of $\mX$ and denote $\mV=\sigma^{-1}\mU$. According to the following commutative diagram of logarithmic schemes
\begin{equation}
\xymatrix@C=2cm{(V_1,\mD'\mid_{V_1})\ar[d]_{\text{close embedding}} \ar[r]^{\Phi_{V_1}} & (\mV,\mD'\mid_{\mV}) \ar[d]^{\sigma}\\ (\mV,\mD'\mid_{\mV}) \ar@{..>}[ur]^{\exists \Phi_V} \ar[r]_{\Phi_\mU\circ\sigma} & (\mU,\mD\mid_{\mU})\\}
\end{equation}
 and the logarithmic \'etaleness of $\sigma$,
the Proposition~3.12 in~\cite{Kato88} implies that there exists a Frobenius lifting $\Phi_{\mV}$ on $\mV$ fitting into the commutative diagram. Since the inverse Cartier transforms is constructed by using the pullback via local Frobenius liftings, the local inverse Cartier transforms
    \footnote{The original inverse Cartier transform is defined by Ogus and Vologodsky~\cite{OgVo07} for characteristic $p$. And lately, it was generalized to the truncated version by Lan-Sheng-Zuo~\cite{LSZ13a} and to the logarithmic version by Schepler~\cite{Sch08} and Lan-Sheng-Yang-Zuo~\cite{LSYZ14}. Here we need a truncated logarithmic version. In this case, the inverse Cartier transform is defined in the same manner as in~\cite{LSZ13a} except that there are some restrictions on the choices of local liftings of the Frobenius map. The existence of such kind liftings is given by proposition~9.7 in~\cite{EV-92}. It is routine to give an explicit definition. We left it to the readers.}
on both logarithmic curves are compatible with respect to $\sigma^*$. After the gluing process, one gets global compatibility.  

According to the compatibility of logarithmic inverse Cartier transforms on both logarithmic curves, the periodicity is preserved by the pull-back $\sigma^*$, i.e. if $(E,\theta)^{(1)}$ is a logarithmic  $\mathcal O(1)^{\otimes{ p-1\over 2}}$-twisted periodic Higgs bundle over $(\mX, \mD)_1$, then $\sigma^*(E,\theta)^{(1)}\otimes \mO(1)^{-1}$ is a logarithmic periodic Higgs bundle over $(\mY, \mD')_1$ and   $\sigma^*\theta\not=0$. Furthermore if $(E,\theta)^{(l)}$ is a lifting of $(E,\theta)^{(1)}$ to $(\mX, \mD)_l$ as a twisted periodic Higgs bundle, then 
$\sigma^*(E,\theta)^{(l)}\otimes \mO(1)^{-1}$ is a lifting of $\sigma^*(E,\theta)^{(1)}\otimes\mO(1)^{-1}$ to $(\mX,\mD)_\ell$ as a periodic Higgs bundle. In this way we show that the projective representation
$\sigma^*\rho$ lifts to a $\mathrm{GL_2}$-crystalline representation
$$\rho':\pi^\text{\'et}_1((\mY\setminus\mD')_{\mathbb Q_p^{ur}})\to \mathrm{GL_2}((\mathbb Z_p^{ur}))$$
 corresponding to the Higgs bundle $\sigma^*((E,\theta))\otimes \mO(-1)=:(E,\theta)'$ over
$(\mY, \mD')_{\bZ_p^{ur}}$ of the form
$$(\mO(1)\oplus \mO(-1),\sigma^*\theta_{\not=0}: \mO(1)\to \mO(-1)\otimes \Omega^1_{\mY}(\log \mD')).$$
 By the same argument used in the proof of Proposition 1.4 in \cite{LSYZ14},   we are going to show that 
$\rho'$ is strongly irreducible. Hence $\rho$ is strongly irreducible. Let $f: \mZ\to \mY$ be a surjective logarithmic morphism between logarithmic curves. By the example in page 861 of \cite{Fal05}, one can see that the generalized representation associated to $(E,\theta)'_{\mathbb C_p}:=(E,\theta)'
\otimes\mathbb C_p$ is compatible with $\bar \rho'$ by tensoring with $\mathbb C_p$. We can find a finite extension field $K'$ of $\mathbb Q_p^{ur}$ with its integral ring $\mathcal O_{K'}$, such that  $ \mZ_{\bar K}$ has an integral model $\mZ_{\mathcal O_{K'}}$ over $ \mathcal O_{K'}$ and with toroidal singularity. By the construction of the correspondence 
 (\cite{Fal05}, Theorem 6), the twisted pullback of the graded Higgs bundle $f^{\circ}(E,\theta)'_{\mathbb C_p}$ corresponds to the pullback representation of
$\bar \rho'\otimes \mathbb C_p$ to $\pi^\text{\'et}_1(\mZ^o_{\mathbb C_p})$.  By the construction of the twisted pullback, one has a short exact sequence
$$ 0\to (f^*\mathcal O(-1), 0)_{\mathbb C_p}\to f^\circ(E,\theta)'_{\mathbb C_p}\to (f^*\mathcal O(1),0)_{\mathbb C_p}\to 0,$$ 
and that the Higgs field of $f^\circ(E,\theta)'_{\mathbb C_p}$ is nonzero. Assume by contradiction $f^*\bar \rho'\otimes \mathbb C_p$ is not irreducible. Then it contains a one-dimensional 
$\mathbb C_p$-subrepresentation. By the last paragraph in Page 860 of \cite{Fal05}, it follows that  $f^\circ(E,\theta)'_{\mathbb C_p}$ contains a rank-1 Higgs subbundle $(N,0)$ of $\deg N=0.$
Since the Higgs field of $f^\circ(E,\theta)'_{\mathbb C_p}$ is nonzero,  $ (f^*\mathcal O(-1), 0)_{\mathbb C_p}$ in the above short exact sequence is the unique rank-1 Higgs subbundle.
Hence, one obtains a nonzero factor map $(N,0)\to (f^*\mathcal O(-1), 0)_{\mathbb C_p}.$ But it is impossible, since $\deg N>\deg f^*\mathcal O(-1).$ The proof is completed.
\end{proof}

\begin{rmk}
The inverse Cartier functor over truncated Witt rings was defined by Lan-Sheng-Zuo~\cite{LSZ13a}
\end{rmk}

\subsection{Examples of dynamics of Higgs-de Rham flow on
$\mathbb{P}^1$ with four-marked points}
In the following, we give some examples in case $k=\F_{3^4}$. For any $\lambda\in k\setminus\{0,1\}$, the map $\varphi_{\lambda,3}$ is a self $k$-morphism on $\P^1_k$. So it can be restricted as a self-map on the set of all $k$-points
\[\varphi_{\lambda,3}:k\cup \{\infty\}\rightarrow k\cup \{\infty\}.\] 
Since $\alpha=\sqrt{1+\sqrt{-1}}$ is a generator of $k=\F_{3^4}$ over $\F_3$, every elements in $k$ can be uniquely expressed in form $a_3\alpha^3+a_2\alpha^2+a_1\alpha+a_0$, where $a_3,a_2,a_1,a_0\in\{0,1,2\}$. We use the integer $27a_3+9a_2+3a_1+a_0\in [0,80]$ stand for the element $a_3\alpha^3+a_2\alpha^2+a_1\alpha+a_0$. By identifying the set $k\cup\infty$ with $\{0,1,2,\cdots,80,\infty\}$ in this way, we get a self-map on $\{0,1,2,\cdots,80,\infty\}$ for all $\lambda\in k$
\[\varphi_{\lambda,3}:\{0,1,2,\cdots,80,\infty\}\rightarrow \{0,1,2,\cdots,80,\infty\}.\]
In the following diagrams, the arrow
\textcircled{$\beta$} $\rightarrow$ \textcircled{$\gamma$} means $\gamma=\varphi_{\lambda,3}(\beta)$.
And an $m$-length loop in the following diagrams just stands for a twisted $m$-periodic Higgs-de Rham flow, which corresponds to $\mathrm{PGL}_2(\F_{3^m})$-representation by Theorem~\ref{equiv:logTFF&THDF} and Theorem~\ref{Mainthm}.

$\bullet$ For $\lambda=2\sqrt{1+\sqrt{-1}}$, we have 
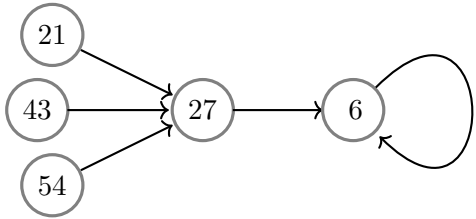
\begin{center}
	\begin{tikzpicture}
	[L1Node/.style={circle,draw=black!50, very thick, minimum size=7mm}]
	\node[L1Node] (n1) at (-2, 1){$21$}; 
	\draw [thick,->](-1.62,0.8) -- (-0.4,0.2);
	\node[L1Node] (n1) at (-2.2, 0){$43$};
	\draw [thick,->](-1.8,0) -- (-0.45,0);
	\node[L1Node] (n1) at (-2, -1){$54$};
	\draw [thick,->](-1.62,-0.8) -- (-0.4,-0.2); 
	\node[L1Node] (n1) at (0, 0){$27$}; 
	\draw [thick,->](0.4,0) -- (1.6,0);
	\node[L1Node] (n1) at (2, 0){$~6\,$};  
	\draw [thick,->](2.3,0.3) .. controls (4,2) and (4,-2) .. (2.35,-0.35);   
	\end{tikzpicture}	
\end{center}

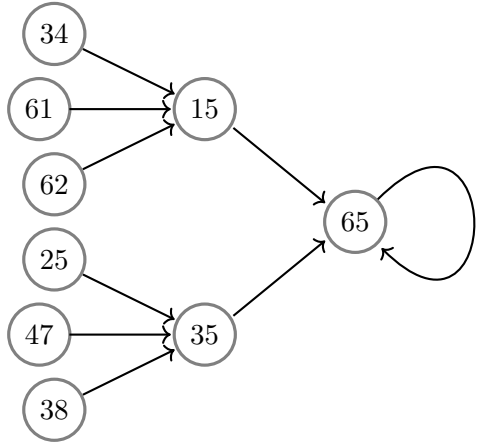
\begin{center}
	\begin{tikzpicture}
	[L1Node/.style={circle,draw=black!50, very thick, minimum size=7mm}]
	\node[L1Node] (n1) at (-2, 2.5){$34$}; 
	\node[L1Node] (n1) at (-2.2, 1.5){$61$};
	\node[L1Node] (n1) at (-2, 0.5){$62$};
	\node[L1Node] (n1) at (0, 1.5){$15$}; 
	\node[L1Node] (n1) at (-2, -2.5){$38$}; 
	\draw [thick,->](-1.62,-2.3) -- (-0.4,-1.7);
	\node[L1Node] (n1) at (-2.2, -1.5){$47$};
	\draw [thick,->](-1.8,-1.5) -- (-0.45,-1.5);
	\node[L1Node] (n1) at (-2, -0.5){$25$};
	\node[L1Node] (n1) at (0, -1.5){$35$};
	\node[L1Node] (n1) at (2, 0){$65$};
	\draw [thick,->](-1.62,-0.7) -- (-0.4,-1.3); 
	\draw [thick,->](-1.8,1.5) -- (-0.45,1.5);
	\draw [thick,->](-1.62,2.3) -- (-0.4,1.7);
	\draw [thick,->](-1.62,0.7) -- (-0.4,1.3); 
	\draw [thick,->](0.38,1.25) -- (1.6,0.25);
	\draw [thick,->](0.38,-1.25) -- (1.6,-0.25);
	\draw [thick,->](2.3,0.3) .. controls (4,2) and (4,-2) .. (2.35,-0.35); 
	\end{tikzpicture}
\end{center}
The $1$-length loops $\xymatrix{\textcircled{\tiny{6}}\ar@(ur,dr)}$ \quad and $\xymatrix{\textcircled{\tiny{65}}\ar@(ur,dr)}$ \quad in the diagrams above correspond to projective representations of form
\[\pi^\text{\'et}_1\left(\P_{W(\F_{3^4})[1/3]}^1\setminus\left\{0,1,\infty,2\sqrt{1+\sqrt{-1}}\right\}\right)\longrightarrow \mathrm{PGL}_2(\F_3),\]
here $W(\F_{3^4})[1/3]$ is the unique unramified extension of $\Q_3$ of degree $4$.

$\bullet$ For $\lambda=\sqrt{-1}$, we have
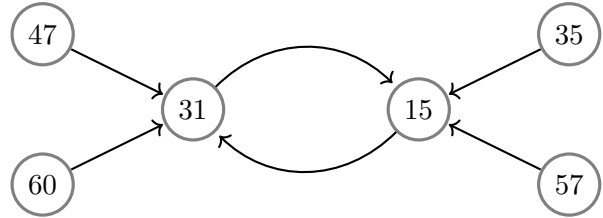
\begin{center}
	\begin{tikzpicture}
	[L1Node/.style={circle,draw=black!50, very thick, minimum size=7mm}]
	\node[L1Node] (n1) at (0, 1){$47$}; 
	\draw [thick,->](0.38,0.8) -- (1.6,0.2);
	\node[L1Node] (n1) at (0, -1){$60$};
	\draw [thick,->](0.38,-0.8) -- (1.6,-0.2);
	\node[L1Node] (n1) at (2, 0){$31$};
	\draw [thick,->](2.3,0.3) .. controls (3,1) and (4,1) .. (4.65,0.35); 
	\node[L1Node] (n1) at (7, 1){$35$}; 
	\draw [thick,->](6.62,0.8) -- (5.4,0.2);
	\node[L1Node] (n1) at (7, -1){$57$};
	\draw [thick,->](6.62,-0.8) -- (5.4,-0.2);
	\node[L1Node] (n1) at (5, 0){$15$};
	\draw [thick,->](4.7,-0.3) .. controls (4,-1) and (3,-1) .. (2.35,-0.35); 
	\end{tikzpicture}
\end{center}
The $2$-length loop $\xymatrix{\textcircled{\tiny{31}}\ar@/^/[r] &\textcircled{\tiny{15}}\ar@/^/[l]}$ corresponds to a projective representation of form \[\pi^\text{\'et}_1\left(\P_{W(\F_{3^4})[1/3]}^1\setminus\left\{0,1,\infty,\sqrt{-1}\right\}\right)\longrightarrow \mathrm{PGL}_2(\F_{3^2}).\] 
We also have diagram
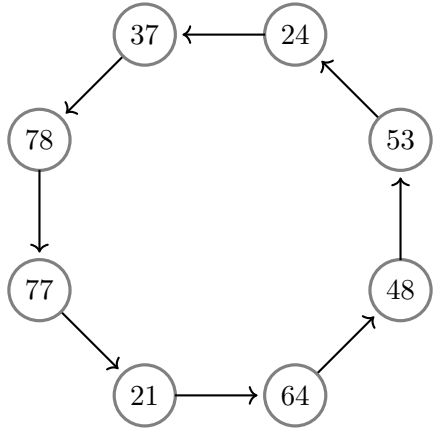
\begin{center}
	\begin{tikzpicture}
	[L1Node/.style={circle,draw=black!50, very thick, minimum size=7mm}]
	\node[L1Node] (n1) at (-1,-2.4){$21$}; 
	\draw [thick,->](-0.6,-2.4) -- (0.5,-2.4);
	\node[L1Node] (n1) at (1,-2.4){$64$}; 
	\draw [thick,->](1.3,-2.1) -- (2.05,-1.35);
	\node[L1Node] (n1) at (2.4,-1){$48$}; 
	\draw [thick,->](2.4,-0.6) -- (2.4,0.5);
	\node[L1Node] (n1) at (2.4,1){$53$}; 
	\draw [thick,->](2.1,1.3) -- (1.35,2.05);
	\node[L1Node] (n1) at (1,2.4){$24$}; 
	\draw [thick,->](0.6,2.4) -- (-0.5,2.4);
	\node[L1Node] (n1) at (-1,2.4){$37$}; 
	\draw [thick,->](-1.3,2.1) -- (-2.05,1.35);
	\node[L1Node] (n1) at (-2.4,1){$78$}; 
	\draw [thick,->](-2.4,0.6) -- (-2.4,-0.5);
	\node[L1Node] (n1) at (-2.4,-1){$77$}; 
	\draw [thick,->](-2.1,-1.3) -- (-1.35,-2.05);
	\end{tikzpicture}
\end{center}
which is an $8$-length loop and corresponds to a projective representation of form \[\pi^\text{\'et}_1\left(\P_{W(\F_{3^8})[1/3]}^1\setminus\left\{0,1,\infty,\sqrt{-1}\right\}\right)\longrightarrow \mathrm{PGL}_2(\F_{3^8}).\]

$\bullet$ For $\lambda=2+\sqrt{1+\sqrt{-1}}$, one has 
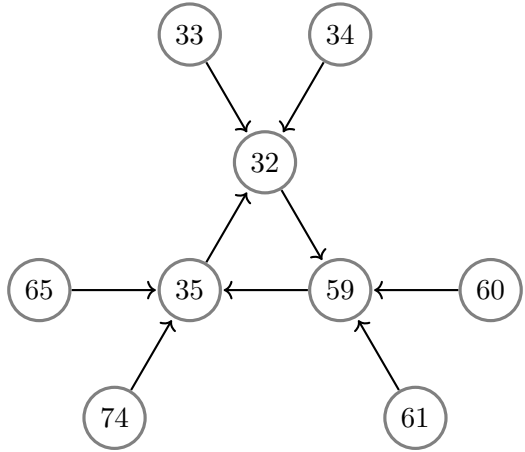
\begin{center}
	\begin{tikzpicture}
	[L1Node/.style={circle,draw=black!50, very thick, minimum size=7mm}] 
	\node[L1Node] (n1) at (-1,3.4){$33$};
	\draw [thick,->](-0.78,3.03) -- (-0.24,2.1); 
	\node[L1Node] (n1) at (1,3.4){$34$}; 
	\draw [thick,->](0.78,3.03) -- (0.24,2.1); 
	\node[L1Node] (n1) at (0,1.7){$32$}; 
	\draw [thick,->](0.22,1.33) -- (0.76,0.4); 
	\node[L1Node] (n1) at (-3,0){$65$}; 
	\draw [thick,->](-2.57,0) -- (-1.45,0);  
	\node[L1Node] (n1) at (-1,0){$35$}; 
	\draw [thick,->](-0.78,0.37) -- (-0.24,1.3);  
	\node[L1Node] (n1) at (1,0){$59$}; 
	\draw [thick,->](0.57,0) -- (-0.55,0);  
	\node[L1Node] (n1) at (3,0){$60$}; 
	\draw [thick,->](2.57,0) -- (1.45,0);  
	\node[L1Node] (n1) at (-2,-1.7){$74$}; 
	\draw [thick,->](-1.78,-1.33) -- (-1.24,-0.4);  
	\node[L1Node] (n1) at (2,-1.7){$61$};  
	\draw [thick,->](1.78,-1.33) -- (1.24,-0.4);  
	\end{tikzpicture}
\end{center}
and the $3$-length loop in this diagram corresponds to a projective representation of form  \[\pi^\text{\'et}_1\left(\P_{W(\F_{3^{12}})[1/3]}^1\setminus\left\{0,1,\infty,2+\sqrt{1+\sqrt{-1}}\right\}\right)\longrightarrow \mathrm{PGL}_2(\F_{3^3}).\]
We also have
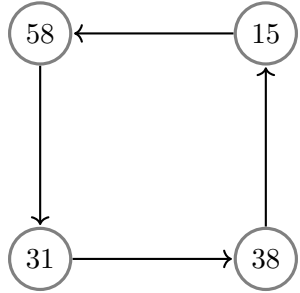
\begin{center}
	\begin{tikzpicture}
	[L1Node/.style={circle,draw=black!50, very thick, minimum size=7mm}]  
	\node[L1Node] (n1) at (1.5,1.5){$15$};
	\draw [thick,->](-1.07,-1.5) -- (1.05,-1.5);   
	\node[L1Node] (n1) at (-1.5,1.5){$58$};
	\draw [thick,->](1.5,-1.07) -- (1.5,1.05);   
	\node[L1Node] (n1) at (1.5,-1.5){$38$};
	\draw [thick,->](1.07,1.5) -- (-1.05,1.5);   
	\node[L1Node] (n1) at (-1.5,-1.5){$31$};
	\draw [thick,->](-1.5,1.07) -- (-1.5,-1.05);  
	\end{tikzpicture}
\end{center}
which is a $4$-length loop and corresponds to a projective representation of form  \[\pi^\text{\'et}_1\left(\P_{W(\F_{3^4})[1/3]}^1\setminus\left\{0,1,\infty,2+\sqrt{1+\sqrt{-1}}\right\}\right)\longrightarrow \mathrm{PGL}_2(\F_{3^4}).\]

\subsection{Question on periodic Higgs bundles and torsion points on the associated elliptic curve.}
For $\mathbb{P}^1_{W(\mathbb{F}_q)}$  with 4 marked 
points $\{0,\,1,\,\infty,\,\lambda\}$  we denote the associated elliptic curve as the double cover $\pi: \mathcal{C}_\lambda\to \mathbb{P}^1$ ramified on 
$\{0,\,1,\,\infty,\,\lambda\}$ and 
$[p]: \mathcal{C}_\lambda \to  \mathcal{C}_\lambda$ to be the multiplication by $p$ map.
\begin{conj}\label{conj-1}
The following diagram commutes
\[
\xymatrix{
	\mathcal{C}_\lambda \ar[r]^{[p]} \ar[d]^{\pi} & \mathcal{C}_\lambda \ar[d]^{\pi} \\
	\mathbb{P}^1 \ar[r]^{\varphi_{\lambda,p}} & \mathbb{P}^1 \\
	}
\]
where $\varphi_{\lambda,p}$ is the self-map induced by the Higgs-de Rham flow.
\end{conj}
Conjecture \ref{conj-1} has been checked to be true for all primes $p\leq 50.$\\[.2cm]
\begin{cor}  Assuming Conjecture \ref{conj-1} holds. A Higgs bundle $(E,\theta)\in M(1,0)_{\mathbb{F}_q}$ is periodic if and only the zero $(\theta)_0=\pi(x)$, where $x$ is a torsion point in $ \mathcal{C}_\lambda$ of order coprime to $p$.
\end{cor}
In Theorem \ref{lifting-thm} we have shown that any periodic Higgs bundle in $M(1,0)_{\mathbb{F}_q}$  lifts to a periodic Higgs bundle over $\mathbb{Z}_p^\text{ur}$. In fact, there are infinitely many liftings if $\phi_{\lambda, p}\not= z^{p^2}.$ \\[.2cm]
\begin{conj}\label{conj-2} 
A periodic Higgs bundle $(E,\theta)$ in $M(1,0)_{\mathbb{F}_q}$
 lifts to a periodic Higgs bundle $(\mathcal{E}, \mathcal{\theta})$  in $M(1,0)_{W(\mathbb{F}_{q'})}$ if and only
 if $(\mathcal{\theta})_0=\pi(x)$, where $x$ is a torsion point in $ \mathcal{C}_\lambda$ of order coprime to $p$.
\end{conj}
\subsection{Question on $\ell$-adic representation and  $\ell$-to-$p$ companions.}
Kontsevich has observed a relation between the set of isomorphic classes of $\text{GL}_2(\bar{\mathbb Q}_l)$-local systems over $\mathbb{P}^1\setminus \{0,\,1,\infty,\lambda\}$ over $\mathbb{F}_q$ and the set of rational points on $C_\lambda$ over $\mathbb{F}_q$ (see {\cite[section 0.1]{kontsevich}})  via the work of Drinfeld on the Langlands program over function field. It looks quite mysterious as the elliptic curve appears in $p$-adic as well in $\ell$-adic case. There should exist a relation between periodic Higgs bundles in the $p$-adic world and the Hecke-eigenforms in the $\ell$-adic world via Abe's solution of Deligne conjecture on $\ell$-to-$p$ companions.
 To make the analogy, we first lift the $\mathrm{PGL}_2$-representations to $\mathrm{GL}_2$-representations.\\

In this paragraph, we keep the notations in the proof of Proposition~\ref{strong_irred}.
Now we want to descent the $\mathrm{GL}_2$-representation 
 \[\rho':\pi^\text{\'et}_1((\mY\setminus\mD')_{\bQ_{q'}})\to \mathrm{GL_2}((\bZ_q))\] 
to a $\mathrm{GL}_2$-representation of $\pi^\text{\'et}_1((\mX\setminus\mD)_{\bQ_{q'}})$, whose projectification is just the projective representation $\rho$. 
There is a natural action of the deck transformation group $G=\mathrm{Gal}(\mY/\mX)$ on $\mO_\mY=\sigma^*\mO_\mX$, with $\mO_\mX=\mO_\mY^G$. Since $0$ and $\infty$ are fixed by $G$, the action of $G$ on $\mO_\mY$ can be extended to $\mO_\mY((0))$ and $\mO_\mY((\infty))$. On the other hand, both $\mO_\mY((0))$ and $\mO_\mY((\infty))$ are isomorphic to $\mO_\mY(1)$.
We could endow two actions of $G$ on the periodic Higgs bundle $\sigma^*(E,\theta)\otimes\mO_\mY(1)^{-1}$. This is equivalent to endow a parabolic structure
\footnote{
Recall that to give a Higgs bundle with parabolic structure is equivalent to give a Higgs bundle over some Galois covering with an action of the deck transformation group. In our case, we can view $\sigma^*(E,\theta)\otimes\mO_\mY(1)^{-1}$ with a $G$-action as the Higgs bundle $(E,\theta)$ with a parabolic structure.
}
on the Higgs bundle $(E,\theta)$.  Extend the actions to the periodic graded Higgs-de Rham flow with initial term $\sigma^*(E,\theta)\otimes\mO_\mY(1)^{-1}$. Via Faltings' functor $\mathbb D$, the actions of $G$ on the flow induce actions of $G$ on the sections of locally constant sheaf on $(\mY\setminus\mD')_{\mathbb Q_{q'}}$. Then those $G$-invariant sections forms a locally constant sheaf on $(\mX\setminus\mD)_{\mathbb Q_{q'}}$. This gives us a $\mathrm{GL_2}((\mathbb Z_q))$-representation of  $\pi^\text{\'et}_1((\mX\setminus\mD)_{\mathbb Q_{q'}})$.

 For example, if we choose the $G$-action on $\mO_\mY(1)$ via the isomorphism $\mO_\mY(1)\simeq \mO_\mY((\infty))$. Then one could lift the $\text{PGL}_2( \mathbb{Z}_q  )$-representation to 
\[\rho: \pi^\text{\'et}_1(X\setminus \{0,\,1,\infty,\lambda\})\to \text{GL}_2(\mathbb{Z}_q)\]
such that the local monodromy around $\{0,\,1,\,\lambda\}$ are unipotent and around $\infty$ is quasi-unipotent with eigenvalue $-1$.\\[.1cm]
 Let $(V,\nabla, Fil^\bullet,\Phi)$ be the Fontaine-Faltings module corresponding to $\rho$.
Forgetting the Hodge filtration on $V$  one obtains
 a logarithmic  $F$-isocrystal
 $(V,\nabla, \Phi)/{(\mathbb{P}^1,\{0,\,1,\,\lambda,\infty\})}_{\mathbb{Q}_{q'}}$, which should correspond to an $\ell$-adic representation $\rho_{\ell} :\pi^\text{et}_1(\mathbb{P}^1_{\mathbb {F}_{q'}}\setminus  \{0,\,1,\,\lambda,\infty\} )\to \text{GL}_2( \bar{\mathbb{Q}}_\ell)$ by
applying Abe's solution of Deligne's conjecture on $\ell$-to-$p$ companion (see {\cite[Theorem 4.4.1]{Abe-13}} or {\cite[Theorem 7.4.1]{Drinfeld-18}}). However,  in order to apply Abe's theorem one has to check the determinant of the $F$-isocrystal $(V,\nabla, \Phi_p)/(\mathbb{P}^1,\{0,\,1,\,\lambda,\infty\})$ is of finite order (note that the category of $F$-isocrystal is a tensor category, and $\text{det}(V,\nabla, \Phi_p)$ is of finite order just means that its some tensor power becomes the trivial $F$-isocrystal $\mathcal{O}_{\mathbb{P}^1}$).
\begin{conj}\label{conj-3}   
There exist elements $ u\in \mathbb{Z}_{q'}^*$ such that
 $(\det V,\det\nabla, u\det \Phi)$
 corresponds to a finite character of 
 $  \pi^\text{\'et}_1(\mathbb{P}^1_{\mathbb{Q}_{q'}}\setminus \{0,\,1,\infty,\lambda\}).$
\end{conj}
 \subsection{Projective $F$-units crystal on smooth projective curves}\label{proj F-units}
 Let $\X$ be a smooth proper scheme over $W(k)$.  In~\cite{LSZ13a} an equivalence between the category of $f$-periodic vector bundles  $(E,0)$ of rank-$r$  over $X_n$ (i.e.  $(E,0)$ initials an $f$-periodic Higgs-de Rham flow with zero Higgs fields in all Higgs terms)  and the category of $\mathrm{GL}_r(W_n(\F_{p^f}))$-representations of $\pi^\text{\'et}_1(X_1)$ has been established. This result generalizes Katz's original theorem for $\X$ being an affine variety. As an application of our main theorem, 
 we show that
 \begin{thm}
 The $\mathbb D^P$ functor is faithful from the category of rank-$r$ twisted $f$-periodic vector bundles $(E,0)$ over $X_n$ to the the category of projective $W_n(\F_{p^f})$-representations of $\pi^\text{\'et}_1(X_{1,k'})$ of rank $r$, where $k'$ is the minimal extension of $k$ containing $\F_{p^f}$.	
 \end{thm} 
 
%  \begin{thm}
% There is an equivalence between the category of rank-$r$ twisted $f$-periodic vector bundles $(E,0)$ over $X_n$ and the the category of projective $W_n(\F_{p^f})$-representations of $\pi^\text{\'et}_1(X_1)$ of rank $r$.	
%\end{thm} 
 
 \begin{rmk}
 	For $n=1$  the above theorem is just a projective version of Lange-Stuhler's theorem.
 \end{rmk}

 \begin{thm}[lifting twisted periodic vector bundles]
 Let $(E,0)/X_1$ be an twisted $f$-periodic vector bundle. Assume $H^2(X_1, End(E))=0$.
 Then for any $n\in \mathbb N$ there exists some positive integer $f_n$ with $f\mid f_n$ such that $(E,0)$ lifts to a twisted $f_n$-periodic vector bundle over $X_n$.	
 \end{thm}

 Translate the above theorem in the terms of representations:
 \begin{thm}[lifting projective representations of $\pi^\text{\'et}_1(X_1)$]
  Let $\rho$ be a projective $\F_{p^f}$-representation of $\pi^\text{\'et}_1(X_1)$. Assume $H^2(X_1, End(\rho))=0$, then there exist an positive integer $f_n$ divided by $f$ such that
  $\rho$  lifts to a projective $W_n(\F_{p^{f_n}})$-representation of $\pi^\text{\'et}_1(X_{1,k'})$ for any $n\in\mathbb N$, where $k'$ is the minimal extension of $k$ containing $\F_{p^{f_n}}$.	
 \end{thm}

 Assume $\X$ is a smooth proper curve over $W(k)$,  de Jong and Osserman (see Appendix A in~\cite{Osserman}) have shown that the subset of periodic vector bundles over $X_{1,\overline{k}}$  is Zariski dense in the moduli space of semistable vector bundles over $X_1$ (Laszlo and Pauly have also studied some special case, see~
 \cite{LP}). Hence by Lange-Stuhler's theorem (see~\cite{LangeStuhe}) every periodic vector bundle corresponds to a $(P)\mathrm{GL}_r(\F_{p^f})$-representations of $\pi^\text{\'et}_1(X_{1, k'})$, where $f$ is the period and $k'$ is a definition field of the periodic vector bundle containing $\F_{p^f}$.

 \begin{cor}Every $(\mathrm{P})\mathrm{GL}_r(\F_{p^f})$-representation of $\pi^\text{\'et}_1(X_{1,k'})$ lifts to a $\mathrm{(P)GL}_r(W_n(\F_{p^{f_n}}))$-representation of $\pi^\text{\'et}_1(X_{1,k''})$ for some positive integer $f_n$ divided by $f$, where  $k''$ is a definition field of the periodic vector bundle containing $\F_{p^{f_n}}$. 
 \end{cor}
 
 \begin{rmk}
 	It shall be very interesting to compare this result with Deninger-Werner's theorem (see~\cite{DW}). Let $\mathscr E$ be a vector bundle over $\X$, we view it as a graded Higgs bundle with trivial filtration and trivial Higgs field. Suppose it is preperiodic over $X_1$. Then it has strongly semistable reduction
of degree zero. According to Deninger-Werner's theorem, this vector bundle induces a $\mathrm{GL}_r(\C_p)$-representation of $\pi^\text{\'et}_1(X_{\bar K})$.
 \end{rmk}

\section{Base change of twisted Fontaine-Faltings modules and twisted Higgs-de Rham flows over very ramified valuation rings}
  Let $k$ be a finite field of characteristic $p$ containing $\mathbb F_{p^f}$. Denote $K_0=W(k)[\frac1p]$. Let $\X$ be a smooth proper scheme over $W(k)$ together with a smooth log structure $\D/W(k)$. For any finite extension $K$ of $K_0$, denote $X_K^o=(\X\times_{W(k)} \Spec K)\setminus (\D\times_{W(k)} \Spec K)$. Recall that Theorem~\ref{Mainthm} guarantees the existence of non-trivial representations of \'etale fundamental group in terms of the existence of semistable graded Higgs bundles. Then the  $\mathrm{PGL}_r(\mathbb{F}_{p^f})$-crystalline representation of $\pi^\text{\'et}_1((X_{K_0}^o)$ corresponding to the stable Higgs bundle should have some stronger property: its restriction to the geometric fundamental group $\pi^\text{\'et}_1((\mathcal{X} \setminus \mathcal{D})_{\bar{\mathbb{Q}}_p})$ is absolutely irreducible.\\
We outline the proof as follows. Fix a $K_0$-point in $X_{K_0}$, one can pull back the representation $\rho$ to a representation of the Galois group, whose image is finite. This finite quotient will give us a field extension $K/K_0$ such that the restriction of $\rho$ on $\mathrm{Gal}(\bar{K_0}/K)$ is trivial. That means $\rho(\pi^\text{\'et}_1(X_{K}^o))=\rho(\pi^\text{\'et}_1(X^o_{\overline{K}_0}))$. So it suffices to prove the irreducibility of $\rho$ on $\pi^\text{\'et}_1(X_{K}^o)$, which gives us the chance to apply the method of twisted periodic Higgs-de Rham flows as before. But the field extension $K/K_0$ is usually ramified. So we have to work out the construction in the previous sections to the very ramified case. Note that this section is deeply inspired by Faltings' work~\cite{Fal99}.

\subsection{Notations in the case of $\Spec \,k$.} In this notes, $k$ will always be a perfect field of characteristic $p>0$. Let $\pi$ be a root of an Eisenstein polynomial
\[f(T)=T^e+\sum_{i=0}^{e-1}a_i T^i\]
of degree $e$ over the Witt ring $W=W(k)$. Denote $K_0=\Frac(W)=W[\frac1p]$ and $K=K_0[\pi]$, where $K_0[\pi]$ is a totally ramified extension of $K_0$ of degree $e$. Denote by $W_\pi=W[\pi]$ the ring of integers of $K$, which is a complete discrete valuation ring with maximal ideal $\pi W_\pi$ and the residue field $W_\pi/\pi W_\pi=k$. Denote by $W[[T]]$ the ring of formal power-series over $W$. Then 
\[W_\pi=W[[T]]/fW[[T]].\]
The PD-hull $\mB_{W_\pi}$ of $W_\pi$ is the PD-completion of the ring obtained by adjoining to $W[[T]]$ the divided powers $\frac{f^n}{n!}$. 
%(it is equivalent to adjoin $\frac{T^{en}}{n!}$, as $(p)$ is already divided powers. In particular, the ring $\mB_{W_\pi}$ only depends on $e$). We have
 More precisely
\[\mB_{W_\pi}=\left\{\left.\sum_{n=0}^{\infty}a_nT^n\in K_0[[T]] \right| a_n[n/e]!\in W \text{ and }a_n[n/e]!\rightarrow 0\right\}.\]
A decreasing filtration is defined on $\mB_{W_\pi}$ by the rule that $F^q(\mB_{W_\pi})$ is the closure of the ideal generated by divided powers $\frac{f^n}{n!}$ with $n\geq q$. Note that the ring $\mB_{W_\pi}$ only depends on the degree $e$ while this filtration depends on $W_\pi$ and $e$. One has 
\[\mB_{W_\pi}/\Fil^1\mB_{W_\pi}\simeq W_\pi. \]
There is a unique continuous homomorphism of $W$-algebra $\mB_{W_\pi}\rightarrow B^+(W_\pi)$ which sends $T$ to $[\underline{\pi}]$. Here $\underline{\pi} = (\pi,\pi^{\frac{1}{p}},\pi^{\frac{1}{p^2}},\dots) \in \varprojlim \bar{R}$. 
We denote
\[\tmB_{W_\pi}=\mB_{W_\pi}[\frac{f}{p}]\]
which is a subring of $K_0[[T]]$. The idea $(\frac{f}{p})$ induces a decreasing filtration $\Fil^\cdot \tmB_{W_\pi}$ such that
\[\tmB_{W_\pi}/\Fil^1\tmB_{W_\pi}\simeq W_\pi. \]
The Frobenius endomorphism on $W$ can be extended to an endomorphism $\varphi$ on $K_0[[T]]$, where $\varphi$ is given by $\varphi(T)=T^p$. Since $\varphi(f)$ is divided by $p$, we have $\varphi(\tmB_{W_\pi})\subset \mB_{W_\pi}$. Thus one gets two restrictions 
\[\varphi:\tmB_{W_\pi}\rightarrow \mB_{W_\pi} \text{ and } \varphi:\mB_{W_\pi}\rightarrow \mB_{W_\pi} .\]
Note that the ideal of $\mB_{W_\pi}$, generated by $\Fil^1\mB_{W_\pi}$ and $T$, is stable under $\varphi$. Then we have the following commutative diagram
\begin{equation}\label{diag:FrobLift}
\xymatrix{
\mB_{W_\pi}\ar[d]^{\varphi} \ar@{->>}[r]  &\mB_{W_\pi}/(\Fil^1\mB_{W_\pi},T)=k \ar[d]^{(\cdot)^p}\\
\mB_{W_\pi} \ar@{->>}[r] &\mB_{W_\pi}/(\Fil^1\mB_{W_\pi},T)=k \\
}.
\end{equation}

\subsection{Base change in the small affine case.}

For a smooth and small $W$-algebra $\R$ (which means there exists an \'etale map  
	\[W[T_1^{\pm1},T_2^{\pm1},\cdots, T_{d}^{\pm1}]\rightarrow R,\]
 see \cite{Fal89}),  Lan-Sheng-Zuo constructed categories $\MIC(\R/p\R)$, $\tMIC(\R/p\R)$, $\MCF(\R/p\R)$ and $\MF(\R/p\R)$. 
A Fontaine-Faltings module over $\R/p\R$ is an object $(V,\nabla,\Fil)$ in $\MCF(\R/p\R)$ together with an isomorphism $\varphi: \widetilde{(V,\nabla,\Fil)}\otimes_{\Phi}\hR\rightarrow (V,\nabla)$ in $\MIC(\R/p\R)$, where $\widetilde{(\cdot)}$ is the Faltings' tilde functor.

We generalize those categories over the $W_\pi$-algebra $R_\pi=R\otimes_W{W_\pi}$. In general, there does not exist Frobenius lifting on the p-adic completion of $\hR_\pi$. We lift the absolute Frobenius map on $R_\pi/\pi R_\pi$ to a map $\Phi:\mBR\rightarrow \mBR$
\begin{equation}\label{diag:FrobLift}
\xymatrix{
 \mBR  \ar[d]^{\Phi} \ar@{->>}[r]  &  R_\pi/\pi R_\pi=\R/p\R \ar[d]^{(\cdot)^p}\\
\mBR \ar@{->>}[r]  &  R_\pi/\pi R_\pi=\R/p\R\\
}
\end{equation}
where $\mBR$ is the $p$-adic completion of $\mB_{W_\pi}\otimes_W \R$. This lifting is compatible with $\varphi:\mB_{W_\pi}\rightarrow \mB_{W_\pi}$.
Denote $\tmBR=\mBR[\frac{f}{p}]$. Then $\Phi$ can be extended to 
\[\Phi:\tmBR\rightarrow \mBR\]
uniquely, which is compatible with $\varphi:\tmB_{W_\pi}\rightarrow \mB_{W_\pi}$. 
%Both $\mBR$ and $\tmBR$ are sub rings of $K_0[[T]]\widehat\otimes_{W} \R$. 
The filtrations on $\mB_{W_\pi}$ and $\tmB_{W_\pi}$ induce filtrations on $\mBR$ and $\tmBR$ respectively, which satisfy
\[\mBR/\Fil^1\mBR\simeq \hR_\pi\simeq \tmBR/\Fil^1\tmBR.\]
\begin{lem}\label{lem:F^iB&F^itB}
Let $n<p$ be a natural number and let $b$ be an element in $F^n\mBR$. Then $\frac{b}{p^n}$ is an element in $F^n\tmBR$.
\end{lem}
\begin{proof} Since the filtrations on $\mBR$ and $\tmBR$ are induced by those on  $\mB_{W_\pi}$ and $\tmB_{W_\pi}$ respectively, we have 
\begin{equation}\label{FnBR}
 F^n\mBR=\left\{\left.\sum_{i\geq n}^{\infty}a_i\frac{f^i}{i!}\right| a_i\in \hR[[T]] \text{ and }a_i \rightarrow 0\right\},
\end{equation}
and
\begin{equation}\label{FntBR}
\begin{split}
F^n\tmBR &=\left\{\left.\sum_{i\geq n}^{\infty}a_i\frac{f^i}{p^i}\right| a_i\in \mBR \text{ and }a_i=0 \text{ for } i\gg0\right\}\\
&=\left\{\left.\sum_{i\geq n}^{\infty}a_i\frac{f^i}{p^i}\right| a_i\in \hR[[T]] \text{ and }\frac{i!}{p^i}a_i\rightarrow 0 \right\}.\\
\end{split}
\end{equation} 
Assume $b=\sum_{i\geq n}a_i\frac{f^i}{i!}$ with $a_i\in \hR[[T]]$ and $a_i\rightarrow 0$, then 
\[\frac{b}{p^n}=\sum_{i\geq n} \frac{p^i a_i }{p^n i!}\cdot \frac{f^i}{p^i},\] 
and the lemma follows.
\end{proof}

Recall that $\mBR$ and $\tmBR$ are  $\hR$-subalgebras of $\hR(\frac{1}{p})[[T]]$. We denote by $\Omega^1_{\mBR}$ (resp. $\Omega^1_{\tmBR}$) the $\mBR$-submodule (resp. $\tmBR$-submodule) of $\Omega^1_{\hR(\frac{1}{p})[[T]]/W}$ generated by elements $\rmd b$, where $b\in \mBR$ (resp. $b\in \tmBR$).  There is a filtration on $\Omega^1_{\mBR}$ (resp. $\Omega^1_{\tmBR}$) given by
$F^n \Omega^1_{\mBR}=F^n \mBR \cdot \Omega^1_{\mBR}$ (resp. $F^n \Omega^1_{\tmBR}=F^n \tmBR \cdot \Omega^1_{\tmBR}$). One gets the following result directly by Lemma~\ref{lem:F^iB&F^itB}. 
\begin{cor}
Let $n<p$ be a natural number. Then $\frac{1}{p^n}F^n \Omega^1_{\mBR}\subset F^n \Omega^1_{\tmBR}$.
\end{cor}

\begin{lem}
The graded pieces $\Gr^n \tmBR$ and $\Gr^n \mBR$ are free $\hR_\pi$-modules of rank $1$ generated by $\frac{f^n}{p^n}$ and $\frac{f^n}{n!}$ respectively.
\end{lem}
\begin{proof}By equation~(\ref{FnBR}), one has $\hR[[T]]\cdot \frac{f^{n+1}}{n!} \subset \hR[[T]]\cdot \frac{f^n}{n!} \bigcap F^{n+1}\mBR$ and 
\begin{equation*}
 \Gr^n \mBR=\frac{\hR[[T]]\cdot \frac{f^n}{n!}}{\hR[[T]]\cdot \frac{f^n}{n!} \bigcap F^{n+1}\mBR}.
\end{equation*} 
On the other hand,
\begin{equation*}
\begin{split}
 \hR[[T]]\cdot \frac{f^n}{n!} \bigcap F^{n+1}\mBR 
&  \subset \hR[[T]]\cdot \frac{f^n}{n!}\bigcap \hR[\frac1p][[T]]\cdot f^{n+1}\\
&  \subset \left(\hR[[T]]\bigcap \hR[\frac1p][[T]]\cdot f\right)\cdot \frac{f^n}{n!}\\
 & =\hR[[T]]\cdot \frac{f^{n+1}}{n!}.\\
\end{split}
\end{equation*}
Then the result for $\Gr^n \mBR$ follows that $\hR_\pi\simeq\hR[[T]]/(f)$. The proof of the result for $\Gr^n \tmBR$ is similar, one only need to replace $n!$ by $p^n$ and to use equation~(\ref{FntBR}).
\end{proof}

We have the following categories 
\begin{itemize}
\item $\MIC(\mBR/p\mBR)$: the category of free $\mBR/p\mBR$-modules with integrable connections.
\item $\tMIC(\tmBR/p\tmBR)$: the category of free $\tmBR/p\tmBR$-modules with integrable nilpotent $p$-connection.
\item $\MCFa(\mBR/p\mBR)$: the category of filtered free $\mBR/p\mBR$-modules equipped with integrable connections, which satisfy the Griffiths transversality, and each of these $\mBR/p\mBR$-modules admits a filtered basis $v_i$ of degree $q_i$, $0\leq q_i\leq a$.
\end{itemize} 
A Fontaine-Faltings module over $\mBR/p\mBR$ of weight $a$ ($0\leq a\leq p-2$) is an object $(V,\nabla,\Fil)$ in the category $\MCFa(\mBR/p\mBR)$ together with an isomorphism in $\MIC(\mBR/p\mBR)$
\[\varphi: \widetilde{(V,\nabla,\Fil)}\otimes_{\Phi}\mBR\rightarrow (V,\nabla),\] 
where $\widetilde{(\cdot)}:\MCF(\mBR/p\mBR) \rightarrow \tMIC(\tmBR/p\tmBR)$ is an analogue of the Faltings' tilde functor.
For an object $(V,\nabla,\Fil)$ in $\MCF(\mBR/p\mBR)$ with filtered basis $v_i$ (of degree $q_i$, $0\leq q_i\leq a$), $\tV$ is defined as a filtered free $\tmBR/p\tmBR$-module 
\[\tV=\bigoplus_i \tmBR/p\tmBR\cdot [v_i]_{q_i}\]
with filtered basis $[v_i]_{q_i}$ (of degree $q_i$, $0\leq q_i\leq a$). Informally one can view $[v_i]_{q_i}$ as ``$\frac{v_i}{p^{q_i}}$". 
Since $\nabla$ satisfies the Griffiths transversality, there are $\omega_{ij}\in F^{q_j-1-q_i}\Omega^1_\mBR$ satisfying
\[\nabla(v_j)=\sum_{i} v_i\otimes \omega_{ij}.\] 
Since $q_j-1-q_i<a\leq p-2$, $\frac{\omega_{ij}}{p^{q_j-1-q_i}}\in F^{q_j-1-q_i}\Omega^1_\tmBR$.
We define a $p$-connection $\tnabla$ on $\tV$  via 
\[\tnabla([v_j]_{q_j})=\sum_{i} [v_i]_{q_i}\otimes \frac{\omega_{ij}}{p^{q_j-1-q_i}}.\]
\begin{lem}
The $\tmBR/p\tmBR$-module $\tV$ equipped with the $p$-connection $\tnabla$ is independent of the choice of the filtered basis $v_i$ up to a canonical isomorphism.   
\end{lem}
\begin{proof}
We write  $v=(v_1,v_2,\cdots)$ and $\omega=(\omega_{ij})_{i,j}$. Then
\[\nabla(v)=v\otimes \omega \text{ and } \tnabla([v])=[v]\otimes (pQ\omega Q^{-1}),\] 
where $Q=\mathrm{diag}(p^{q_1},p^{q_2},\cdots)$ is a diagonal matrix.
Assume that $v_i'$ is another filtered basis (of degree $q_i$, $0\leq q_i\leq a$) and $(\tV',\tnabla')$ is the corresponding $\tmBR/p\tmBR$-module equipped with the $p$-connection. Similarly, we have
\[\nabla(v')=v'\otimes \omega' \text{ and } \tnabla([v'])=[v']\otimes (pQ\omega' Q^{-1}),\] 
Assume $v_j'=\sum_i a_{ij}v_i$ ($a_{ij}\in F^{q_j-q_i}\mBR$). Then $A=(a_{ij})_{i,j}\in \GL_{\rank(V)}(\mBR)$ and 
$QAQ^{-1}=\left(\frac{a_{ij}}{p^{q_j-q_i}}\right)_{i,j}\in \GL_{\rank(V)}(\tmBR)$. 
We construct an isomorphism from $\tV'$ to $\tV$ by
\[\tau([v'])=[v] \cdot (QAQ^{-1}),\]
where $[v]=([v_1]_{q_1},[v_2]_{q_2},\cdots)$ and $[v']=([v_1']_{q_1},[v_2']_{q_2},\cdots)$. Now we only need to check that $\tau$ preserve the $p$-connections. Indeed, 
\begin{equation}
\tau\circ\tnabla'([v']) =[v]\otimes (QAQ^{-1}\cdot p Q\omega' Q^{-1})=[v]\otimes (pQ\cdot A\omega'\cdot Q^{-1})
\end{equation}
and 
\begin{equation}
\begin{split}
\tnabla\circ\tau([v']) & =\tnabla([v]QAQ^{-1})\\
& =[v]\otimes (pQ\omega Q^{-1}\cdot QAQ^{-1} +p\cdot Q \rmd A Q^{-1})\\
&=[v]\otimes (pQ\cdot (\omega A  + \rmd A)\cdot Q^{-1})\\
\end{split}
\end{equation}
Since $\nabla(v')=\nabla(vA)=v\otimes \rmd A + v\otimes \omega A=v'\otimes (A^{-1}\rmd A+ A^{-1}\omega A)$, we have $\omega'=A^{-1}\rmd A+ A^{-1}\omega A$ by definition. Thus $\tau\circ\tnabla'=\tnabla\circ\tau$.

If there are third filtered bases $v''$, one has the following commutative diagram under the isomorphisms constructed:
\begin{equation}
\xymatrix{
(\tV,\tnabla) \ar[r] \ar[rd] &(\tV',\tnabla') \ar[d]\\
& (\tV'',\tnabla'')\\
}
\end{equation}
This can be checked easily by definition. In this sense, the isomorphism constructed is canonical.
\end{proof}
The functor 
\[-\otimes_\Phi\mBR : \tMIC(\tmBR/p\tmBR) \rightarrow \MIC(\mBR/p\mBR)\]
 is induced by base change under $\Phi$. Note that the connection on $(\tV,\tnabla)\otimes_\Phi\mBR$ is given by
\[\rmd + \frac{\rmd \Phi}{p}(\Phi^*\tnabla)\]
Denote by $\MFa(\mBR/p\mBR)$ the category of all Fontaine-Faltings module over $\mBR/p\mBR$ of weight $a$.

Let $(M,\nabla,\Fil,\Psi)$ be an object defined in Definition 2 of \cite{Fal99}. Then we can construct an Fontaine-Faltings module over $B_{W_\pi}/pB_{W_\pi}$ as follows. Denote
\begin{itemize}
	\item $V:=M/pM$;
	\item $\overline{\nabla}=\nabla\pmod{p}$;
	\item $\overline{\Fil}^iV=(\Fil^iM+pM)/pM$.
\end{itemize}
By i) and ii) in Definition 2 of \cite{Fal99}, one gets 
\[(V,\overline{\nabla},\overline{\Fil})\in \MCFa(\mB_{W\pi}/p\mB_{W_\pi}).\]
Assume $\{m_i\}$ is a filtered basis of $M$ with filtered degree $q_i$. Then $v_i=m_i+pM\in V$ forms a filtered basis of $V$ with filtered degree $q_i$. By the definition of tilde functor we have
\[\tV=\bigoplus_i \tmB_{W_\pi}/p\tmB_{W_\pi}\cdot [v_i]_{q_i}.\]
Now we can construct a $\mB_{W_{\pi}}$-morphism
\[\varphi:\tV\otimes_{\Phi}\mB_{W_{\pi}} \rightarrow V,\]
by giving
\[\varphi([v_i]_{q_i}\otimes_{\Phi}1)=\frac{\Psi(m_i)}{p^{q_i}}\pmod{p}.\]
Since $\Psi$ is a $\nabla$-horizontal semilinear endomorphism and $\frac{\Psi(m_i)}{p^{q_i}}$ forms a new $R_{W_{\pi}}$-basis of $M$, the morphism $\varphi$ is actually an isomorphism of modules with connections. Thus we get a Fontaine-Faltings module
\[(V,\overline{\nabla},\overline{\Fil},\varphi)\in \MFa(\mB_{W_\pi}/p\mB_{W_\pi}).\]
Replacing every $W_\pi$ by $R_\pi$, one gets a functor from the category of Fontaine modules defined in \cite{Fal99} to the category $\MFa(\mBR/p\mBR)$. In this sense the Fontaine-Faltings modules we defined above is compatible with the Fontaine modules defined in \cite{Fal99}.

\begin{lem}\label{lem:coeffExt} We have the following commutative diagram by extending the coefficient ring from $\R$ to $\mBR$ (or $\tmBR$)
\begin{equation*}
\xymatrix@C=1cm{
\MCF(\R/p\R) \ar[r]^{\widetilde{(\cdot)}}  \ar[d]^{-\otimes_R \mBR} 
& \tMIC(\R/p\R)\ar[r]^{-\otimes_{\Phi} {\R}}  \ar[d]^{-\otimes_R \tmBR} 
& \MIC(\R/p\R)  \ar[d]^{-\otimes_R \mBR} \\
\MCF(\mBR/p\mBR) \ar[r]^{\widetilde{(\cdot)}}  %\ar[d]^{\mod (\pi,\Fil^1\mB_R)}  
& \tMIC(\tmBR/p\tmBR)\ar[r]^{-\otimes_\Phi\mBR}  %\ar[dd]^{\mod (\pi,\Fil^1\tmB_R)}  
& \MIC(\mBR/p\mBR) \\ %\ar[dd]^{\mod (\pi,\Fil^1\mB_R)}  \\
%\MIC(\R/p\R) &&\\ 
%\MCF(\R/p\R)\ar[u] \ar[r]^{\widetilde{(\cdot)}} 
%& \tMIC(\R/p\R)\ar[r]^{-\otimes_{\Phi} {\R}} 
%& \MIC(\R/p\R) \\
}
\end{equation*} 
In particular, we get a functor from the category of Fontaine-Faltings modules over $R/pR$ to that over $\mBR/p\mBR$  
\[\MFa(\R/p\R)\rightarrow \MFa(\mBR/p\mBR).\] 
%And $\mod (\pi,\Fil^1\mB_R)$ restricts to objects in $\MFa(\mBR/p\mBR)$ of rank $1$.  
\end{lem}
Those categories of Fontaine-Faltings modules are independent of the choice of the Frobenius lifting by the Taylor formula.
\begin{thm}\label{thm:gluingFFmod} For any two choices of $\Phi_{\mBR}$ there is an equivalence between the corresponding categories $\MFa(\mBR/p\mBR)$ with different $\Phi_{\mBR}$. These equivalences satisfy the obvious cocycle condition. Therefore, $\MFa(\mBR/p\mBR)$ is independent of the choice of $\Phi_{\mBR}$ up to a canonical isomorphism.
\end{thm}

\begin{defi}\label{lem:defiFunctorD}
For an object $(V,\nabla,\Fil,\varphi)$ in $\MFa(\mBR/p\mBR)$, denote
\[\bD(V,\nabla,\Fil,\varphi)=\Hom_{B^+(R),\Fil,\varphi}(V\otimes_{\mBR} B^+(R),B^+(R)/pB^+(R)).\]  
\end{defi}
The proof of Theorem~2.6 in~\cite{Fal89} works in this context. we can define an adjoint functor $\bE$ of $\bD$ as 
\[\bE(L)= \varinjlim \{V\in \MFa(\mBR/p\mBR)\mid L\rightarrow \bD(V)\}.\]
The proof in page~41 of~\cite{Fal89} still works. Thus we obtain:
\begin{thm} $i).$ The homomorphism set $\bD(V,\nabla,\Fil,\varphi)$ is an $\bF_p$-vector space with a linear $\mathrm{Gal}(\overline{R}_K/R_K)$-action whose rank equals to the rank of $V$.\\
$ii).$ The functor $\bD$ from $\MFa(\mBR/p\mBR)$ to the category of $W_n(\bF_p)$-$\mathrm{Gal}(\overline{R}_K/R_K)$-modules is fully faithful and its image on objects is closed under subobjects and quotients. 
\end{thm}

\subsection{ Categories and Functors on proper smooth variety over very ramified valuation ring $W_\pi$} 
Let $\mX$ be a smooth proper $W$-scheme and $\mX_\pi=\mX\otimes_{W}W_\pi$.
Let $\sX_\pi$ be the formal completion of $\mX\otimes_{W}\mB_{W_\pi}$ and $\tsX_\pi$ be the formal completion of $\mX\otimes_W{\tmB_{W_\pi}}$. Then $\sX_\pi$ is an infinitesimal thickening of $\mX_\pi$ and the ideal defining $\mX_\pi$ in $\sX_\pi$ has a nilpotent PD-structure which is compatible with that on $F^1(\mB_{W_\pi})$ and $(p)$
\begin{equation}
\xymatrix@R=2mm{
\mX_\pi \ar[rr] \ar[dd] \ar[dr] && \sX_\pi \ar[dd]|(0.5)\hole \ar[dr]&\\
& \widetilde{\sX_\pi}  \ar[rr] \ar[dd] && \mX \ar[dd]\\
\Spec W_\pi  \ar[rr]|(0.5)\hole
 \ar[dr] && \Spec \mB_{W_\pi} \ar[dr] &\\
& \Spec\widetilde{\mB}_\pi \ar[rr] && \Spec W \ .\\
}
\end{equation}
 Let $\{\mU_{i}\}_i$ be a covering of small affine open subsets of $\mX$. By base change, we get a covering $\{\sU_i=\mU_{i}\times_\mX \sX_\pi\}_i$ of $\sX_\pi$ and a covering $\{\widetilde{\sU}_i=\mU_{i}\times_\mX \tsX_\pi\}_i$ of $\widetilde{\sX_\pi}$. For each $i$, we denote $R_{i}=\mO_{\mX_\pi}(\mU_{i}\times_\mX \mX_\pi)$. Then $\mB_{R_i}=\mathcal O_{\sX_\pi}(\sU_i)$ and 
$\tmB_{R_i}=\mathcal O_{\widetilde{\sX_\pi}}(\widetilde{\sU}_i)$ are the coordinate rings. Fix a Frobenius-lifting $\Phi_i:\tmB_{R_i}\rightarrow \mB_{R_i}$, one gets those categories of Fontaine-Faltings modules
\[\MFa(\mB_{R_i}/p\mB_{R_i}).\]
By the Theorem~\ref{thm:gluingFFmod}, these categories are glued into one category. Moreover those underlying modules are glued into a bundle over $\sX_{\pi,1}=\sX_\pi\otimes_{\bZ_p}\bF_p$. We denote this category by $\MFa(\sX_{\pi,1})$. 

\subsubsection{Inverse Cartier functor and a description of $\MFa(\sX_{\pi,1})$ via Inverse Cartier functor} Let  $\overline{\Phi}: \mBR/p\mBR \rightarrow \mBR/p\mBR$ be the $p$-th power map. Then we get the following lemma directly.
\begin{lem}\label{lem:FrobLift}
Let $\Phi:\mBR\rightarrow \mBR$ and $\Psi:\mBR\rightarrow \mBR$ be two liftings of $\overline{\Phi}$ which are both compatible with the Frobenius map on $\mB_{W_\pi}$. 
\begin{itemize}
\item[i).] Since $\varphi(f)$ is divided by $p$, we extend $\Phi$ and $\Psi$ to maps on $\tmBR$ via $\frac{f^n}{p}\mapsto \left(\frac{\varphi(f)}{p^n}\right)^n$  uniquely;
\item[ii).] the difference $\Phi-\Psi$ on $\tmBR$ is still divided by $p$;
\item[iii).] the differentials $\rmd \Phi : \Omega^1_{\tmBR}\rightarrow \Omega^1_{\mBR}$ and $ \rmd \Psi : \Omega^1_{\tmBR}\rightarrow \Omega^1_{\mBR}$ are divided by $p$.
\end{itemize}
\end{lem} 
From now on, we call the extension given by i) of Lemma~\ref{lem:FrobLift} the \emph{Frobenius liftings} of $\overline{\Phi}$ on $\tmBR$.

\begin{lem}Let $\Phi:\tmBR\rightarrow \mBR$ and $\Psi:\tmBR\rightarrow \mBR$ be two Frobenius liftings of $\overline{\Phi}$ on $\tmBR$. Then there exists a $\mBR/p\mBR$-linear morphism \[h_{\Phi,\Psi}:\Omega^1_{\tmBR/p\tmBR}\otimes_{\overline\Phi}\mBR/p\mBR \rightarrow  \mBR/p\mBR\]
 satisfying that:
\begin{itemize}
\item[i).] we have $\frac{\rmd \Phi}{p}-\frac{\rmd \Psi}{p}=\rmd h_{\Phi,\Psi}$ over $\Omega^1_{\tmBR/p\tmBR}\otimes_{\overline\Phi}1$;
\item[ii).] the cocycle condition holds.
\end{itemize}
\end{lem}
\begin{proof}
As $\Omega^1_{\tmBR/p\tmBR}\otimes_{\overline\Phi}\mBR/p\mBR$ is an $\mBR/p\mBR$-module generated by elements of the form $\rmd g\otimes 1$ ($g\in\tmBR/p\tmBR$) with relations $\rmd(g_1+g_2)\otimes 1-\rmd g_1 \otimes 1-\rmd g_2 \otimes 1$ and $\rmd(g_1g_2)\otimes 1-\rmd g_1 \otimes\overline{\Phi}(g_2)-\rmd g_2 \otimes \overline{\Phi}(g_1)$.
Since $\Phi-\Psi$ is divided by $p$, we can denote $h_{ij}({\rmd g}\otimes 1)=\frac{\Phi(\hat{g})-\Psi(\hat{g})}{p}\pmod{p}\in \mBR/p\mBR$ for any element $g\in \mO_{U_1}$ (the definition does not depend on the choice of the lifting $\hat{g}$ of $g$ in $\mO_{\mU}$). By direct calculation, we have
\[h_{ij}(\rmd(g_1+g_2)\otimes 1)=h_{ij}(\rmd g_1 \otimes 1)+h_{ij}(\rmd g_2 \otimes 1)\]
and
\[h_{ij}(\rmd(g_1g_2)\otimes 1)=\overline{\Phi}(g_2)\cdot h_{ij}(\rmd g_1 \otimes 1)+\overline{\Phi}(g_1)\cdot h_{ij}(\rmd g_2 \otimes 1)\]
Thus $h_{ij}$ can be $\mBR/p\mBR$-linearly extended.
One checks i) and ii) directly by definition.
\end{proof}

Let $(\tV,\tnabla)$ be a locally filtered free sheaf over $\tsX_{\pi,1}=\tsX_\pi\otimes_{\bZ_p}\bF_p$ with an integrable $p$-connection. Here a ``filtered free'' module over a filtered ring $R$ is a direct sum of copies of $R$ with the filtration shifted by a constant amount. The associated graded then has a basis over $gr_F(R)$ consisting of homogeneous elements(see \cite{Fal99}).  Let $(\tV_i,\tnabla_i)=(\tV,\tnabla)\mid_{\widetilde{\sU}_{i,1}}$ be its restriction on the open subset $\widetilde{\sU}_{i,1}=\widetilde{\sU}_{i}\otimes_{\bZ_p}\bF_p$. By taking functor $\Phi_i^*$, we get bundles with integrable connections over $\sU_{i,1}=\sU_i\otimes_{\bZ_p}\bF_p$
\[\left(\Phi_i^*\tV_i, \rmd + \frac{\rmd \Phi_i}{p}(\Phi_i^*\tnabla)\right).\]
\begin{lem}\label{gluingFlatBundle} Let $(\tV,\tnabla)$ be a locally filtered free sheaf over $\tsX_{\pi,1}$ with an integrable $p$-connection. Then these local bundles with connections
	\[\left(\Phi_i^*\tV_i, \rmd + \frac{\rmd \Phi_i}{p}(\Phi_i^*\tnabla)\right)\]
can be glued into a global bundle with a connection on $\sX_{\pi,1}$ via transition functions
\[G_{ij}=\exp\left(h_{\Phi_i,\Phi_j}(\overline{\Phi}^*\tnabla)\right):\Phi_{i}^*(\tV_{ij}) \rightarrow \Phi_{j}^*(\tV_{ij}).\] 
Denote this global bundle with connection by $C^{-1}_{\sX_{\pi,1}}(\tV,\tnabla)$. Then we can construct a functor 
\[C^{-1}_{\sX_{\pi,1}}: \tMIC(\tsX_{\pi,1})\rightarrow \MIC(\sX_{\pi,1}).\] 
\end{lem}
\begin{proof}The cocycle condition easily follows from the integrability of the Higgs field. We show that the local connections coincide on the overlaps, that is
\[\left(G_{ij}\otimes \id\right)
\circ
 \left(\rmd + \frac{\rmd \Phi_i}{p}(\Phi_i^*\tnabla)\right)
=\left(\rmd + \frac{\rmd \Phi_j}{p}(\Phi_j^*\tnabla)\right) \circ G_{ij}. \]
It suffices to show
\[\frac{\rmd \Phi_i}{p}(\Phi_i^*\tnabla)
=
G^{-1}_{ij}\circ \rmd G_{ij}
+
G^{-1}_{ij}\circ \frac{\rmd \Phi_j}{p}(\Phi_j^*\tnabla) \circ G_{ij}.\]
Since $G^{-1}_{ij}\circ \rmd G_{ij}=\rmd h_{\Phi_i,\Phi_j}(\overline{\Phi}^*\tnabla)$ and $G_{ij}$ commutes with $\frac{\rmd \Phi_j}{p}(\Phi_j^*\tnabla)$ we have
\begin{equation*}
\begin{split}
G^{-1}_{ij}\circ \rmd G_{ij}+G^{-1}_{ij}\circ \frac{\rmd \Phi_j}{p}(\Phi_j^*\tnabla) \circ G_{ij} 
& =\rmd h_{\Phi_i,\Phi_j}(\overline{\Phi}^*\tnabla)
+\frac{\rmd \Phi_j}{p}(\Phi_j^*\tnabla)\\
& =\frac{\rmd \Phi_i}{p}(\Phi_i^*\tnabla)
\end{split}
\end{equation*}
by the integrability of the Higgs field.
Thus we glue those local bundles with connections into a global bundle with connection via $G_{ij}$.
\end{proof}
%%%%%%%%%%%%%%%%%%%%%%%%%%%%%%%%%%%%%%%%
%%%%%%%%%%filtration degree%%%%%%%%%%%%%
%%%%%%%%%%%%%%%%%%%%%%%%%%%%%%%%%%%%%%%%
\begin{lem} To give an object in the category $\MF(\sX_{\pi,1})$ is equivalent to give a tuple $(V,\nabla,\Fil,\phi)$ satisfying
\begin{itemize}
\item[i).] $V$ is filtered local free sheaf over $\sX_{\pi,1}$ with local basis having filtration degrees contained in $[0,a]$;
\item[ii).] $\nabla:V\rightarrow V\otimes_{\mO_{\sX_{\pi,1}}} \Omega^1_{\sX_{\pi,1}}$ is an integrable connection satisfying the Griffiths transversality;
\item[iii).] $ \varphi:C_{\sX_{\pi,1}}^{-1}\widetilde{(V,\nabla,\Fil)}\simeq (V,\nabla)$ is an isomorphism of sheaves with connections over $\sX_{\pi,1}$.
\end{itemize}
\end{lem}

\subsubsection{The functors $\bD$ and $\bD^P$}
For an object in $\MFa(\sX_{\pi,1})$, we get locally constant sheaves on $\mU_K$ by applying the local $\bD$-functors. These locally constant sheaves can be expressed in terms of certain finite \'etale coverings. They can be glued into a finite covering of $\mX_{\pi,K}=\mX_K$. We have the following result.   
\begin{thm}
Suppose that $X$ is a proper smooth and geometrically connected scheme over $W$. Then there exists a fully faithful contravariant functor $\bD$ from $\MFa(\sX_{\pi,1})$ to the category of $\bF_p$-representations of $\pi^\text{\'et}_1(\mX_K)$. The image of $\bD$ on objects is closed under subobjects and quotients. Locally $\bD$ is given by the same as in Lemma~\ref{lem:defiFunctorD}. 
\end{thm}

Again one can define the category $\MFa(\sX_{\pi,1}^o)$ in the logarithmic case, if one replaces all "connections" by "logarithmic connections" and "Frobenius lifting" by "logarithmic Frobenius lifting". We also have the version of $\MF_{[0,a],f}(\sX_{\pi,1}^o)$ with endomorphism structures of $\bF_{p^f}$, which is similar as the \emph{Variant 2} discussed in section $2$ of\cite{LSZ13a}. And the twisted versions $\TMF_{[0,a],f}(\sX_{\pi,1}^o)$ can also be defined on $\sX_{\pi,1}$ in a similar way as before. More precisely, let $\mL$ be a line bundle over $\sX_{\pi,1}$. The $\mL$-twisted Fontaine-Faltings module is defined as follows.
\begin{defi}
An $\mL$-twisted Fontaine-Faltings module over $\sX_{\pi,1}$ with endomorphism structure is a tuple \[((V,\nabla,\Fil)_0,(V,\nabla,\Fil)_1,\cdots,(V,\nabla,\Fil)_{f-1},\varphi_\cdot)\]
where $(V,\nabla,\Fil)_i$ are objects in $\MCF(\sX_{\pi,1}^o)$ equipped with isomorphisms in $\MIC(\sX_{\pi,1}^o)$
\[\varphi_i:C_{\sX_{\pi,1}}^{-1}\widetilde{(V,\nabla,\Fil)}_i\simeq (V,\nabla)_{i+1} \text{ for } i=0,1,\cdots,f-2;\]
and 
\[\varphi_{f-1}:C_{\sX_{\pi,1}}^{-1}\widetilde{(V,\nabla,\Fil)}_{f-1} \otimes (\mL^{p},\nabla_\can)\simeq (V,\nabla)_0.\] 
\end{defi}
 The proof of Theorem~\ref{ConsFunc:D^P} works in this context. Thus we obtain the following result.
\begin{thm}\label{thm:functorD^P}
Suppose that $\mX$ is a proper smooth and geometrically connected scheme over $W$ equipped with a smooth log structure $\mD/W(k)$. Suppose that the residue field $k$ contains $\bF_{p^f}$. Then there exists an exact and fully faithful contravariant functor $\bD^P$ from $\TMF_{a,f}(\sX_{\pi,1}^o)$ to the category of projective $\bF_{p^f}$-representations of $\pi^\text{\'et}_1(\mX_{K}^o)$. The image of $\bD^p$  is closed under subobjects and quotients.
\end{thm}

Recall that $\{\sU_i\}_i$ is an open covering of $\sX$. A line bundle on $\sX$ can be expressed by the transition functions on $\sU_{ij}$.
\begin{lem}
Let $\mL$ be a line bundle on $\sX_{\pi,1}$ expressed by $(g_{ij})$. Denote by $\widetilde{\mL}$ the line bundle on $\tsX_{\pi,1}$ defined by the same transition functions $(g_{ij})$. Then one has 
\[C^{-1}_{\sX_{\pi,1}}(\widetilde{\mL},0)=\mL^p.\]
\end{lem}
\begin{proof}
Since $g_{ij}$ is an element in $\mB_{R_{ij}} \subset \tmB_{R_{ij}}$, by diagram~(\ref{diag:FrobLift}), one has
\[\Phi(g_{ij})\equiv g_{ij}^p \pmod{p}.\] 
On the other hand, since the $p$-connection is trivial, one has 
 \[C^{-1}_{\sX_{\pi,1}}(\widetilde{\mL},0)=(\Phi\,\mathrm{mod}\,p)^*(\widetilde{\mL}).\]
Thus one has $C^{-1}_{\sX_{\pi,1}}(\widetilde{\mL},0)=(\mO_{\widetilde{\sU}_{i,1}},g_{ij}^p)=\mL^p$.
\end{proof}

In a similar way, one can define the Higgs-de Rham flow on $\sX_{\pi,1}$ as a sequence consisting of infinitely many alternating terms of Higgs bundles over $\tsX_{\pi,1}$ and filtered de Rham bundles over $\sX_{\pi,1}$
	\[\left\{ (E,\theta)_{0},  
	(V,\nabla,\Fil)_{0},
	(E,\theta)_{1},  
	(V,\nabla,\Fil)_{1},
	\cdots\right\}\]
with $(V,\nabla)_i=C^{-1}_{\sX_{\pi,1}}((E,\theta)_i)$ and $(E,\theta)_{i+1}=\widetilde{(V,\nabla,\Fil)_i}$ for all $i\geq 0$.

 \emph{$f$-periodic  $\mL$-twisted Higgs-de Rham flow} over $\sX_{\pi,1}$ of level in $[0,a]$ is a Higgs-de Rham flow over $\sX_{\pi,1}$
	\[\left\{ (E,\theta)_{0},  
	(V,\nabla,\Fil)_{0},
	(E,\theta)_{1},  
	(V,\nabla,\Fil)_{1},
	\cdots\right\}\]
equipped with isomorphisms $\phi_{f+i}:(E,\theta)_{f+i}\otimes (\widetilde{\mL}^{p^i},0)\rightarrow (E,\theta)_i$ of Higgs bundles for all $i\geq0$ 
\begin{equation*}
\tiny\xymatrix@W=2cm@C=-13mm{
	&\left(V,\nabla,\Fil\right)_{0} \ar[dr]^{\widetilde{(\cdot)}}
	&&\left(V,\nabla,\Fil\right)_{1}\ar[dr]^{\widetilde{(\cdot)}}
	&&\cdots \ar[dr]^{\widetilde{(\cdot)}}
	&&\left(V,\nabla,\Fil\right)_{f}\ar[dr]^{\widetilde{(\cdot)}}  %\ar@/_20pt/[llllll]_{\mathcal C^{-1}_{\sX_{\pi,1}}(\phi_{f})}
	&&\left(V,\nabla,\Fil\right)_{f+1}\ar[dr]^{\widetilde{(\cdot)}} %\ar@/_20pt/[llllll]_{\mathcal C^{-1}_{\sX_{\pi,1}}(\phi_{f+1})}
	&&\cdots\\%\ar@/_20pt/[llllll]_{\cdots}\\
	\left(E,\theta\right)_{0}\ar[ur]_{\mathcal C^{-1}_{\sX_{\pi,1}}}
	&&\left(E,\theta\right)_{1}\ar[ur]_{\mathcal C^{-1}_{\sX_{\pi,1}}}
	&& \cdots \ar[ur]_{\mathcal C^{-1}_{\sX_{\pi,1}}}
	&&\left(E,\theta\right)_{f}\ar[ur]_{\mathcal C^{-1}_{\sX_{\pi,1}}} \ar@/^20pt/[llllll]^{\phi_f}|(0.33)\hole
	&&\left(E,\theta\right)_{f+1}\ar[ur]_{\mathcal C^{-1}_{\sX_{\pi,1}}} \ar@/^20pt/[llllll]^{\phi_{f+1}}|(0.33)\hole
	&& \cdots \ar@/^20pt/[llllll]^{\cdots}\ar[ur]_{\mathcal C^{-1}_{\sX_{\pi,1}}}\\} 
\end{equation*} 
And for any $i\geq 0$ the isomorphism
\begin{equation*}
 C^{-1}_{\sX_{\pi,1}}(\phi_{f+i}): (V,\nabla)_{f+i}\otimes (\mL^{p^{i+1}},\nabla_{\mathrm{can}})\rightarrow (V,\nabla)_{i},
\end{equation*} 
		strictly respects filtrations $\Fil_{f+i}$ and $\Fil_{i}$. Those $\phi_{f+i}$'s are related to each other by formula
\[\phi_{f+i+1}=\mathrm{Gr}\circ C^{-1}_{\sX_{\pi,1}}(\phi_{f+i}).\]
Just taking the same construction as before, we obtain the following result.
\begin{thm}\label{thm:equivalent}
There exists an equivalent functor $\IC_{\sX_{\pi,1}}$ from the category of twisted periodic Higgs-de Rham flows over $\sX_{\pi,1}$ to the category of twisted Fontaine-Faltings modules over $\sX_{\pi,1}$ with a commutative diagram 
\begin{equation}
\xymatrix{
\THDF(X_1) \ar[r]^{\IC_{X_1 }}  \ar[d]_{-\otimes_{\mO_{X_{1}}}\mO_{\sX_{\pi,1}}} & \TMF(X_1)  \ar[d]^{-\otimes_{\mO_{X_{1}}}\mO_{\sX_{\pi,1}} }\\
\THDF(\sX_{\pi,1}) \ar[r]^{\IC_{\sX_{\pi,1}}} & \TMF(\sX_{\pi,1})\ .\\
}
\end{equation}
\end{thm}

\subsection{degree and slope}
Recall that $\sX_\pi$ is a smooth formal scheme over $\mB_{W_\pi}$. Then $\sX_{\pi,1}$ and $X_1$ are the modulo-$p$ reductions of $\sX_\pi$ and $X$ respectively. Also note that $X_1$ is the closed fiber of $Y_1=\mX_\pi\otimes_{\bZ_p}\bF_p$, $\sX_{\pi,1}=\sX_\pi\otimes_{\bZ_p}\bF_p$, $\mX_\pi$ and $\sX_\pi$.
\begin{equation}
\xymatrix@R=2mm{
X_1\ar[r] \ar[dd] &\mX_{\pi,1} \ar[rr] \ar[dd] \ar[dr] && \sX_{\pi,1} \ar[dd]|(0.5)\hole \ar[dr]&\\
&& \widetilde{\sX_\pi}_1  \ar[rr] \ar[dd] && X_1 \ar[dd]\\
X_1\ar[r] \ar[dd]  & \mX_\pi \ar[dr] \ar[dd]  \ar[rr]|(0.5)\hole && \sX_\pi \ar[dr]\ar[dd]|(0.5)\hole &\\
&& \widetilde{\sX_\pi}  \ar[rr] \ar[dd]  && \mX \ar[dd]  \\
\Spec k \ar[r] & \Spec W_\pi \ar[drrr]|(0.34)\hole \ar[rr]|(0.5)\hole
 \ar[dr] && \Spec \mB_{W_\pi} \ar[dr] &\\
&& \Spec\widetilde{\mB}_\pi \ar[rr] && \Spec W\\
}
\end{equation}
For a line bundle $V$ on $\sX_{\pi,1}$ (resp. $\tsX_{\pi,1}$), $V\otimes_{\mO_{\sX_{\pi,1}}}\mO_{X_1}$ (resp. $V\otimes_{\mO_{\tsX_{\pi,1}}}\mO_{X_k}$) forms a line bundle on the special fiber $X_1$ of $\mX$. We denote
\[\deg(V):=\deg(V\otimes_{\mO_{\sX_{\pi,1}}}\mO_{X_1}).\]
For any bundle $V$ on $\sX_{\pi,1}$ (resp. $\tsX_{\pi,1}$) of rank $r>1$, we denote
\[\deg(V):=\deg(\bigwedge_{i=1}^{r}V).\]

By Lemma~\ref{lem:FrobLift}, the modulo-$p$ reduction of the Frobenius lifting is globally well-defined. We denote it by $\Phi_1:\tsX_{\pi,1}\rightarrow \sX_{\pi,1}$. Since $\tsX_{\pi,1}$ and $\sX_{\pi,1}$ have the same closed subset $X_1$, we have the following diagram 
\begin{equation}
\xymatrix{
X_1 \ar[r]^{\widetilde{\tau}} \ar[d]^{\Phi_{X_1}} & \tsX_{\pi,1} \ar[d]^{\Phi_1}\\
X_1 \ar[r]^{\tau} & \sX_{\pi,1}\\
}
\end{equation}
Here $\tau$ and $\widetilde{\tau}$ are closed embeddings and $\Phi_{X_1}$ is the absolute Frobenius lifting on $X_1$. We should remark that the diagram above is not commutative, because $\Phi_1$ does not preserve the defining ideal of $X_1$.

\begin{lem}\label{lem:IsoPullbacks} Let $(V,\nabla,\Fil)$ be an object in $\MCF(\sX_{\pi,1})$ of rank $1$. Then there is an isomorphism 
\[\Phi_{X_1}^*\circ\widetilde{\tau}^*(\tV)\overset{\sim}{\longrightarrow} \tau^*\circ\Phi_1^*(\tV).\]
\end{lem}
\begin{proof}Recall that $\{\sU_i\}_i$ is an open covering of $\sX$. We express the line bundle $V$ by the transition functions $(g_{ij})$, where $g_{ij}\in \left(\mB_{R_{ij}}/p\mB_{R_{ij}}\right)^\times$. Since $V$ is of rank $1$, the filtration $\Fil$ is trivial. Then 
by definition $\tV$ can also be expressed by $(g_{ij})$. Since $g_{ij}\in\mB_{R_{ij}}/p\mB_{R_{ij}}$, one has
\[(\Phi_{X_1}\mid_{U_{i,1}})^*\circ(\widetilde{\tau}\mid_{\widetilde{\sU}_{i,1}})^*(g_{ij}) = (\tau\mid_{\sU_{i,1}})^*\circ(\Phi_1\mid_{\sU_{i,1}})^*(g_{ij}),\]
by diagram~(\ref{diag:FrobLift}). This gives us the isomorphism $\Phi_{X_1}^*\circ\widetilde{\tau}^*(\tV)\overset{\sim}{\longrightarrow} \tau^*\circ\Phi_1^*(\tV)$.
\end{proof}

\begin{lem}\label{lem:deg&C^-1} Let $(V,\nabla,\Fil)$ be an object in $\MCF(\sX_{\pi,1})$. Then we have
\[\deg(\tV)=\deg(V) \text{ and } \deg(C^{-1}_{\sX_{\pi,1}}(\tV))=p\deg(\tV).\]
\end{lem}
\begin{proof}
Since the tilde functor and inverse Cartier functor preserve the wedge product and the degree of a bundle is defined to be that of its determinant, we only need to consider the rank $1$ case. 
Now let $(V,\nabla,\Fil)$ be of rank $1$. The reductions of $V$ and $\tV$ on the closed fiber $X_1$ are the same, by the proof of Lemma~\ref{lem:IsoPullbacks}. Then we have 
\[\deg(\tV)=\deg(V).\]
Since the filtration is trivial, the $p$-connection $\tnabla$ is also trivial. In this case, the transition functions $G_{ij}$ in Lemma~\ref{gluingFlatBundle} are identities. Thus 
\[C^{-1}_{\sX_{\pi,1}}(\tV)=\Phi_1^*(\tV).\]
Recall that $\deg(\Phi_1^*(\tV))=\deg(\tau^*\circ\Phi_1^*(\tV))$ and $\deg(\tV)=\deg(\widetilde{\tau}^*(\tV))$. Lemma~\ref{lem:IsoPullbacks} implies $\deg(\tau^*\circ\Phi_1^*(\tV))=\deg(\Phi_{X_1}^*\circ\widetilde{\tau}^*(\tV))$. Since $\Phi_{X_1}$ is the absolute Frobenius, one has $\deg(\Phi_{X_1}^*\circ\widetilde{\tau}^*(\tV))=p\deg(\widetilde{\tau}^*(\tV))$. Composing above equalities, we get $\deg(C^{-1}_{\sX_{\pi,1}}(\tV))=p\deg(\tV)$.
\end{proof}

\begin{thm}\label{ramified_Thm}
Let	$\mE=\left\{ (E,\theta)_{0},  
(V,\nabla,\Fil)_{0},
(E,\theta)_{1},  
(V,\nabla,\Fil)_{1},
\cdots\right\}$ be an $L$-twisted $f$-periodic Higgs-de Rham flow with endomorphism structure and log structure over $X_{1}$. Suppose that the degree and rank of the initial term $E_0$ are coprime. Then the projective representation $\bD^P\circ\IC_{\sX_{\pi,1}}(\mE)$ of $\pi^\text{\'et}_1(X_{K_0}^o)$ is still irreducible after restricting to the geometric fundamental group $\pi^\text{\'et}_1(X^o_{\overline{K}_0})$, where $K_0=W[\frac1p]$.
\end{thm}
\begin{proof}
Let $\rho:\pi^\text{\'et}_1(X_{K_0}^o)\rightarrow \mathrm{PGL}(\bD^P\circ\IC_{\sX_{\pi,1}}(\mE))$ be the projective representation. Fix a $K_0$-point in $X_{K_0}$, which induces a section $s$ of the surjective map $\pi^\text{\'et}_1(X_{K_0}^o)\rightarrow \mathrm{Gal}(\overline{K}_0/K_0)$. We restrict $\rho$ on $\mathrm{Gal}(\overline{K}_0/K_0)$ by this section $s$. Since the module $\bD^P\circ\IC_{\sX_{\pi,1}}(\mE)$ is finite, the image of this restriction is finite. And there is a finite field extension $K/K_0$ such that the restriction of $\rho$ by $s$ on $\mathrm{Gal}(\overline{K}_0/K)$ is trivial. Thus 
\[\rho(\pi^\text{\'et}_1(X_{K}^o))=\rho(\pi^\text{\'et}_1(X^o_{\overline{K}_0})).\]
It is sufficient to show that the restriction of $\rho$ on $\pi^\text{\'et}_1(X_{K}^o)$ is irreducible. Suppose that the restriction of $\bD^P\circ\IC_{X_1}(\mE)$  on $\pi^\text{\'et}_1(X_{K}^o)$ is not irreducible.
Since the functors $\bD^P$ and $C^{-1}_{\sX_{\pi,1}}$ are compatible with those over $X_1$, the projective representation 
$\bD^P\circ \IC_{\sX_{\pi,1}}(\mE\otimes_{\mO_{X_1}}\mO_{\sX_{\pi,1}})=\bD^P\circ\IC_{X_1}(\mE)$ is also not irreducible. Thus there exists a non-trivial quotient, which is the image of some nontrivial sub $\mL=L\otimes_{\mO_{X_1}}\mO_{\sX_{\pi,1}}$-twisted $f$-periodic Higgs-de Rham flow of $\mE\otimes_{\mO_{X_1}}\mO_{\sX_{\pi,1}}$
\[\left\{ (E',\theta')_{0},  
(V',\nabla',\Fil')_{0},
(E',\theta')_{1},  
(V',\nabla',\Fil')_{1},
\cdots\right\},\] 
 under the functor $\bD^P\circ\IC_{\sX_{\pi,1}}$ according to Theorem~\ref{thm:functorD^P} and Theorem~\ref{thm:equivalent}. 
Since $E'_0$ is a sub bundle of $E_0\otimes_{\mO_{X_1}}\mO_{\sX_{\pi,1}}$, we have $1\leq \mathrm{rank}(E'_0) <\mathrm{rank}(E_0)$. 
By Theorem~4.17 in~\cite{OgVo07}, $\deg(E_{i+1})=p\deg(E_i) \text{ for }i\geq 0$
and 
$\deg(E_{0})=p\deg(E_{f-1})+\mathrm{rank}(E_{0})\times \deg(L)$.
Thus 
\begin{equation}
\frac{\deg (E_0)}{\mathrm{rank} (E_0)}=\frac{\deg (L)}{1-p^f}.
\end{equation}
Similarly, by Lemma~\ref{lem:deg&C^-1}, one gets 
\[\frac{\deg (E'_0)}{\mathrm{rank} (E'_0)}=\frac{\deg (\mL)}{1-p^f}.\]
Since $\deg(\mL)=\deg(L)$, one has $\deg(E_0)\cdot\mathrm{rank}(E_0')=\deg(E_0')\cdot\mathrm{rank}(E_0)$.
Since $\deg (E_0)$ and $\mathrm{rank} (E_0)$ are coprime, and $\mathrm{rank}(E'_0)$ is divided by $\mathrm{rank}(E_0)$. This contradicts to $1\leq \mathrm{rank}(E'_0) <\mathrm{rank}(E_0)$. Thus the projective representation $\bD^P\circ\IC_{X_1}(\mE)$ is irreducible.
\end{proof}
 
\section{Appendix: explicit formulas}
 In this appendix, we give an explicit formula of the self-map $\varphi_{\lambda,p}$ in Theorem~\ref{Thm:selfmap_formula} and an explicit formula of multiplication by $p$ map in Theorem~\ref{thm:multp_formula}. Then the Conjecture~\ref{conj-1} is equivalent to:
\begin{equation}\label{equ:main}
\frac1{a^p}\left(\frac{\det(B_0)}{\det(B_{m+1})}\right)^2=\frac{a^p}{\lambda^{p-1}}\left(\frac{\det(A_{m+1})}{\det(A_p)}\right)^2
\end{equation}
here $m=\frac{p-1}{2}$ and matrices $A_{m+1}$, $A_p$, $B_0$ and $B_{m+1}$ are given as following:
\begin{equation*}
A_{i}=\left( \begin{array}{cccccc}
\delta_{m} & \cdots & \delta_{i-2} & \delta_{i} &\cdots &\delta_{p-1} \\
\delta_{m-1} & \cdots & \delta_{i-3} & \delta_{i-1} &\cdots &\delta_{p-2} \\
\vdots & \ddots & \vdots &\vdots &\ddots &\vdots \\
\delta_{2} & \cdots & \delta_{i-m} & \delta_{i+2-m} &\cdots &\delta_{m+2} \\ 
\delta_{1} & \cdots  & \delta_{i-1-m}&\delta_{i+1-m} &\cdots &\delta_{m+1} \\
\end{array} \right)
\end{equation*}
\begin{equation*} 
 B_0=\left(\begin{array}{rrrrr}
a^p\gamma_{3m} & a^p\gamma_{3m-1} &\cdots & a^p\gamma_{2m+1} & a^p\gamma_{2m}\\
\gamma_{m} & a^p\gamma_{3m} &\cdots & a^p\gamma_{2m+2} & a^p\gamma_{2m+1}\\
\gamma_{m+1} &\gamma_{m} &\cdots & a^p\gamma_{2m+3} & a^p\gamma_{2m+2}\\
\vdots &\vdots&\ddots  &\vdots  &\vdots \\
\gamma_{2m-1} &\gamma_{2m-2} &\cdots & \gamma_{m} & a^p\gamma_{3m}\\
\end{array}\right)
\end{equation*}
\begin{equation*} 
 B_{m+1}= \left(\begin{array}{rrrrrr}
\gamma_m & a^p\gamma_{3m} & a^p\gamma_{3m-1} &\cdots & a^p\gamma_{2m+1} \\
\gamma_{m+1} & \gamma_{m} & a^p\gamma_{3m} &\cdots & a^p\gamma_{2m+2} \\
\gamma_{m+2} & \gamma_{m+1} &\gamma_{m} &\cdots & a^p\gamma_{2m+3} \\
\vdots &\vdots &\vdots&\ddots  &\vdots  \\
\gamma_{2m} & \gamma_{2m-1} &\gamma_{2m-2} &\cdots & \gamma_{m} \\
\end{array}\right).
\end{equation*} 
and 
\[\delta_n=\frac{\lambda^p(1-a^p)-(\lambda^p-a^p)\lambda^n}{n}\]
\[ \gamma_n=(-1)^{m+n}\sum_{\begin{array}{c}
		i+j=n-m\\ 0\leq i,j\leq m\\ 
		\end{array}}{m\choose i}{m\choose j}\lambda^{m-j}. \]
%(\ref{matrixA_i}), (\ref{matrixB0}) and (\ref{matrixB_m+1}).

By Proposition~\ref{compTwoConj}, we reduce Conjecture~\ref{conj-1} to the following conjecture:
\begin{conj}\label{var_conj} The following equation holds
 \begin{equation}
\det(A_p)=c\lambda^{m^2}(\lambda-1)^{m^2}\cdot \det(B_{m+1}), 
\end{equation}
where 
\[c=(-1)^{m}\cdot \det\left(\begin{array}{cccc}
\frac{1}{m} & \frac{1}{m+1} & \cdots & \frac{1}{p-2} \\
\frac{1}{m-1} & \frac{1}{m} & \cdots & \frac{1}{p-3} \\
\vdots &\vdots &\ddots &\vdots \\
\frac{1}{1} & \frac{1}{2} & \cdots & \frac{1}{m} \\
\end{array}\right).\]
\end{conj}
 By using Maple, Conjecture~\ref{var_conj} has been checked for odd prime $p<50$. Thus Conjecture~\ref{conj-1} holds for $p<50$.

\subsection{Self-map}\label{Calculate_Selfmap}
To compute the self-map $\varphi_{\lambda,p}$, we recall the explicit construction of the inverse Cartier functor in curve case and give some notations used in the computation. For the general case, see the appendix of \cite{LSYZ14}.

Let $k$ be a perfect field of characteristic $p\geq 3$. For simplicity, we may assume that $k$ is algebraic closed. Let $W=W(k)$ be the ring of Witt vectors and $W_n=W/p^n$ for all $n\geq1$ and $\sigma:W\rightarrow W$ be the Frobenius map on $W$. Let $X_1$ be a smooth algebraic curve over $k$ and $\overline{D}$ be a simple normal crossing divisor. We assume that $(X_1,\overline{D})$ is $W_2(k)$-liftable and fix a lifting $(X_2,D)$ \cite{EV-92}. 

For a sufficiently small open affine subset $U$ of $X_2$, the Proposition 9.7 and Proposition 9.9 in \cite{EV-92} give the existence of log Frobenius lifting over $U$, respecting the divisor $D\cap U$. We choose a covering of affine open subsets $\{U_i\}_{i\in I}$ of $X_2$ together with a log Frobenius lifting $F_i:U_i\rightarrow U_i$, respecting the divisor $D\cap U_i$ for each $i\in I$. Denote $R_i=\mathcal O_{X_2}(U_i)$,  $R_{ij}=\mathcal O_{X_2}(U_{ij})$ and
\[\Phi_i=F_i^\#: R_i\rightarrow R_i.\] 
For any object $\aleph$ (e.g. open subsets, divisors, sheaves, etc.)  over $X_2$, we denote by $\overline{\aleph}$ its reduction on $X_1$. Denote by $\Phi$ the $p$-th power map on all rings of characteristic $p$. Thus $\overline{\Phi}_i=\Phi$ on $\overline{R}_{ij}$.

Since $F_i$ is a log Frobenius lifting, $\mathrm{d}\Phi_i$ is divided by $p$ and which induces a map
\begin{equation*}
\frac{\mathrm{d}\Phi_i}{p}:\Omega_{X_1}^1(\mathrm{log} \overline{D})(\overline{U}_i)\otimes_{\Phi} \overline{R}_i \rightarrow \Omega_{X_1}^1(\mathrm{log} \overline{D})(\overline{U}_i). \eqno{(\frac{\mathrm{d}\Phi_i}{p})} 
\end{equation*}

Let $(E,\theta)$ be a logarithmic Higgs bundle with nilpotent Higgs field over $X_1$ of exponent$\leq p-1$ and rank $r$. 
Now, we give the construction of $C^{-1}_{X_1\subset X_2}(E,\theta)$. Locally we set
\begin{equation*}
\begin{split}
&V_i =E(\overline{U}_i)\otimes_\Phi \overline{R}_i,\\
&\nabla_i = \mathrm{d} + \frac{\mathrm{d}\Phi_i} {p}(\theta \otimes_\Phi1): V_i\rightarrow V_i\otimes_{\overline{R}_i} \Omega_{X_1}^1(\mathrm{log} \overline{D})(\overline{U}_i),\\
&G_{ij}= \mathrm{exp}(h_{ij}(\theta \otimes_\Phi1)): V_i\mid_{\overline{U}_{ij}}\rightarrow V_j\mid_{\overline{U}_{ij}}.\\
\end{split}  
\end{equation*}
where $h_{ij}:\Omega^1_{X_1}(\overline{U}_{ij})\otimes_\Phi \overline{R}_{ij}\rightarrow \mathcal{O}_{\overline{U}_{ij}}$ is the homomorphism given by the Deligne-Illusie's Lemma~\cite{JJIC02}. Those local data $(V_i,\nabla_i)$'s can be glued into a global sheaf $V$ with an integrable connection $\nabla$ via the transition maps $\{ G_{ij} \}$ (Theorem 3 in \cite{LSZ12a}).  The inverse Cartier functor on $(E,\theta)$ is defined by
	\[C^{-1}_1(E,\theta):=(V,\nabla).\]
Let $e_{i,\cdot}=\{e_{i,1},e_{i,2},\cdots,e_{i,r}\}$ be a basis of $E(\overline{U}_i)$. Then 
\[\Phi^*e_{i,\cdot}:=\{e_{i,1}\otimes_\Phi 1,e_{i,2}\otimes_\Phi 1,\cdots,e_{i,r}\otimes_\Phi 1\}\]
 forms a basis of $V_i$. Now under those basis, there are $r\times r$-matrices $\omega_{\theta,i}$, $\omega_{\nabla,i}$ with coefficients in $\Omega_{X_1}^1(\mathrm{log}\overline{D})(\overline{U}_i)$, and matrices $\mathcal F_{ij}$, $\mathcal{G}_{ij}$ over $\overline{R}_{ij}$, such that
\begin{equation*}
(e_{i,\cdot})=(e_{j,\cdot}) \cdot \mathcal F_{ij}  \eqno{(\mathcal F_{ij})}
\end{equation*}
\begin{equation*}
\theta(e_{i,\cdot})=(e_{i,\cdot}) \cdot \omega_{\theta,i} \eqno{(\omega_{\theta,i})}
\end{equation*}
\begin{equation*}
\nabla_i(\Phi^*e_{i,\cdot})=(\Phi^*e_{i,\cdot})\cdot \omega_{\nabla,i} \eqno{(\omega_{\nabla,i})}
\end{equation*}
\begin{equation*}
G_{ij}(\Phi^*e_{i,\cdot})=(\Phi^*e_{j,\cdot})\cdot \mathcal{G}_{ij} \eqno{(\mathcal{G}_{ij})}
\end{equation*}
By the definition of $\nabla_i$, one has
\begin{equation} \label{nablai}
\omega_{\nabla,i}=\frac{\mathrm{d}\Phi_i}{p}(\omega_{\theta,i}\otimes_\Phi 1).
\end{equation}
We choose and fix a parameter $t_{ij}$ on $U_{ij}$ for every two elements $\{i,j\}$ in I. Then $\Omega^1_{X_1}(\overline{U}_{ij})$ is a free module over $\overline{R}_{ij}=\mathcal{O}_{X_1}(\overline{U}_i\cap \overline{U}_j)$ of rank $1$ generated by $\mathrm{d}t_{ij}$, and there is a matrix $A_{\theta,ij}$ over $\overline{R}_{ij}$ with
\begin{equation*}
\omega_{\theta,i}=A_{\theta,ij}\cdot \mathrm{d} t_{ij}. \eqno{(A_{\theta,ij})}
\end{equation*}
Explicitly, the Deligne-Illusie's map $h_{ij}$ is given by 
\begin{equation}\label{h_ij}
h_{ij}(f\cdot\mathrm{d}t_{ij}\otimes_\Phi 1)=\Phi(f)\cdot\frac{\Phi_i(t_{ij})-\Phi_j(t_{ij})}{p}, 
\end{equation} 
So we have
\begin{equation}
h_{ij}(\theta\otimes_\Phi 1)(\Phi^*e_{i,\cdot}) =(\Phi^*e_{i,\cdot})\cdot \mathcal G^\Delta_{ij}
\end{equation}
and
\begin{equation}
\mathcal G_{ij}=\Phi(\mathcal{F}_{ij})\exp(\mathcal G^\Delta_{ij})
\end{equation}
where 
\begin{equation*}
 \mathcal G^\Delta_{ij}=\Phi(A_{\theta,ij}) \frac{\Phi_i(t_{ij})-\Phi_j(t_{ij})}{p}. \eqno{(\mathcal G^\Delta_{ij})}
\end{equation*}

\subsubsection*{Computation of our example:} Let $\lambda\in W_2(k)$ with $\lambda\not\equiv 0,1\pmod{p}$ and let $X_2=\mathrm{Proj}\, W_2[T_0,T_1]$. Let $D$ be the divisor of $X_2$ associated to the homogeneous ideal $(T_0T_1(T_1-T_0)(T_1-\lambda T_0))$. By using $t=T_0^{-1}T_1$ as a parameter, we can simply write $D=\{0,1,\lambda,\infty\}$. Denote $U_1=X_2 \setminus \{0,\infty\}$, $U_2=X_2 \setminus \{1,\lambda\}$, $D_1=\{1,\lambda\}$ and $D_2=\{0,\infty\}$. Then $\{U_1,U_2\}$ forms a covering of $X_2$,
\[\begin{split}
& R_1=\mathcal O(U_1)=W_2[t,\frac1t], \\
& R_2=\mathcal O(U_2)=W_2[\frac{t-\lambda}{t-1},\frac{t-1}{t-\lambda}],\\
& R_{12}=\mathcal O(U_1\cap U_2)=W_2[t,\frac1t,\frac{t-\lambda}{t-1},\frac{t-1}{t-\lambda}],\\
& \Omega_{X_2}^1(\log D)(U_1)=W_2[t,\frac1t]\cdot \mathrm{d}\log\left(\frac{t-\lambda}{t-1}\right),\\
& \Omega_{X_2}^1(\log D)(U_2)=W_2[\frac{t-\lambda}{t-1},\frac{t-1}{t-\lambda}]\cdot \mathrm{d}\log t.\\
\end{split}\] 
Over $U_{12}$,  one has 
\[\mathrm{d}\log \left(\frac{t-\lambda}{t-1}\right)=\frac{(\lambda-1)t}{(t-\lambda)(t-1)}\cdot \mathrm{d}\log t.\] 
Denote $\Phi_1(\frac{t-\lambda}{t-1})=\left(\frac{t-\lambda}{t-1}\right)^p$ and $\Phi_2(t)=t^p$, which induce two Frobenius liftings on $R_{12}$. One checks that $\Phi_i$ can be restricted on $R_i$ and forms a log Frobenius lifting respecting the divisor $D_i$. Moreover
\begin{equation}\label{equ:dF_1/p}
\frac{\mathrm{d} \Phi_1}{p}\left(\mathrm{d}\log \frac{t-\lambda}{t-1} \otimes_{\Phi} 1\right) = \mathrm{d}\log \frac{t-\lambda}{t-1},
\end{equation}
and 
\begin{equation}\label{equ:dF_2/p}
\frac{\mathrm{d} \Phi_2}{p}\left(\mathrm{d}\log t \otimes_{\Phi} 1\right) = \mathrm{d}\log t.
\end{equation}

\paragraph{\emph{Local expressions of the Higgs field and the de Rham bundle.}}
Let $(E,\theta)$ be a logarithmic graded semistable Higgs bundle over $X_1=\mathbb P^1_k$ with $E=\mathcal O\oplus \mathcal O(1)$. Then the cokernel of 
\[\theta: \mathcal O(1) \rightarrow \mathcal O\otimes \Omega_{X_1}^1(\log \overline{D})\]
 is supported at one point $a\in\mathbb P_{k}^1(\overline{k})$, which is called the zero of the Higgs field. Conversely, for any given $a\in\mathbb P_{k}^1(\overline{k})$, up to isomorphic, there is a unique graded semistable logarithmic Higgs field on $\mathcal O\oplus \mathcal O(1)$ such that its zero equals to $a$. Assume $a\neq \infty$, we may choose and fix a basis $e_{i,j}$ of $\mathcal O(j-1)$ over $U_i$ for $1\leq i,j\leq 2$ such that 
\begin{equation}
\mathcal{F}_{12}=\left(\begin{array}{cc}
1 & 0\\0 & \frac{t}{t-1}\\
\end{array}\right),
\end{equation}
\begin{equation}
\omega_{\theta,1}=\left(\begin{array}{cc}
0 & \frac{t-a}{\lambda-1}\\0 & 0\\
\end{array}\right)\cdot\mathrm{d}\log\frac{t-\lambda}{t-1},
\end{equation}
%\begin{equation}
%\omega_{\theta,2}= \left(\begin{array}{cc}
%0 & \frac{t-a}{t-\lambda}\\  0 & 0
%\end{array}\right)\cdot \mathrm{d}\log t.
%\end{equation}
By (\ref{nablai}), we have
\begin{equation}
\omega_{\nabla,1}=
\left(\begin{array}{cc}
0 & \left(\frac{t-a}{\lambda-1}\right)^p\\
0 & 0 \\ \end{array}\right)
\cdot\mathrm{d}\log\frac{t-\lambda}{t-1}
\end{equation}
%\begin{equation}
%\omega_{\nabla,2}=
%\left(\begin{array}{cc}
%0 & \left(\frac{t-a}{t-\lambda}\right)^p\\
%0 & 0\\ \end{array}\right) \cdot \mathrm{d}\log t
%\end{equation}
We choose $t_{12}=t$ as the parameter on $U_{12}$. Then 
\begin{equation}
A_{\theta,12}=\left(\begin{array}{cc}
0 & \frac{t-a}{(t-1)(t-\lambda)}\\0 & 0\\
\end{array}\right)
\end{equation}
%\begin{equation}
%A_{\theta,21}=\left(\begin{array}{cc}
%0 & \frac{t-a}{t(t-\lambda)}\\  0 & 0
%\end{array}\right)
%\end{equation}
\begin{equation}
\mathcal G^\Delta_{12}  =\Phi(A_{\theta,12})\cdot z_{12}
\end{equation}
\begin{equation}
\mathcal G_{12}=\left(\begin{array}{cc}
1& g \\ 0 & \frac{t^p}{(t-1)^p}\\
\end{array}\right)
\end{equation}
where
\begin{equation*}
 z_{12}= \frac{\Phi_1(t)-\Phi_2(t)}{p} \eqno{(z_{12})}
\end{equation*} 
and 
\begin{equation*}
g=\frac{(t-a)^p}{(t-\lambda)^p(t-1)^p} \cdot z_{12}. \eqno{(g)} 
\end{equation*}

%with $f_\lambda(t)=\sum\limits_{i=1}^{p-1} \frac{\lambda^{p-i}}{i}t^i+\Delta \in k[t]$, where $\Delta=\frac{\lambda^\sigma-\lambda^p}{p}$.

%
%the $k$-vector spaces $\frac{R_{12}}{R_1+\left(\frac{t}{t-1}\right)^{m+1} R_2}$ is generated by 
%\[\left\{\mathrm{pr}\left(\frac{t}{t^p-1}\right),\mathrm{pr}\left(\frac{t^2}{t^p-1}\right),\cdots,\mathrm{pr}\left(\frac{t^{m}}{t^p-1}\right)\right\}\]
%where 
%
%
% 
%\begin{equation}
%\begin{split}
%\frac{t^i f_\lambda(t)}{t^p-\lambda^p} \equiv& \frac{\lambda^{i-1}}{1-i}\cdot\frac{t}{t^p-1}+\frac{\lambda^{i-2}}{2-i}\cdot\frac{t^2}{t^p-1}+\cdots+ \frac{\lambda^{i-m}}{m-i}\cdot\frac{t^{m}}{t^p-1}\\
%&  \mod k[t,\frac1t]+\left(\frac{t}{t-1}\right)^{m+1}\cdot k[\frac{t-\lambda}{t-1},\frac{t-1}{t-\lambda}];\\
%\end{split}
%\end{equation}
%and 
%\begin{equation}
%\begin{split}
%\frac{t^i f_1(t)}{t^p-1} \equiv &
%  \frac{1}{1-i}\cdot\frac{t}{t^p-1}+\frac{1}{2-i}\cdot\frac{t^2}{t^p-1}+\cdots+ \frac{1}{m-i}\cdot\frac{t^{m}}{t^p-1}\\
%& \mod k[t,\frac1t]+\left(\frac{t}{t-1}\right)^{m+1}\cdot k[\frac{t-\lambda}{t-1},\frac{t-1}{t-\lambda}].\\
%\end{split}
%\end{equation}
\paragraph{\emph{Hodge filtration.}}
Since $X_1= \mathbb{P}^1_k$ and $(V,\nabla)$ is semi-stable of degree $p$, the bundle $V$ is isomorphic to $\mathcal{O}(m) \oplus \mathcal{O}(m+1)$ with $p=2m+1$. So the filtration on $(V,\nabla)$
\[
0\subset \mathcal{O}(m+1) \subset V
\]
is the graded semi-stable Hodge filtration on $V$. Choose a basis $e_i$ of $\mathcal{O}(m+1)$ on $U_i$ such that $e_1=\left(\frac{t}{t-1}\right)^{m+1}e_2$ on $U_{12}$. In the following, we will write down the inclusion map 
$\iota:\mathcal{O}(m+1) \rightarrow V$ explicitly via those basis.
%, i.e. finding $h,f\in R_1$ such that \[\iota(e_1)=e_{11}\otimes_\Phi h+e_{12}\otimes_\Phi f.\]
Before this, we shall fix some notations $\mathrm{pr}$, $A$ and $\alpha_i$. The map $\mathrm{pr}$  is the quotient map of $k$-vector spaces
\begin{equation*}
\mathrm{pr}:R_{12} \twoheadrightarrow \frac{R_{12}}{R_1+\left(\frac{t}{t-1}\right)^{m+1} R_2}  \eqno{(\mathrm{pr})}
\end{equation*}
For all $n\in \{1,2,\cdots,p-2\}$, we denote
\begin{equation}\label{nota:delta}
\delta_n=\frac{\lambda^p(1-a^p)-(\lambda^p-a^p)\lambda^n}{n},
\end{equation}
and $A=A(\lambda,a)$ the matrix of size $m\times (m+1)$
\begin{equation*}
A=\left( \begin{array}{ccccc}
\delta_{m} & \delta_{m+1} & \cdots &\delta_{p-2} &\delta_{p-1} \\
\delta_{m-1} & \delta_{m} & \ddots &\vdots &\vdots \\
\vdots & \ddots & \ddots  &\delta_{m+1} &\delta_{m+2}  \\
\delta_{1} & \cdots  & \delta_{m-1}&\delta_{m} &\delta_{m+1} \\
\end{array} \right)_{m\times (m+1)} \eqno{(A)}
\end{equation*} 
For $m+1\leq i\leq p$, we denote by $A_i$ the submatrix of $A$ by removing the $(i-m)$-column
\begin{equation}\label{matrixA_i}
A_{i}=\left( \begin{array}{cccccc}
\delta_{m} & \cdots & \delta_{i-2} & \delta_{i} &\cdots &\delta_{p-1} \\
\delta_{m-1} & \cdots & \delta_{i-3} & \delta_{i-1} &\cdots &\delta_{p-2} \\
\vdots & \ddots & \vdots &\vdots &\ddots &\vdots \\
\delta_{2} & \cdots & \delta_{i-m} & \delta_{i+2-m} &\cdots &\delta_{m+2} \\ 
\delta_{1} & \cdots  & \delta_{i-1-m}&\delta_{i+1-m} &\cdots &\delta_{m+1} \\
\end{array} \right).
\end{equation}
and 
\begin{equation}\label{nota:alpha}
\alpha_{i}=(-1)^{i}\cdot\det A_i
\end{equation}
Obviously, the vector $(\alpha_{m+1},\alpha_{m+2},\cdots,\alpha_p)^T$ is a solution of $AX=0$. 
\begin{lem}\label{mainlem: f h}
$\mathrm{i)}$. Let $f,h$ be two elements in $\overline{R}_1$. Then the $\overline{R}_1$-linear map from $\overline{R}_1\cdot e_1$ to $V(\overline{U}_1)$, which maps $e_1$ to $e_{11}\otimes_\Phi h+e_{12}\otimes_\Phi f$, can be extended to a global map of vector bundles $\mathcal O(m+1)\rightarrow V$ if and only if 
\[f\in \sum_{i=0}^m k\cdot\frac1{t^i} \quad \text{ and } \quad \mathrm{pr}(fg)=0.\]
 
$\mathrm{ii)}$. Suppose $f=(1,\frac1t,\cdots,\frac{1}{t^m})X$ with $X\in k^{(m+1)\times 1}$. Then $\mathrm{pr}(fg)=0$ if and only if $AX=0$.

$\mathrm{iii)}$. The matrix $A$ is of maximal rank and the vector $(\alpha_{m+1},\alpha_{m+2},\cdots,\alpha_p)^T$ is a $k$-basis of the $1$-dimensional space of solutions of $AX=0$.
\end{lem}
\begin{proof} i). Over $U_{12}$, one has 
\begin{equation}
\iota(e_2)=(e_{21}\otimes_\Phi 1,e_{22}\otimes_\Phi 1)\left(
\begin{array}{c}
(h+fg)\cdot \left(\frac{t-1}{t}\right)^{m+1}\\
\\
f\cdot \left(\frac{t}{t-1}\right)^{m}\\
\end{array}
\right)
\end{equation}
Thus $\iota$ can be extended globally if and only if 
$h+fg\in (\frac{t}{t-1})^{m+1} R_2$ and $f\cdot (\frac{t}{t-1})^{m}\in R_2$. That is equivalent to $f\in R_1\cap \left(\frac{t-1}{t}\right)^m R_2$ and $fg\in R_1+(\frac{t}{t-1})^{m+1} R_2$.
The result follows from the fact that $R_1\cap \left(\frac{t-1}{t}\right)^m R_2$ is a $k$-vector space with a basis $\{1,\frac1t,\cdots,\frac{1}{t^m}\}$.

ii). By directly computation, one checks that
\[\mathrm{pr}(fg)=\mathrm{pr}\left(\frac{t}{t^p-1},\frac{t^2}{t^p-1},\cdots,\frac{t^{m}}{t^p-1}\right)\cdot A \cdot 
\left(\begin{array}{c}
a_{m+1}\\ a_{m+2}\\ \vdots \\ a_p\\
\end{array}\right). \]

iii). Since $V\simeq \mathcal O(m)\oplus\mathcal O(m+1)$, the $k$-vector space $\mathrm{Hom}(\mathcal O(m+1),V)$ is of $1$-dimensional. By i) and ii), the $k$-vector space of solutions of $AX=0$ is of $1$-dimensional.
\end{proof}

\paragraph{\emph{Two notations $[\cdot]$ and $\{\cdot\}$.}}
We have inclusion maps $\overline{R}_i\rightarrow \overline{R}_{12}$,
for $i=1,2$. Under these inclusions, we have the following direct sum decomposition as free $k$-vector spaces 
	\[\overline{R}_{12}= \overline{R}_1 \oplus \frac{t}{t-1} \overline{R}_2.\] 
We denote the projection map to the first summand by $[\cdot]$ and the projection map to the second summand by $\{\cdot\}$. Denote 
\begin{equation}
f_o=\frac{\alpha_{m+1}t^{m+1}+\alpha_{m+2}t^{m+2}+\cdots+\alpha_pt^p }{t^p}\quad \text{ and } \quad h_o=-[f_og].
\end{equation}  
By Lemma~\ref{mainlem: f h}, we have following result.
\begin{cor}\label{cor:HodgeFil}
the Hodge filtration of $(V,\nabla)$ on $U_1$ is given by 
\[0\subset \overline{R}_1\cdot v_{12} \subset V(\overline{U}_1),\]
where $v_{12}=e_{11}\otimes_\Phi h_o + e_{12}\otimes_\Phi f_o$.
\end{cor}  
\paragraph{\emph{The Higgs field of the graded Higgs bundle.}}
We extend $v_{12}$ (defined in Corollary~\ref{cor:HodgeFil}) to an $\overline{R}_1$-basis $\{v_{11},v_{12}\}$ of $V(\overline{U}_1)$. Assume $v_{11}=e_{11}\otimes_\Phi h_1 + e_{12}\otimes_\Phi f_1$ and denote $P=\left(\begin{array}{cc}
h_1& h_2\\ f_1 & f_2\\
\end{array}\right)$, which is an invertible matrix over $\overline{R}_1$ with determinant $d:=\det(P)\in \overline{R}_1^\times$. One has 
\begin{equation}
(v_{11},v_{12})=(e_{11}\otimes_\Phi 1,e_{12}\otimes_\Phi 1)\left(\begin{array}{cc} h_1 & h_o\\f_1 & f_o\\
\end{array}\right)
\end{equation}
and 
\begin{equation}
 \nabla(v_{11},v_{12})= 
 (v_{11},v_{12})\cdot \upsilon_{\nabla,1}
\end{equation}
where $\upsilon_{\nabla,1}=\left(P^{-1}\cdot \mathrm{d}P+ P^{-1}\cdot \omega_{\nabla,1} \cdot P\right)$ equals to
\begin{equation*}
\left(\begin{array}{cc}
\frac{f_o\mathrm{d}h_1-h_o\mathrm{d}f_1}{\mathrm{d}\log \frac{t-\lambda}{t-1}} 
+ f_1f_o \left(\frac{t-a}{\lambda-1}\right)^p  
& \frac{f_o\mathrm{d}h_o-h_o\mathrm{d}f_o}{\mathrm{d}\log \frac{t-\lambda}{t-1}}
+f_o^2 \left(\frac{t-a}{\lambda-1}\right)^p \\
\frac{-f_1\mathrm{d}h_1+h_1\mathrm{d}f_1 }{\mathrm{d}\log \frac{t-\lambda}{t-1}}-f_1^2 \left(\frac{t-a}{\lambda-1}\right)^p 
& \frac{-f_1\mathrm{d}h_o+h_1\mathrm{d}f_o}{\mathrm{d}\log \frac{t-\lambda}{t-1}}-f_1f_o \left(\frac{t-a}{\lambda-1}\right)^p \\
\end{array}\right)\cdot \frac{\mathrm{d}\log \frac{t-\lambda}{t-1}}{d}.
\end{equation*}
 Taking the associated graded Higgs bundle,  the Higgs field $\theta'$ on $\mathrm{Gr}(V,\nabla,\Fil)(\overline{U}_1)=V(\overline{U}_1)/(\overline{R}_1\cdot v_{12}) \oplus \overline{R}_1\cdot v_{12}$ is given by
\begin{equation}\label{equ:gradingHiggsField}
\theta'(e_{12}')=
\frac1d\left(\frac{f_o\mathrm{d}h_o-h_o\mathrm{d}f_o}{\mathrm{d}\log \frac{t-\lambda}{t-1}}
+f_o^2 \left(\frac{t-a}{\lambda-1}\right)^p\right)\cdot \left(e'_{11}\otimes \mathrm{d}\log \frac{t-\lambda}{t-1}\right)
\end{equation}
over $\overline{U}_1$, where $e'_{11}$ is the image of $v_{11}$ in $V(\overline{U}_1)/(\overline{R}_1\cdot v_{12})$ and $e'_{12}=v_{12}$ in $\overline{R}_1v_{12}$. 
Thus the zero of the graded Higgs bundle $\mathrm{Gr}(V,\nabla,\Fil)$ is the root of polynomial
\begin{equation}
P_{\theta'}(t)= \frac{f_o\cdot\mathrm{d}h_o-h_o\cdot \mathrm{d}f_o}{\mathrm{d}\log \frac{t-\lambda}{t-1}}
 +f_o^2\cdot \left(\frac{t-a}{\lambda-1}\right)^p.
\end{equation}

\begin{lem}\label{lem:gradHiggPoly} Define $\alpha_p$ and $\alpha_{m+1}$  as in (\ref{nota:alpha}). Then  
\[ P_{\theta'}(t)=	 \frac{\alpha_p^2}{\lambda-1}t-\frac{\alpha_{m+1}^2}{\lambda-1}\cdot\frac{a^p}{\lambda^{p-1}}.\]
\end{lem}
\begin{proof} Since $h_o=-[f_og]=\{f_og\}-f_og$, the polynomial $P_{\theta'}(t)$ equals to 
\[\frac{(t-\lambda)(t-1)}{\lambda-1} \left(
f_o\frac{\mathrm{d}\{f_og\}}{\mathrm{d}t}
-\{f_og\}\frac{\mathrm{d}f_o}{\mathrm{d}t}
\right)+
f_o^2\left(
\left(\frac{t-a}{\lambda-1}\right)^p
-\frac{(t-\lambda)(t-1)}{\lambda-1}\cdot\frac{\mathrm{d}g}{\mathrm{d}t}
\right).\]
Recall that $\Phi_2(t)=t^p$ and $\Phi_1(t)=\frac{\left(\frac{t-\lambda}{t-1}\right)^p-\lambda^\sigma}{\left(\frac{t-\lambda}{t-1}\right)^p-1}$,
one has $\mathrm{d}\Phi_2(t)=p\cdot t^{p-1}\mathrm{d}t$ and $\mathrm{d}\Phi_1(t)=p\cdot\frac{(t-1)^{p-1}(t-\lambda)^{p-1}}{(1-\lambda)^{p-1}}\mathrm{d}t$. Since $g=\frac{(t-a)^p}{(t-\lambda)^p(t-1)^p} \cdot z_{12}$ with $z_{12}=\frac{\Phi_1(t)-\Phi_2(t)}{p}$, we have 
\[\left(\frac{t-a}{\lambda-1}\right)^p
-\frac{(t-\lambda)(t-1)}{\lambda-1}\cdot\frac{\mathrm{d}g}{\mathrm{d}t}=\frac{(t-\lambda)(t-1)}{\lambda-1}\cdot\frac{t^{p-1}(t^p-a^p)}{(t^p-\lambda^p)(t^p-1)}.\]

\[\frac{\mathrm{d}g}{\mathrm{d}t}=\frac{(t-a)^p}{(t-\lambda)^p(t-1)^p}\cdot\left((\lambda-1)^{p-1}(t-1)^{p-1}(t-\lambda)^{p-1}-t^{p-1}\right)\] 
\emph{Claim}: Suppose $G$ is some power series contained in \[
\left\{\sum\limits_{\ell=m+1}^{\infty} a_{\ell}\cdot\left(\frac{t}{t-1}\right)^\ell
	+\sum\limits_{\ell=m+1}^{\infty}b_{\ell}\cdot \left(\frac{t}{t-\lambda}\right)^\ell | a_{\ell},b_{\ell} \in k
\right\}
\] 
and $F$ belongs to $\left\{\sum\limits_{i=0}^{m}a_i\cdot \frac1{t^i} |a_i \in k\right\}$. Then
$\frac{(t-1)(t-\lambda)}{\lambda-1}\left(F\frac{\mathrm{d}G}{\mathrm{d}t}-G\frac{\mathrm{d}F}{\mathrm{d}t}\right)$ is contained in $R_2$ and divided by $\frac{t}{t-1}$ in $R_2$. 
The Claim follows from
\[\frac{(t-1)(t-\lambda)}{\lambda-1} \left( \left(\frac{t-1}{t}\right)^i  \frac{\mathrm{d}}{\mathrm{d}t} \left(\frac{t}{t-1}\right)^j-
 \left(\frac{t}{t-1}\right)^j 
 \frac{\mathrm{d}}{\mathrm{d}t}
 \left(\frac{t-1}{t}\right)^i \right)\]
\[=(i+j)\left( \left(\frac{t}{t-1}\right)^{j-i} -\frac{\lambda}{\lambda-1}\cdot \left(\frac{t}{t-1}\right)^{j-1-i}\right),\]
and
\[\frac{(t-1)(t-\lambda)}{\lambda-1} \left( \left(\frac{t-\lambda}{t}\right)^i  \frac{\mathrm{d}}{\mathrm{d}t} \left(\frac{t}{t-\lambda}\right)^j-
 \left(\frac{t}{t-\lambda}\right)^j 
 \frac{\mathrm{d}}{\mathrm{d}t}
 \left(\frac{t-\lambda}{t}\right)^i \right)\]
\[=(i+j)\left(-\left(\frac{t}{t-\lambda}\right)^{j-i} -\frac{1}{\lambda-1}\cdot \left(\frac{t}{t-\lambda}\right)^{j-1-i}\right).\]	
By the claim, $\frac{(t-\lambda)(t-1)}{\lambda-1} \left(
f_o\frac{\mathrm{d}\{f_og\}}{\mathrm{d}t}
-\{f_og\}\frac{\mathrm{d}f_o}{\mathrm{d}t}
\right)\in \frac{t}{t-1}\overline{R}_2$. i.e.
\[\left[\frac{(t-\lambda)(t-1)}{\lambda-1} \left(
f_o\frac{\mathrm{d}\{f_og\}}{\mathrm{d}t}
-\{f_og\}\frac{\mathrm{d}f_o}{\mathrm{d}t}
\right)\right]=0.\] 
On the other hand, $P_{\theta'}(t)\in R_1$, one has 
\begin{equation*}
\begin{split}
P_{\theta'}(t) & =[P_{\theta'}(t)]\\
& = \left[f_o^2\left(
 \left(\frac{t-a}{\lambda-1}\right)^p
 -\frac{(t-\lambda)(t-1)}{\lambda-1}\cdot\frac{\mathrm{d}g}{\mathrm{d}t}\right)\right]\\
 &=\left[f_o^2\left(\frac{(t-\lambda)(t-1)}{\lambda-1}\cdot\frac{t^{p-1}(t^p-a^p)}{(t^p-\lambda^p)(t^p-1)}\right)\right]\\
 &=\left[\frac{(\alpha_{m+1}+\alpha_{m+2}t+\cdots+\alpha_pt^m)^2\cdot (t-a)^p }{(\lambda-1)(t-\lambda)^{p-1}(t-1)^{p-1}}\right]\\
\end{split}
\end{equation*}
Obviously, there are polynomials $f_\infty(t)$, $f_1(t)$ and $f_\lambda(t)$, which are divided by $t$, such that
\begin{equation*}
\frac{(\alpha_{m+1}+\alpha_{m+2}t+\cdots+\alpha_pt^m)^2\cdot (t-a)^p }{(\lambda-1)(t-\lambda)^{p-1}(t-1)^{p-1}}-f_\infty(t)-f_1\left(\frac{t}{t-1}\right)-f_\lambda\left(\frac{t}{t-\lambda}\right)
\end{equation*}
is a constant. Taking value at $t=0$, this constant is just $-\frac{\alpha_{m+1}^2}{\lambda-1}\cdot\frac{a^p}{\lambda^{p-1}}$. By definition of $[\cdot]$, we know that 
\[ P_{\theta'}(t)=f_\infty(t) -\frac{\alpha_{m+1}^2}{\lambda-1}\cdot\frac{a^p}{\lambda^{p-1}}.\]
Since $\deg((\alpha_{m+1}+\alpha_{m+2}t+\cdots+\alpha_pt^m)^2\cdot (t-a)^p)-\deg((\lambda-1)(t-\lambda)^{p-1}(t-1)^{p-1})=1$,  the polynomial  $f_\infty(t)$ is of degree $1$. Comparing the coefficients of the first terms, we get
\[f_\infty(t)=\frac{\alpha_p^2}{(\lambda-1)}\cdot t.\qedhere\]
\end{proof}

\begin{thm}\label{Thm:selfmap_formula} Let $A_p$ and $A_{m+1}$ be defined as in (\ref{matrixA_i}). Then the self-map $\varphi_{\lambda,p}$ is given by  
\[\varphi_{\lambda,p}(a) = \frac{a^p}{\lambda^{p-1}}\cdot \left(\frac{\det(A_{m+1})}{\det(A_p)}\right)^2.\]
\end{thm} 
\begin{proof}
For $a\neq\infty$, this follows Lemma~\ref{lem:gradHiggPoly}. For $a=\infty$, we can change the parameter $t$ to compute the self-map at $a$. 
\end{proof}

\subsection{Multiple by $p$ map on elliptic curve}

	Let $k$ be a perfect field of characteristic $p\geq 3$.  Let $\lambda\in k$ with $\lambda\not\equiv 0,1\mod p$.  The Weierstrass function
\[y^2=x(x-1)(x-\lambda)\]
defines an elliptic curve $C_{\lambda}$ over $k$. Let $Q_1=(a,b)$ be a $k$-point on $C_\lambda$. We denote 
\[Q_n=(a_n,b_n):=\underbrace{Q_1+Q_1+\cdots+Q_1}_n.\]
Then $a_n$ is a rational function of $a$. In this appendix, we will give an explicit formula of this rational function for the case of $n=p$. Without lose of generality, we may assume that  $k$ is algebraic closed and $a\neq 0,1,\lambda,\infty$.

Since the divisor $(p+1)(\infty)-p(Q_1)$ is of degree $1$, the space of global sections of its associated line bundle is of $1$-dimension. Choosing a nontrivial global section $\alpha$, then 
\[\mathrm{div}(\alpha)=p(Q_1)+(Q_p)-(p+1)(\infty)\]
and it is also a global section of $\mathcal O_{C_\lambda}\Big((p+1)(\infty)\Big)$. On the other hand, the $k$-vector space of the global sections of $\mathcal O_{C_\lambda}\Big((p+1)(\infty)\Big)$ are of $(p+1)$-dimension with basis ($m=\frac{p-1}{2}$)
\[1,x,x^2,\cdots,x^{m+1},y,yx,\cdots,yx^{m-1}.\]
So we can write $\alpha$ in the form
\[\alpha=f-yg,\]
where $f,g\in k[x]$ with $\deg(f)\leq m+1$ and $\deg(g)\leq m-1$. Since
\begin{equation}
\begin{split}
\mathrm{div}((x-a)^p(x-a_p)) & =p(Q_1)+p(-Q_1)+(Q_p)+(-Q_p)-(2p+2)(\infty)\\
&=\mathrm{div}(\alpha\overline{\alpha})
\end{split}
\end{equation}
Here $\overline{\alpha}:=f+yg$.\\ 
By multipling suitable constant to $\alpha$, we may assume 
\begin{equation}\label{TwoWayFaction}
(x-a)^p(x-a_p)=\alpha\overline{\alpha}=f^2-x(x-1)(x-\lambda)g^2.
\end{equation} 
Comparing the degree on both side, one gets
\[\deg(f)=m+1 \text{ and } \deg(g)\leq m-1.\]
Writing $f$ in form $f=\beta_0+\beta_1x+\cdots+\beta_{m+1}x^{m+1},$ we denote
\begin{equation}
 \beta=\left(\begin{array}{c}
\beta_0\\ \beta_1\\ \cdots\\ \beta_{m+1}\\
\end{array}\right)
\end{equation}
and consider the first terms and the constant terms on both sides of (\ref{TwoWayFaction}). Then one gets
\begin{equation}
a_p=\frac{1}{a^p} \left(\frac{\beta_0}{\beta_{m+1}}\right)^2.
\end{equation} 
So in order to get the rational function we want, we need to determine the ratio $[\beta_0:\beta_{m+1}]$.
In the following, we will define a full rank matrix $B$ of size $(m+1)\times(m+2)$ such that $(\beta_0,\beta_1,\cdots,\beta_{m+1})^T$ is a non-zero solution of $BX=0$. Then the ratio $[\beta_0:\beta_{m+1}]$ can be described by the determinants of submatrices of $B$. Expand the polynomial 
\begin{equation}
 \Big(x(x-1)(x-\lambda)\Big)^{m}=\gamma_m x^m+\gamma_m x^{m+1}+\cdots \gamma_{3m} x^{3m},
\end{equation}
where
\begin{equation}\label{nota:gamma_n}
 \gamma_n=(-1)^{m+n}\sum_{\begin{array}{c}
		i+j=n-m\\ 0\leq i,j\leq m\\ 
		\end{array}}{m\choose i}{m\choose j}\lambda^{m-j}. 
\end{equation}
and denote
\begin{equation*}
B=\left(\begin{array}{rrrrrr}
	\gamma_m & a^p\gamma_{3m} & a^p\gamma_{3m-1} &\cdots & a^p\gamma_{2m+1} & a^p\gamma_{2m}\\
	\gamma_{m+1} & \gamma_{m} & a^p\gamma_{3m} &\cdots & a^p\gamma_{2m+2} & a^p\gamma_{2m+1}\\
	\gamma_{m+2} & \gamma_{m+1} &\gamma_{m} &\cdots & a^p\gamma_{2m+3} & a^p\gamma_{2m+2}\\
	\vdots &\vdots &\vdots&\ddots  &\vdots  &\vdots \\
	\gamma_{2m} & \gamma_{2m-1} &\gamma_{2m-2} &\cdots & \gamma_{m} & a^p\gamma_{3m}\\
	\end{array}\right). \eqno{(B)}
\end{equation*}

\begin{lem}  $B\cdot\beta=0$  
\end{lem}
\begin{proof} Since $a\neq 0,1,\lambda,\infty$, the function $x(x-1)(x-\lambda)$ is invertible in $k[[x-a]]$. Thus the $\frac1{\sqrt{x(x-1)(x-\lambda)}}$ is an element in $k[[x-a]]$. Since 
\[\Big(x(x-1)(x-\lambda)\Big)^{p}\equiv \Big(a(a-1)(a-\lambda)\Big)^{p} \pmod{(x-a)^p}, \]
one has 
\begin{equation}\label{equ:sqrt1}
\sqrt{x(x-1)(x-\lambda)} \equiv \pm\frac{\Big(a(a-1)(a-\lambda)\Big)^{p/2}}{\Big(x(x-1)(x-\lambda)\Big)^{m}}  \pmod{(x-a)^p}. 
\end{equation} 
Since $(x-a)^p(x-a_p)=f^2-x(x-1)(x-\lambda)g^2$ and $x-a\nmid fg$, one gets 
\begin{equation}\label{equ:sqrt2}
\sqrt{x(x-1)(x-\lambda)} \equiv \pm \frac{f}{g} \pmod{(x-a)^p}.
\end{equation} 
Now comparing (\ref{equ:sqrt1}) and (\ref{equ:sqrt2}), one gets
\begin{equation}\label{equ:mainf}
f\cdot \Big(x(x-1)(x-\lambda)\Big)^{m} \equiv \pm \Big(a(a-1)(a-\lambda)\Big)^{p/2}\cdot g \pmod{(x-a)^p}.
\end{equation}
Consider the map of $k$-vector spaces 
\begin{equation*}
\mathrm{pr': k[[x-a]]} \twoheadrightarrow \frac{k[[x-a]]}{(x-a)^p\cdot k[[x-a]]+\sum\limits_{i=0}^{m-1} k\cdot x^i  } \eqno{(\mathrm{pr'})}.
\end{equation*}
From (\ref{equ:mainf}), we have $\mathrm{pr'}\left(f\cdot \Big(x(x-1)(x-\lambda)\Big)^{m}\right)=0$. By direct computation, one checks that for all $0\leq i\leq m+1$
\[\mathrm{pr'}\left(x^i\cdot \Big(x(x-1)(x-\lambda)\Big)^{m}\right)= \sum_{j=m}^{3m} \gamma_j\cdot \mathrm{pr'}(x^{i+j}).\]
Since $x^{n+p}\equiv a^px^n \mod(x-a)^p$, one has $\mathrm{pr'}(x^{n+p})=a^p\cdot \mathrm{pr'}(x^n)$. Thus 
\[\mathrm{pr'}\left(x^i\cdot \Big(x(x-1)(x-\lambda)\Big)^{m}\right)=\sum_{j=m+i}^{p-1} \gamma_{j-i}\cdot \mathrm{pr'}(x^{j})+ \sum_{j=m}^{m+i-1} a^p\gamma_{p+j-i}\cdot \mathrm{pr'}(x^{j})  \]
Therefore 
\begin{equation}
\mathrm{pr'}\left(f\cdot \Big(x(x-1)(x-\lambda)\Big)^{m}\right)=\mathrm{pr'}(x^m,x^{m+1},\cdots,x^{2m})\cdot (B\cdot\beta)
\end{equation}
The lemma follows.
\end{proof}
Recall that $\gamma_n$ is defined in~(\ref{nota:gamma_n}), we denote $B_0$ and $B_{m+1}$ to be two submatrices of $B$ as following 
\begin{equation}\label{matrixB0}
 B_0=\left(\begin{array}{rrrrr}
a^p\gamma_{3m} & a^p\gamma_{3m-1} &\cdots & a^p\gamma_{2m+1} & a^p\gamma_{2m}\\
\gamma_{m} & a^p\gamma_{3m} &\cdots & a^p\gamma_{2m+2} & a^p\gamma_{2m+1}\\
\gamma_{m+1} &\gamma_{m} &\cdots & a^p\gamma_{2m+3} & a^p\gamma_{2m+2}\\
\vdots &\vdots&\ddots  &\vdots  &\vdots \\
\gamma_{2m-1} &\gamma_{2m-2} &\cdots & \gamma_{m} & a^p\gamma_{3m}\\
\end{array}\right)
\end{equation}
\begin{equation}\label{matrixB_m+1}
 B_{m+1}= \left(\begin{array}{rrrrrr}
\gamma_m & a^p\gamma_{3m} & a^p\gamma_{3m-1} &\cdots & a^p\gamma_{2m+1} \\
\gamma_{m+1} & \gamma_{m} & a^p\gamma_{3m} &\cdots & a^p\gamma_{2m+2} \\
\gamma_{m+2} & \gamma_{m+1} &\gamma_{m} &\cdots & a^p\gamma_{2m+3} \\
\vdots &\vdots &\vdots&\ddots  &\vdots  \\
\gamma_{2m} & \gamma_{2m-1} &\gamma_{2m-2} &\cdots & \gamma_{m} \\
\end{array}\right).
\end{equation} 
\begin{cor}
 $[\beta_0:\beta_{m+1}]=[(-1)^{m+1}\det(B_0):\det(B_{m+1})]$.
\end{cor}
Now,  we get the self-map on $\mathbb P^1_k$ induced by the multiplication by $p$ map.
\begin{thm}\label{thm:multp_formula} $a_p=\frac{1}{a^p}\cdot \left(\frac{\det(B_{0})}{\det(B_{m+1})}\right)^2$.
\end{thm}

\begin{prop}\label{compTwoConj}
 The Conjecture~\ref{var_conj} implies the equation (\ref{equ:main}) and the Conjecture~\ref{conj-1}.
\end{prop}
\begin{proof} Regard every terms of  $A_{m+1}$, $A_p$, $B_0$ and $B_{m+1}$ as polynomials of $a,\lambda$. One checks directly that 
\begin{equation*}  A^T_{m+1}(\lambda,a)=\lambda^{2p}a^pA_p(\frac1\lambda,\frac{1}{a})
\end{equation*} 
\begin{equation*} B^T_{0}(\lambda,a)=\lambda^{m}a^pB_{m+1}(\frac1\lambda,\frac{1}{a}) 
\end{equation*} 
On the other hand, by Conjecture~\ref{var_conj}, 
\[\det(A_p(\frac1\lambda,\frac1a))=c\lambda^{-2m^2}(1-\lambda)^{m^2}\cdot \det(B_{m+1}(\frac1\lambda,\frac1a)).\]
Thus 
\[\det(A_{m+1})=c\lambda^{m^2+m}(1-\lambda)^{m^2}a^{-p}\cdot \det(B_0).\]
The Proposition follows. 
\end{proof}

\begin{rmk}
The Conjecture~\ref{var_conj} holds for odd prime $p<50$. This was checked directly by using Maple. By Proposition~\ref{compTwoConj}, our main conjecture holds for $p<50$.
\end{rmk}

\section{Appendix: the torsor maps induced by inverse Cartier functor and Grading functor}
In this section, we will describe the torsor map induced by inverse Cartier functor (Proposition~\ref{prop:torsor IC}) and Grading functor (Proposition~\ref{prop:torsor_Grading}) via maps between cohomology groups. 

\subsection*{Some notations}
Let $T$ be a vector space and assume $L$ is $T$-torsor space. Then for a given point $\ell\in L$, the torsor structure gives a bijection $\iota_\ell: L\rightarrow T$. We call $\iota_\ell(\ell')$ the difference between $\ell'$ and $\ell$, or we say that $\ell'$ differs from $\ell$ by $\iota_\ell(\ell')$. Denote $c(\ell',\ell):=\iota_\ell(\ell')$.

Denote $W= W(k)$ and $W_n=W/p^nW$. 
Let $\mathcal{X}$ be a proper smooth $W$-scheme with normal crossing divisor $\mD$, $X_n := \mathcal{X} \times_W W_n$ and $D_n:=\mD\otimes_WW_n$. Let $\mX = \bigcup \mU_i$ be a covering of small affine open subsets and let $\Phi_i$ be a Frobenius lifting on $\mU_i$ which preserves the log divisor, i.e. $\Phi_i^*(\mD\cap \mU_i)=p(\mD\cap \mU_i)$. Denote $\mU_{ij}=\mU_i\cap\mU_j$. For a vector bundle $L$ over $X_n$, write $L(\mU_i):=L(\mU_i\times_WW_n)$ and $L(\mU_{ij}):=L((\mU_i\cap\mU_j)\times_WW_n)$ for short.

In this section, all Higgs bundles and de Rham bundles are logarithmic with respect to the divisor $\mD$. 

\subsection{Lifting space of Higgs Bundles and de Rham bundles}\label{section_LSHBFB}
Let $(\overline{E},\overline{\theta})$ be a Higgs bundle over $X_{n-1}$. Denote its reduction modulo $p$ by $(E , \theta)$. Now we want to study the space of $W_n$-liftings of $(\overline{E},\overline{\theta})/X_{n-1}$.\\

\begin{lem}\label{lem:HiggsTorsor}
The space of $W_n$-liftings of $(\overline{E},\overline{\theta})/X_{n-1}$ is an $H^1_{Hig}(X_1 , \mEnd (E , \theta))$-torsor.  
\end{lem}
\begin{proof}
Consider two $W_n$-lifting $(\wt E,\wt \theta)/X_n$ and $(\wh E,\wh\theta)$ of $(\overline{E},\overline{\theta})/X_{n-1}$. Denote by $(\wt E_i , \wt\theta_i)$ (resp. $(\wh E_i , \wh\theta_i)$) the restriction of $(\wt E,\wt \theta)$ (resp. $(\wh E, \wh\theta)$) on $\mU_i\times_WW_n$. Locally we can always find isomorphisms $\gamma_i: \wt E_i \overset{\sim}{\longrightarrow} \wh E_i$ over $\mU_i\times_WW_n$ which lifts $\mathrm{id}_{\overline{E}_i}$. Set
\begin{equation*}
f_{ij}:= \gamma^{-1}_j|_{\mU_{ij}}\circ \gamma_i|_{\mU_{ij}} - \mathrm{id} \in \mEnd(\wt E)(\mU_{ij})
\end{equation*} 
\begin{equation*}
 \omega_i := \gamma_i^{-1} \circ \wh\theta_i \circ \gamma_i - \wt\theta_i \in \mEnd(\wt E)(\mU_i) \otimes \Omega^1_{\mX/W}(\log \mD)(\mU_i)
\end{equation*}    
Since $(\wt E, \wt \theta)$ and $(\wh E,\wh \theta)$ are both $W_n$-liftings of $(\overline{E},\overline{\theta})$, we have
\[ f_{ij} \equiv 0\pmod{p^{n-1}}  \qquad \text{ and } \qquad \omega_i  \equiv 0 \pmod{p^{n-1}}.\]
Thus $\overline{f}_{ij}:=\frac{f_{ij}}{p^{n-1}}\pmod{p}$ is a well defined element in $\mEnd(E)(\mU_{ij})$ and $\overline\omega_i:=\frac{\omega_i}{p^{n-1}}\pmod{p}$ is a well defined element in $\mEnd(E)(\mU_i) \otimes \Omega^1_{\mX/W}(\log \mD)(\mU_i)$.
These local datum give us a $\cech$ representative 
\[(\overline{f}_{ij} ,  \overline\omega_i) \in H^1_{Hig}(X_1 , \mEnd (E,\theta))\]
of the difference of the two liftings.

Conversely, we can construct a Higgs bundle from datum $(\wt E,\wt\theta)$ and $(\overline f_{ij}, \overline\omega_i)\in H^1_{Hig}(X_1 , \mEnd (E,\theta))$. Locally over $\mU_i\times_WW_n$, the new Higgs bundle is given by
\[(\wt E_i,\wt \theta_i + p^{n-1}\overline\omega_i),\]
and the gluing transform is 
\[\mathrm{id}+p^{n-1}\overline f_{ij}:{\wh E\mid_{(\mU_i\cap \mU_j)\times_WW_n}}\rightarrow {\wh E\mid_{(\mU_i\cap \mU_j)\times_WW_n}}.\]
Since $(\overline f_{ij}, \overline\omega_i)$ is a $1$-cocycle,the local datum are glued into a new Higgs bundle. Moreover the following diagram commutes
\begin{equation}
\xymatrix@C=3cm{
\left(\wt E_i,\wt \theta_i+p^{n-1}\overline\omega_i\right)  \ar[r]^{\mathrm{id}+p^{n-1}\overline f_{ij}} \ar[d]^{\gamma_i}
&  \left(\wt E_j,\wt \theta_j+p^{n-1}\overline\omega_j\right)\ar[d]^{\gamma_j}\\
\left(\wh E_i,\wh \theta_i\right) \ar[r]^{\mathrm{id}}
&  \left(\wh E_j,\wh \theta_j\right)\\
}
\end{equation}
 and this new Higgs bundle is isomorphic to $(\wh E,\wh\theta)$ via local isomorphisms $\gamma_i$.
\end{proof}
Similarly, for de Rham bundle $(\overline{V},\overline{\nabla})$ over $X_{n-1}$, denote its reduction modulo $p$ by $(V,\nabla)$.
\begin{lem}
The space of $W_n$-liftings of $(\overline{V},\overline{\nabla})$ is an $H^1_{dR}(X_1 , \mEnd (V,\nabla))$-torsor.  
\end{lem}

\subsection{lifting space of graded Higgs bundles and filtered de Rham bundles with Griffith transversality}

Let $(\overline{V},\overline{\nabla},\overline{\Fil})$ be a filtered de Rham bundle over $X_{n-1}$ satisfying Griffith transversality. Denote its modulo $p$ reduction by $(V,\nabla,\Fil)$. Denote by $(\overline{E},\overline{\theta})$ the graded Higgs bundle $\mathrm{Gr}(\overline{V},\overline{\nabla},\overline{\Fil})$ over $X_n$ and denote by $(E,\theta)$ the graded Higgs bundle $\mathrm{Gr}(V,\nabla,\Fil)$ over $X_1$.

The filtration $\overline{\Fil}$ on $\overline{V}$ induced sub complex of the de Rham complex $\left(\mEnd(V)\otimes \Omega_{\mX/W}^\bullet(\log \mD),\nabla^{\mEnd}\right)$ 
\begin{equation*}
\xymatrix {
  0 \ar[r] 
  &   \Fil^p\mEnd(V) \ar[r] \ar@{^(->}[d]
  &  \Fil^{p-1}\mEnd(V)\otimes \Omega_{\mX/W}^1(\log \mD) \ar[r] \ar@{^(->}[d]
  & \Fil^{p-2}\mEnd(V)\otimes \Omega_{\mX/W}^2(\log \mD) \ar[r] \ar@{^(->}[d]
  &  \cdots\\
0 \ar[r] \ar[r]
  &  \mEnd(V) \ar[r]^(0.4){\nabla^{\mEnd}}
  &  \mEnd(V)\otimes \Omega_{\mX/W}^1(\log \mD) \ar[r]^{\nabla^{\mEnd}}
  &  \mEnd(V)\otimes \Omega_{\mX/W}^2(\log \mD) \ar[r]^(0.7){\nabla^{\mEnd}}
  &  \cdots\\}
\end{equation*}
here $\Fil^p\mEnd(V)= \sum\limits_{j-i=p}\left(\Fil^iV\right)^\vee\otimes \Fil^jV$. We denote this sub complex by $\Fil^p\left(\mEnd(V)\otimes \Omega_{\mX/W}^\bullet(\log \mD),\nabla^{\mEnd}\right)$. When $p$ runs all integers, these sub complexes give an exhaustive and decreasing Filtration on $\left(\mEnd(V)\otimes \Omega_{\mX/W}^\bullet(\log \mD),\nabla^{\mEnd}\right)$. Taking the associated graded object, one gets an isomorphism of Higgs complexes
\[\mathrm{Gr}\left(\mEnd(V)\otimes \Omega_{\mX/W}^\bullet(\log \mD),\nabla^{\mEnd}\right) \simeq \left(\mEnd(E)\otimes \Omega_{\mX/W}^\bullet(\log \mD), \theta^{\mEnd}\right).\]

Denote by $\overline{E}^p=\mathrm{Gr}^p(\overline{V})$ the $p$-th grading piece of $\overline{V}$. Then Higgs bundle $(\overline{E},\overline{\theta})$ is graded with $\overline{E}=\oplus \overline{E}^p$ and $\overline{\theta}:\overline{E}^p\rightarrow \overline{E}^{p-1}\otimes\Omega_{\mX/W}^1(\log \mD)$. The decomposition of $\overline{E}$ induces a decomposition of its reduction $E=\oplus E^p$ and a decomposition $\mEnd(E)=\oplus \mEnd(E)^p$ with
\[\mEnd(E)^p:= \bigoplus_{j-i=p} (E^i)^\vee\otimes E^j.\]
Thus the complex
\begin{equation}\label{Higgs_direct_summand}
0\rightarrow \mEnd(E)^0 \rightarrow \mEnd(E)^{-1}\otimes \Omega_{\mX/W}^1(\log \mD) \rightarrow \mEnd(E)^{-2}\otimes \Omega_{\mX/W}^2(\log \mD)\rightarrow\cdots, 
\end{equation}
is a direct summand of the Higgs complex, which is just the $0$-th grading piece of the de Rham complex $\left(\mEnd(V)\otimes \Omega_{\mX/W}^\bullet(\log \mD),\nabla^{\mEnd}\right)$.
Taking the second hypercohomology of this complex one gets a direct summand of $H^1_{Hig}(X_1 , \mEnd (E ,\theta))$, we denote it by $\mathrm{Gr^0}H^1_{Hig}(X_1 , \mEnd (E ,\theta))$.
\begin{lem}
The space of $W_n$-liftings of the graded Higgs bundle $(\overline{E},\overline{\theta})/X_{n-1}$ is a $\mathrm{Gr}^0H^1_{Hig}(X_1 , \mEnd (E ,\theta))$-torsor.
\end{lem}
\begin{proof}
The proof is the same as that of Lemma~\ref{lem:HiggsTorsor}, except that one may choose $\gamma_i$ which preserves the grading structures on both sides. Thus
\begin{equation*}
\overline f_{ij} \in \mEnd(E)^0(\mU_{ij}) 
\qquad \text{ and } \qquad  
\overline\omega_i \in \mEnd(E)^{-1}(\mU_i) \otimes \Omega^1_{\mX/W}(\log \mD)(\mU_i). 
\end{equation*} 
Hence one has $(\overline{f}_{ij},  \overline\omega_i) \in \mathrm{Gr}^0H^1_{Hig}(X_1, \mEnd (E,\theta))$.
\end{proof}

Let $(\wt V,\wt \nabla,\wt \Fil)$ and $(\wh V,\wh \nabla,\wh \Fil)$ be two filtered de Rham bundles satisfying Griffith transversality over $X_{n}$, which are liftings of $(\overline{V},\overline{\nabla},\overline{\Fil})$.
\begin{lem}\label{lem:Hodge_Torsor}
	The difference between $(\wh V,\wh \nabla)$ and $(\wt V,\wt\nabla)$ is contained in the hyper cohomology $\bH^1\left(\Fil^0\left(\mEnd(V)\otimes \Omega_{\mX/W}^\bullet(\log \mD),\nabla^{\mEnd}\right)\right)$.
\end{lem}

\begin{proof} The proof the similar as that of Lemma~\ref{lem:HiggsTorsor}.  Denote by $(\wt V_i , \wt\nabla_i)$ (resp. $(\wh V_i , \wh\nabla_i)$) the restriction of $(\wt V,\wt \nabla)$ (resp. $(\wh V, \wh\nabla)$) on $\mU_i\times_WW_n$. Locally we can always find isomorphisms  $\gamma_i: \wt V_i \overset{\sim}{\longrightarrow} \wh V_i$ over $\mU_i\times_WW_n$ that is strict under the Hodge filtrations on both sides and lifts the $\mathrm{id}_{\overline{V}_i}$. Then
\begin{equation*}
f_{ij}:= \gamma^{-1}_j \circ \gamma_i|_{U_{ij}} - Id_{\wt V_{ij}} \in \Fil^0\mEnd(\wt V)((\mU_i\cap\mU_j)\times_WW_n)
\end{equation*}
Since the connections satisfy Griffith transversality,  
\begin{equation*}
\omega_i := \gamma_i^{-1} \circ \wh \nabla_i \circ \gamma_i - \wt\nabla_i \in \Fil^{-1}\mEnd(\wt V)(\mU_i\times_WW_n) \otimes \Omega^1_{\mX/W}(\log \mD)(\mU_i) 
\end{equation*}    
Similarly as in the proof of Lemma~\ref{lem:HiggsTorsor}, these local data give us a $\cech$ representative 
\[(\overline f_{ij} ,  \overline\omega_i) \in \bH^1\left(\Fil^0\left(\mEnd(V)\otimes \Omega_{\mX/W}^\bullet(\log \mD),\nabla^{\mEnd}\right)\right).\]
of the difference of the two liftings.	
\end{proof}

\subsection{The torsor map induced by grading}
In this subsection, we will describe the difference between the grading objects of two filtered de Rham bundles satisfying Griffith transversality over $X_{n}$.
The morphism of complexes
\[\Fil^0 \left(\mEnd(V)\otimes \Omega_{\mX/W}^\bullet(\log \mD),\nabla^{\mEnd}\right) 
\rightarrow \mathrm{Gr}^0\left(\mEnd(V)\otimes \Omega_{\mX/W}^\bullet(\log \mD),\nabla^{\mEnd}\right)\]
induces a $k$-linear map of the hyper cohomology groups
\begin{equation}\label{Grading_Torsor_map}
\xymatrix{
\bH^1\left(\Fil^0\left(\mEnd(V)\otimes \Omega_{\mX/W}^\bullet(\log \mD),\nabla^{\mEnd}\right)\right)  \ar[r]\ar[dr]
 &\bH^1\left(\mathrm{Gr}^0\left(\mEnd(V)\otimes \Omega_{\mX/W}^\bullet(\log \mD),\nabla^{\mEnd}\right)\right) \ar@{=}[d]\\
 &\mathrm{Gr}^0H^1_{Hig}(X_1 , \mEnd (E ,\theta))\\
}
\end{equation}

Suppose $(\wt V,\wt \nabla,\wt \Fil)$ and $(\wh V,\wh \nabla,\wh \Fil)$ are two filtered de Rham bundles satisfying Griffith transversality over $X_{n}$, which are liftings of $(\overline{V},\overline{\nabla},\overline{\Fil})$.
\begin{prop}\label{prop:torsor_Grading}
	The difference between $\mathrm{Gr}(\wh V,\wh \nabla,\wh \Fil)$ and $\mathrm{Gr}(\wt V,\wt\nabla,\wt\Fil)$ is just the image of the difference between $(\wh V,\wh \nabla)$ and $(\wt V,\wt\nabla)$ under the morphism in~\emph{(\ref{Grading_Torsor_map})}.
\end{prop}
\begin{proof}Choosing local isomorphisms $\gamma_i:\wt V_i\rightarrow \wh V_i$ as in the proof of Lemma~\ref{lem:Hodge_Torsor} and taking the associated graded object, one gets isomorphism $\mathrm{Gr}(\wt V_i,\wt\Fil_i)\rightarrow \mathrm{Gr}(\wh V_i,\wh\Fil_i)$.  Then this proposition can be checked directly.
\end{proof}

\begin{rmk} Let $\mL$ be a line bundle over $\mX$. And $\mEnd((E,\theta)\otimes \mL)$ is canonical isomorphic to $\mEnd((E,\theta))$. Let $(\wt E,\wt\theta)$ and $(\wh E,\wh\theta)$ be two liftings of $(\overline E,\overline \theta)$. Then the difference between $(\wt E,\wt\theta)$ and $(\wh E,\wh\theta)$ is the same as 
the difference between $(\wt E,\wt\theta)\otimes \mL$ and $(\wh E,\wh\theta)\otimes \mL$.
Now, the proposition still holds for replacing the grading functor by the composition functor of grading functor and twisting by a line bundle.  
\end{rmk}

\subsection{Torsor map induced by inverse Cartier functor} Assume $(\overline{E},\overline\theta)$ is a graded Higgs bundle and assume it is isomorphic to  $\mathrm{Gr}((\overline{V},\overline{\nabla},\overline{\Fil})_{-1})$ for some filtered de Rham bundle $(\overline{V},\overline{\nabla},\overline{\Fil})_{-1}$ over $X_{n-1}$ satisfying Griffith transversality. Denote $(\overline{V},\overline{\nabla})=C^{-1}_{n-1}((\overline{E},\overline{\theta}),(\overline{V},\overline{\nabla},\overline{\Fil})_{-1}\pmod{p^{n-2}})$. Let $(\wt E,\wt \theta)$ be a graded Higgs bundle which lifts $(\overline{E},\overline\theta)$. The inverse Cartier transform $C^{-1}_n$ on $((\wt E,\wt \theta),(\overline{V},\overline{\nabla},\overline{\Fil})_{-1})$ define a de Rham bundle $(\wt V,\wt \nabla)$ over $X_n$ which lifts $(\overline{V},\overline{\nabla})$.
\begin{equation*}
\xymatrix@C=2cm@R=0cm{
& & (\wt V,\wt \nabla)\ar@{..}[dd]\\
&(\wt E,\wt \theta) \ar@{..}[dd] \ar[ur]^{C^{-1}_n} & \\
(\overline{V},\overline{\nabla},\overline{\Fil})_{-1} \ar[dr]^{\mathrm{Gr}} && (\overline{V},\overline{\nabla})\\
&(\overline{E},\overline{\theta})  \ar[ur]^{C^{-1}_{n-1}} &\\
}
\end{equation*}
Let $(\wh E,\wh \theta)$ be another graded Higgs bundle which lifts $(\overline{E},\overline\theta)$. Denote 
\[(\wh V,\wh \nabla):=C^{-1}_n((\wh E,\wh \theta),(\overline{V},\overline{\nabla},\overline{\Fil})_{-1}).\] 

Assume $(\wh E,\wh \theta)$ differs from $(\wt E,\wt \theta)$ by an element $(\overline f_{ij} ,  \overline\omega_i) \in H^1_{Hig}(X_1 , \mEnd (E ,\theta))$. 

Let $t_1,t_2,\cdots,t_d$ be local coordinate of the small affine open subset $\mU_i\cap\mU_j$ of $\mX$. There are two Frobenius liftings $\Phi_i$ and $\Phi_j$ on the overlap subset. Denote $z^J=\prod\limits_\ell z_\ell^{j_\ell}$ and $z_\ell=\frac{\Phi^*_i(t_\ell)-\Phi^*_j(t_\ell)}{p}$. For a field $\theta$ on a bundle over $X_1$, we denote 
\[h_{ij}(\theta)=\sum_{\ell} \theta(\partial_\ell)\otimes_\Phi z_\ell\]
which is obviously Frobenius semilinear. Now, we want compute the difference between $(\wh V,\wh \nabla)$ and $(\wt V,\wt \nabla)$.

\begin{prop}\label{prop:torsor IC}
The de Rham bundle $(\wh V,\wh \nabla)$ differs from $(\wt V,\wt \nabla)$ by the element 
\[\left(\Phi^*(\overline f_{ij})+h_{ij}(\omega_j),\frac{\Phi^*}{p}(\overline\omega_i)\right) \in H^1_{dR}(X_1 , \mEnd (V ,\nabla)).\]
\end{prop}
\begin{proof} Recall the diagram~\ref{diag:C^{-1}} and denote $(\widetilde{\wt V},\widetilde{\wt \nabla})=\mT_n((\wt E,\wt \theta),(\overline{V},\overline{\nabla},\overline{\Fil})_{-1})$ and $(\widetilde{\wh V},\widetilde{\wh \nabla})=\mT_n((\wh E,\wh \theta),(\overline{V},\overline{\nabla},\overline{\Fil})_{-1})$. From the the definition of functor $\mT_n$, both are bundles with $p$-connections and both are lifting of a bundle with $p$-connection over $X_{n-1}$. Their reductions modulo $p$ are the same Higgs bundle $(E,\theta)$. From the definition of the functor $\mT_n$, the difference between $(\widetilde{\wh V},\widetilde{\wh \nabla})$ and $(\widetilde{\wt V},\widetilde{\wt \nabla})$ is also equal to $(\overline f_{ij} ,  \overline\omega_i) \in H^1_{Hig}(X_1 , \mEnd (E ,\theta))$.

Now, let $\widetilde{\gamma}_i: \widetilde{\wt V}\mid_{\mU_i} \rightarrow \widetilde{\wh V}\mid_{\mU_i} $ be a local isomorphism and $\Phi_i$ be the local lifting of the absolute Frobenius on $\mU_i$. The following diagram (not commutative in general)
\begin{equation}
\xymatrix{  \Phi_i^*(\widetilde{\wt V}\mid_{\mU_{ij}}) \ar[r]^{G_{ij}(\widetilde{\wt \nabla})} \ar[d]_{\Phi_i^*(\widetilde{\gamma}_i)} 
&  \Phi_j^*(\widetilde{\wt V}\mid_{\mU_{ij}}) \ar[d]^{\Phi_j^*(\widetilde{\gamma}_j)}\\
 \Phi_i^*(\widetilde{\wh V}\mid_{\mU_{ij}}) \ar[r]^{G_{ij}(\widetilde{\wh \nabla})}  &  \Phi_j^*(\widetilde{\wh V}\mid_{\mU_{ij}}) \\
 }
\end{equation} 
Recall the de Rham bundle $(\wt V,\wt \nabla)$ (resp. $(\wh V,\wh \nabla)$) is defined by gluing $\{\Phi_i^*(\widetilde{\wt V}\mid_{\mU_i})\}$ (resp. $\{\Phi_i^*(\widetilde{\wh V}\mid_{\mU_i})\}$) via isomorphisms $G_{ij}(\widetilde{\wt \nabla})$ (resp. $G_{ij}(\widetilde{\wh \nabla})$). The $G_{ij}$'s are given by Taylor formula
\begin{equation}
 G_{ij}(\widetilde{\wt\nabla})(e\otimes 1)=\sum_J\frac{\widetilde{\wt\nabla}(\partial)^J}{J!}(e)\otimes z^J, \quad e\in \widetilde{\wt V}(\mU_i).
\end{equation}
 Denote
\[ g_{ij}=G_{ij}(\widetilde{\wt \nabla})^{-1}\circ \Phi_j^*(\widetilde{\gamma}_j)^{-1}\circ G_{ij}(\widetilde{\wh \nabla}) \circ \Phi_i^*(\widetilde{\gamma}_i)-\mathrm{id}_{\Phi_i^*(\widetilde{\wt V}\mid_{U_i\cap U_j})},\]
and 
\[\omega_{\nabla,i} = \Phi_i^*(\widetilde{\gamma}_i)^{-1}\circ \wt\nabla_i\circ \Phi_i^*(\widetilde{\gamma}_i)-\wh\nabla_i\]
which are trivial modulo $p^{n-1}$. Denote $\overline g_{ij}=\frac{g_{ij}}{p^{n-1}}\pmod{p}\in \mEnd(V)(\mU_{ij})$ and $\overline\omega_{\nabla,i}=\frac{\omega_{\nabla,i}}{p^{n-1}}\in \mEnd(V(\mU_i))\otimes\Omega_{\mX/W}(\log \mD)(\mU_i)$ . Then the de Rham bundle $(\wh V,\wh \nabla)$ differs from $(\wt V,\wt \nabla)$ by the element $(\overline g_{ij},\omega_{\nabla,i})\in H^1_{dR}(X_1,\mEnd(V,\nabla))$.

Now let's express $\overline g_{ij}$ and $\overline \omega_{\nabla,i}$ with $\overline f_{ij}$ and $\overline{\omega}_i$. 
Since 
\[\widetilde{\gamma}_j: (\widetilde{\wt V}\mid_{\mU_j},\widetilde\gamma_j^{-1}\circ\widetilde{\wh\nabla}\circ\widetilde\gamma_j) \rightarrow (\widetilde{\wh V}\mid_{\mU_j}, \widetilde{\wh \nabla})\] 
is an isomorphism of twisted de Rham bundle, one has following commutative diagram
\begin{equation}
\xymatrix@C=3cm{
  \Phi_i^*(\widetilde{\wt V}\mid_{\mU_{ij}}) \ar[r]^{G_{ij}(\widetilde\gamma_j^{-1}\circ\widetilde{\wh\nabla}\circ\widetilde\gamma_j)} \ar[d]_{\Phi_i^*(\widetilde{\gamma}_j)} 
&  \Phi_j^*(\widetilde{\wt V}\mid_{\mU_{ij}}) \ar[d]^{\Phi_j^*(\widetilde{\gamma}_j)}\\
 \Phi_i^*(\widetilde{\wh V}\mid_{\mU_{ij}}) \ar[r]^{G_{ij}(\widetilde{\wh \nabla})}  &  \Phi_j^*(\widetilde{\wh V}\mid_{\mU_{ij}})\\
}
\end{equation}
Hence 
\begin{equation}
\begin{split}
 g_{ij} & =G_{ij}(\widetilde{\wt \nabla})^{-1}
\circ G_{ij}(\widetilde\gamma_j^{-1}\circ\widetilde{\wh\nabla}\circ\widetilde\gamma_j)
\circ \Phi_i^*(\widetilde{\gamma}_j)^{-1} 
\circ \Phi_i^*(\widetilde{\gamma}_i)-\mathrm{id}\\
& =G_{ij}(\widetilde{\wt \nabla})^{-1}
\circ G_{ij}(\widetilde{\wt \nabla}+p^{n-1}\overline\omega_j)
\circ \Phi_i^*(\widetilde{\gamma}_j^{-1} \circ\widetilde{\gamma}_i)
-\mathrm{id}\\
& =\left(\mathrm{id}+p^{n-1}\sum_{\ell}\overline{\omega}_j(\partial_\ell) \otimes_{\Phi_j} z_\ell\right)
\circ \Phi_i^*(\mathrm{id}+p^{n-1}\overline f_{ij})
-\mathrm{id}\\
&=p^{n-1}\left(\sum_{\ell}\overline{\omega}_j(\partial_\ell) \otimes_{\Phi_j} z_\ell + \Phi^*(\overline f_{ij}) \right)\\
&=p^{n-1}\left( h_{ij}(\overline\omega_j)+ \Phi^*(\overline f_{ij}) \right)\\
\end{split}
\end{equation}
and $\overline{g}_{ij}=\sum\limits_{\ell}\overline{\omega}(\partial_\ell) \otimes_{\Phi} z_\ell + \Phi(\overline f_{ij})$. Since $\wt\nabla_i=\frac{\Phi_i^*}{p}\left(\widetilde{\wt\nabla}\right)$ and $\wh\nabla_i=\frac{\Phi_i^*}{p}\left(\widetilde{\wh\nabla}\right)$ 
\begin{equation*}
\begin{split}
\omega_{\nabla,i} & = \frac{\Phi_i^*}{p}\left( \widetilde{\gamma}_i^{-1} \circ \widetilde{\wh \nabla}_i \circ \widetilde{\gamma}_i -\widetilde{\wt \nabla}_i \right)
\end{split}
\end{equation*}
Thus $\overline{\omega}_{\nabla,i}=\frac{\omega_{\nabla,i}}{p^{n-1}}=\frac{\Phi^*}{p}\left(\overline\omega_i\right)$.
\end{proof}

\begin{cor}\label{InvCar_Torsor_map}
The map from $\mathrm{Gr}^0H^1_{Hig}(X_1 , \mEnd (E ,\theta))$ to $H^1_{dR}(X_1 , \mEnd (V ,\nabla))$ mapping 
$(\overline f_{ij} ,  \overline\omega_i)$ to $\left(\Phi^*(\overline f_{ij})+h_{ij}(\overline\omega_j),\frac{\Phi^*}{p}(\overline\omega_i)\right)$ is a Frobenius semilinear map between $k$-vector spaces.
\end{cor}
\begin{proof}This follows that $\Phi^*$, $h$ and $\frac{\Phi^*}{p}$ are all Frobenius semilinear.
\end{proof}

\textbf{Acknowledgement}.  We would like to thank Adrian Langer particularly for pointing out an incomplete proof of the local freeness Higgs sheaves in Section 3.3. Fortunately, the proof can be now fixed by applying for Theorem 2.1 and Corollary 2.8 in his very recent paper~\cite{Langer19}. We thank Mao Sheng and Carlos Simpson for their interests in this paper. We thank warmly Christian Pauly for pointing out a paper of de Jong and Osserman. We are grateful to Xiaotao Sun, Deqi Zhang, Shouwu Zhang and Junyi Xie for the discussion on dynamic properties of self-maps in characteristic $p$. We also thank Duco van Straten for telling us the paper of Kontsevich about the $\ell$-adic representations. The discussion with Ariyan Javanpeykar about the self-map on the moduli space of Higgs bundles over the projective line with logarithmic structure on marked points is very helpful for us, and he also read the preliminary version of this paper carefully and suggested a lot for the improvement. We thank him heavily.

 \end{document}